\documentclass[11pt,a4paper]{article}

\usepackage[utf8]{inputenc}
\usepackage{amsmath}
\usepackage{amsfonts}
\usepackage{amssymb}
\usepackage{makeidx}
\usepackage{graphicx}
\graphicspath{ {./} }
\usepackage{listings}
\usepackage{tocloft}
\usepackage{verbatim}
\usepackage{xcolor}

\usepackage{calc}
\usepackage{CJKutf8}
\usepackage[hidelinks]{hyperref}

\newtheorem{theorem}{Theorem}[section]
\newtheorem{le[mm]a}[theorem]{Le[mm]a}
\newtheorem{proposition}[theorem]{Proposition}

\newenvironment{proof}[1][Proof]{\begin{trivlist}
\item[\hskip \labelsep {\bfseries #1}]}{\end{trivlist}}
\newenvironment{definition}[1][Definition]{\begin{trivlist}
\item[\hskip \labelsep {\bfseries #1}]}{\end{trivlist}}

\newenvironment{remark}[1][Remark]{\begin{trivlist}
\item[\hskip \labelsep {\bfseries #1}]}{\end{trivlist}}
\newenvironment{remarks}[1][Remarks]{\begin{trivlist}
\item[\hskip \labelsep {\bfseries #1}]}{\end{trivlist}}
\newenvironment{exercise}[1][Exercise]{\begin{trivlist}
\item[\hskip \labelsep {\bfseries #1}]}{\end{trivlist}}

\newcommand{\qed}{\nobreak \ifvmode \relax \else
      \ifdim\lastskip<1.5em \hskip-\lastskip
      \hskip1.5em plus0em minus0.5em \fi \nobreak
      \vrule height0.75em width0.5em depth0.25em\fi}
      
  \newcounter{exer}[section]
  \newenvironment{exer}[1][]{ \refstepcounter{exer} \par\medskip
 \noindent \textbf{Exercise \arabic{exer}:} \rmfamily \slshape}{\medskip}

    \makeindex

\title{Geometry on surfaces, a source for mathematical developments}
\author{Norbert A'Campo and Athanase Papadopoulos}
 
\begin{document}

\maketitle
\begin{abstract}

We present a variety of geometrical and combinatorial tools that are used in the study of geometric structures on surfaces: volume, contact, symplectic, complex and almost complex structures. We start with a series of local rigidity results for such structures. Higher-dimensional analogues are also discussed. Some constructions with Riemann surfaces lead, by analogy, to notions that hold for arbitrary fields, and not only the field of complex numbers. The Riemann sphere is also defined using  surjective homomorphisms of real algebras from the ring of real univariate polynomials to (arbitrary) fields, in which the field with one element is interpreted as the point at infinity of the Gaussian plane of complex numbers.  Several models of the hyperbolic plane and hyperbolic 3-space appear, defined in terms of complex structures on surfaces, and in particular also a rather elementary construction of the hyperbolic plane using
real monic univariate polynomials of degree two without real roots.  Several notions and problems connected with conformal structures in dimension 2 are discussed, including dessins d'enfants, the combinatorial characterization of polynomials and rational maps of the sphere, the type problem, uniformization, quasiconformal mappings, Thurston's characterization of Speiser graphs, stratifications of spaces of monic polynomials,  and others. Classical methods and new techniques complement each other. 

The final version of this paper will appear as a chapter in the Volume \emph{Surveys in Geometry. II} (ed. A. Papadopoulos), Springer Nature Switzerland, 2024.

\bigskip

\noindent {\bf Keywords:} geometric structure, conformal structure, almost complex structure ($J$-field), Riemann sphere, uniformization, the type problem,  rigidity, model for hyperbolic space, cross ratio, Belyi's theorem,  Riemann--Hurwitz formula, Chasles 3-point function, branched covering, type problem, dessin d'enfants, slalom polynomial, slalom curve, space of monic polynomials, stratification, fibered link, divide,  Speiser curve, Speiser graph, line complex, quasiconformal map, almost analytic function.

\bigskip

\noindent {\bf AMS classification:} 12D10, 26C10, 14H55, 30F10, 30F20, 30F30, 53A35, 53D30, 57K10, 

\end{abstract}

\vfill\eject

\tableofcontents

\section{Introduction}

Given a differentiable surface, i.e., a 2-dimensional differentiable manifold,  one can enrich it with various kinds of  geometric structures.  Our first aim in the present survey is to give an introduction to the study of surfaces equipped with locally rigid and homogeneous geometric structures. 

Formally, a geometric structure on a surface $S$ is given by a section of some bundle associated with its tangent bundle $TS$. We shall deal with specific examples, mostly, volume forms, almost complex structures (equivalently, conformal structures, since we are dealing with surfaces) and  Riemannian metrics of constant Gaussian curvature. We shall also consider quasiconformal structures on surfaces. Foliations with singularities, Morse functions,  meromorphic functions and differentials on almost complex surfaces induce geometric structures that are 
locally rigid and homogeneous only in the complement of a discrete set of points on the surface. Laminations, measured foliations and quadratic differentials are examples of less homogeneous geometric structures. They play important roles in the theory of surfaces, as explained by Thurston, but we shall not consider them here. 

A theorem of Riemann gives a complete classification of non-empty simply connected open subsets of 
$\mathbb{R}^2$ that are equipped with almost complex structures. Only two classes remain!  This takes care at the same time of the topological classification of such surfaces without extra geometric structure: they are all homeomorphic. The classical proof of this topological fact invokes the Riemann Mapping Theorem,\index{Riemann Mapping Theorem} that is, it assumes the existence of an almost complex structure on the surface.
 Likewise, only two classes remain in the classification of non-empty open connected and simply connected subsets of 
$\mathbb{R}^2$ that are equipped with a Riemannian metric of constant curvature: the Euclidean and the Bolyai--Lobachevsky plane. The latter is also called the non-Euclidean or hyperbolic plane.

No classification theorem similar to that of simply connected open subsets of  $\mathbb{R}^2$ holds in 
$\mathbb{R}^3$, even if one restricts to contractible subsets. See \cite{W1935} for the historical example, now called ``Whitehead manifold",\index{Whitehead manifold} which, by a result of Dave Gabai (2011)  \cite{Ga2011}, is a manifold of small category,\index{manifold of small category} i.e.,  it is covered by two charts, both of which being copies of $\mathbb{R}^3$ that moreover intersect along a third copy of $\mathbb{R}^3$.

We shall be particularly concerned with almost complex structures,\index{almost complex structure} i.e., conformal structures, on surfaces. The theory of such structures is intertwined with topology. This is not surprising: Riemann's first works on functions of one complex variable gave rise at the same time to fundamental notions of topology. He conceived the notion of ``$n$-extended multiplicity" (\emph{Mannigfaltigkeit}), an early version of $n$-manifold, he introduced basic notions like 
connectedness and
degree of connectivity for surfaces, which led him to the discovery of  Betti numbers in the general setting (see And\'e Weil's article  \cite{Weil-Tardy} on the history of the topic), he classified closed surfaces according to their genus, he introduced branched coverings, and he was the first to notice the topological properties of functions of one complex variable (one may think of the construction of a Riemann surface associated with a multi-valued meromorphic function). At about the same time, Cauchy, in his work on the theory of functions of one complex variable, introduced path integrals and the notion of homotopy of paths. We shall see below many such instances of topology meeting complex geometry.

In several passages of the present survey, we shall encounter graphs that are used in the study of Riemann surfaces. They will appear  in the form of:
\begin{enumerate}
\item  Speiser graphs\index{Speiser graph} associated with branched coverings of the sphere: these are used in Thurston's realization theorem for branched coverings\index{branched covering!postcritically finite!Thurston theorem} (\S \ref{s:Thurston}), in the type problem (\S \ref{s:Nevanlinna}), in the theory of  dessins d'enfants (\S \ref{s:dessins}) and in a cell-decomposition of the space of rational maps (\S \ref{ss:Rational});
\item
 rooted colored trees associated with slalom polynomials (\S \ref{s:slalom});
 \item
 pictures of monic polynomials used for the stratification of the space of slalom polynomials (\S \ref{s:stratification}).
  \end{enumerate}
  
 At several places, we shall see how familiar constructions using the field of complex numbers can be generalized to other fields. Conversely, algebraic considerations will lead to several models of the Riemann sphere and of 2- and 3-dimensional hyperbolic spaces. Relations with the theory of knots and links will also appear.

Let us give now a more detailed outline of the next sections:

In \S\ref{s:Moser} we present a few classical examples of rigidity\index{Moser rigidity theorem!geometric structures} and  local rigidity results in the setting of geometric structures on $n$-dimensional manifolds. A theorem due to J\"urgen Moser, whose proof is sometimes called ``Moser's Trick", deals with the classification up to isotopy of volume forms on compact connected oriented $n$-dimensional manifolds. We show how this proof can be adapted to the symplectic and contact settings.
A local rigidity result (which we call a Darboux local rigidity theorem) gives a canonical form for volume, symplectic and contact forms on
non-empty connected and simply connected open subsets $(S,\omega)$ of 
$\mathbb{R}^n$.

In dimension two, almost complex structures are also locally rigid, and we present a Darboux-like theorem for them. \index{Darboux rigidity!almost complex structure} The question of the existence and integrability of $J$-structures\index{J@$J$-structure} on higher-dimensional spheres arises naturally.  We survey a result due to Adrian Kirchhoff which says that an $n$-dimensional sphere admits a $J$-structure if and only if the $(n+1)$-dimensional sphere admits a parallelism, that is, a global field of frames.  This deals with the question of the existence of $J$-fields on higher-dimensional spheres, which we also discuss in the same section: only $S^6$ carries such a structure. 

  Section \ref{s:First} is concerned with the first example of Riemann surface, namely, the Riemann sphere. We give several models of this surface. Its realization as the projective space $\mathbb{P}^1(\mathbb{C})$ leads to constructions that are valid for any field $k$ and not only for $\mathbb{C}$. In the same section, we review a realization of $\mathbb{P}^1(\mathbb{C})$ with its round metric as a quotient space  of 
the group  ${\rm SU}(2)$   of linear transformations of $\mathbb{C}^2$ of determinant 1 preserving the standard
 Hermitian product. The intermediate quotient ${\rm SU}(2)/\{\pm {\rm Id\}}$ is isometric to 
$P^3(\mathbb{R})$ and also to the space $T_{l=1}\mathbb{P}^1(\mathbb{C})$ of length $1$ tangent vectors to $S^2$.  
In this description, oriented M\"obius circles on $\mathbb{P}^1(\mathbb{C})$ (that is, the circles of the conformal geometry of $\mathbb{P}^1(\mathbb{C})$)  lift naturally to oriented great 
circles on $S^3={\rm SU}(2)$.
 Also, closed immersed curves without self-tangencies lift to classical links (that is, links in the 3-sphere).

In \S \ref{s:Models}, we present another model of the Riemann sphere, together with models of the hyperbolic plane  and of hyperbolic 3-space. A model of the Riemann sphere is obtained using algebra, namely, fields and ring homomorphisms. In this model, the point at infinity of the complex plane is represented by $\mathbb{F}_1$, the field with one element. The notion of shadow number is introduced, as a geometrical way of viewing the cross ratio.  The hyperbolic plane appears as a space of ideals  equipped
with a geometry naturally given by a family of lines. In this way, the hyperbolic plane has a very simple description which arises from algebra. 
 The cross ratio is used to prove a necessary and sufficient condition for a generic configuration of planes in a real 4-dimensional vector space to be a configuration of complex planes. We then introduce the notions of compatible (or $J$-conformal)\index{J@$J$-conformal metric}\index{J@$J$-compatible metric} Riemannian metric and we prove the existence and uniqueness of such 
  metrics on homogeneous Riemann surfaces with commutative stabilizers. We describe several models of spherical geometry (surfaces of constant curvature +1), and of 2- and 3-dimensional hyperbolic spaces in terms of the complex geometry of surfaces. We then study the notion of $J$-compatible Riemannian metrics. An existence result of such metrics is the occasion to characterize homogeneous Riemann surfaces up to bi-holomorphic equivalence.
  
In \S \ref{s:Uni}, we reduce generality by assuming that the surface $S$ is an open connected and path-connected non-empty subset of the real plane $\mathbb{R}^2$.
Fundamental results appear. For instance, the theorem saying that two open connected and path-connected non-empty subsets of the real plane $\mathbb{R}^2$ are diffeomorphic, a consequence of the Riemann Mapping Theorem.\index{Riemann Mapping Theorem} This theorem says that any nonempty open subset of the complex plane which is not the entire plane is biholomorphically equivalent to the unit disc.  The Riemann Mapping Theorem\index{Riemann Mapping Theorem} generalized to any simply connected Riemann surface (and not restricted to open subsets of the plane) is the famous Uniformization Theorem.\index{uniformization theorem} It leads to the type problem, which we consider in \S \ref{s:Type}.

The next section,  \S \ref{s:Branched}, is concerned with some  aspects of branched coverings between surfaces. A classical combinatorial formula associated with such an object is the Riemann--Hurwitz formula.\index{Riemann--Hurwitz formula} It leads to some natural problems which are still unsolved. A combinatorial object associated with a branched covering of the sphere is a Jordan curve that passes through all the critical values  and which we call a \emph{Speiser curve}.\index{Speiser curve} Its lift by the covering map is a graph we call a  \emph{Speiser graph},\index{Speiser graph} an object that will be used several times   in the rest of the survey. A theorem of Thurston which we recall in this section gives a characterization of oriented graphs on the sphere that are Speiser graphs of some branched covering of the sphere by itself. Thurston proved this theorem as part of his project of understanding what he called the ``shapes" of rational functions of the Riemann sphere. In the same section, we introduce a graph on a surface which is dual to the Speiser graph, often known in the classical literature under the name \emph{line complex},   which we use in an essential way in \S \ref{s:Type}. We reserve the name line complex to another graph.

The type problem,\index{type problem}\index{problem!type} reviewed in \S \ref{s:Type}, is the problem of finding a method for deciding whether a simply connected Riemann surface, defined in some  specific manner (e.g., as a branched covering of the Riemann sphere, or as a surface equipped with some Riemannian metric, or obtained by gluing polygons, etc.) is conformally equivalent to the Riemann sphere, or to the complex plane, or to the open unit disc. We review several methods of dealing with this problem, mentioning works of Ahlfors, Nevanlinna, Teichm\"uller, Lavrentieff and Milnor. Besides the combinatorial tools introduced in the previous sections (namely, Speiser graphs and line complexes), 
 the works on the type problem that we review use the notions of almost analytic function and quasiconformal mapping.\index{quasiconformal mapping}\index{mapping!quasiconformal}

In the last section, \S \ref{s:combinatorics}, combinatorial tools are used for other approaches to Riemann surfaces, in particular, in the theory of dessins d'enfants, in applications to knots and links and  in the theory of slalom polynomials\index{slalom polynomial}. Two\index{polynomial!slalom} different stratifications of the space of monic polynomials are presented.

\section{Rigidity of geometric structures}\label{s:Moser}

In this section, we give several examples of locally rigid structures on surfaces.  A classical example of a non-locally rigid structure is a Riemannian metric on any manifold of dimension $\geq 2$.

\subsection{Volume, symplectic and contact forms}

Moser's theorem\index{Moser trick}\index{Moser rigidity theorem!volume form} says that only the total volume\index{volume form} of a smooth volume form on a connected compact manifold matters, namely, two volume forms of equal total volume are isotopic. More precisely:\index{Moser trick}\index{Moser rigidity theorem!volume form} 

\begin{theorem}[J. Moser \cite{Mos65}]
Let $M$ be a compact connected oriented manifold of dimension $n$ equipped with two smooth volume forms $\omega_0$ and $\omega_1$ of equal total volume. Then there exists an isotopy $\phi_t,\, t\in [0,1],$ satisfying $\phi_t^*(t\omega_1+(1-t)\omega_0)=\omega_0$. In particular, we have $\omega_0=\phi_1^{*}\omega_1$.
\end{theorem} 

\begin{proof} Clearly $\omega_1=f\omega_0$ for some positive function $f$, since for any $p\in M$ and for any oriented frame $X_1,\cdots ,X_n$ at $p$ we have $\omega_0(X_1,\cdots ,X_n)>0$ and $\omega_1(X_1,\cdots ,X_n)>0$. It follows that $t\mapsto \omega_t=t\omega_1+(1-t)\omega_0$ is a path of volume forms that connects the form $\omega_0$ to the form $\omega_1$ and we have 
\begin{align*}
\frac{d}{dt}\int_{[M]}\omega_t &=\int_{[M]}\frac{d}{dt}\omega_t \\ 
&=\int_{[M]}\omega_1-\omega_0 \hbox{ (differentiating the formula for $t\mapsto w_t$)} \\
&=0 \hbox{ (since the two forms have the same volume)}.
\end{align*}
Thus, the de Rham cohomology class $[\omega_1-\omega_0]$ vanishes on the connected manifold $M$, therefore there exists a smooth $(n-1)$-form $\alpha$ with $d\alpha=\omega_1-\omega_0$. Hence, $\frac{d}{dt}\omega_t=d\alpha$.

In order to construct the required isotopy 
$\phi_t$ satisfying $(\phi_t )^*\omega_t=\omega_0$, we need a time-dependent vector field $X_t$ whose flow $\phi^X_t$ induces the isotopy $\phi_t$ and such that the equality 
$(\phi^X_t)^*\omega_t=\omega_0$ holds. Differentiating, using the Cartan formula and the fact that $d\omega_t=0$, yields 
$$0=\frac{d}{dt}(\phi^X_t)^*\omega_t=(\phi^X_t)^*(d(i_{X_t}\omega_t)+d\alpha)=(\phi^X_t)^*(d(i_{X_t}\omega_t+\alpha)).$$
The family of vector fields $X=(X_t)_{t\in [0,1]}$ defined by
$i_{X_t}\omega_t=-\alpha$ satisfies the above equation. The equation $i_{X_t}\omega_t=-\alpha$ has, for a given $(n-1)$-form $\alpha$, has a unique solution, since for each $t\in [0,1]$, $\omega_t$ is a non-degenerate volume form. Therefore the family of forms $ (\phi^X_t)^*\omega_t$ is constant, hence $ (\phi^X_1)^*\omega_1=\omega_0$ as required.\qed
\end{proof}

The above result also holds for a symplectic form,\index{symplectic form} that is, a closed nondegenerate differential 2-form, at the price of a stronger assumption. The proof works verbatim. Thus we get:\index{Moser rigidity theorem!volume form}

\begin{theorem}[J. Moser] Let $M$ be a compact connected oriented manifold of dimension $n$ equipped with two symplectic\index{symplectic form} forms $\omega_0$ and $\omega_1$ of equal periods, i.e., with equal de Rham cohomology classes. Assume that the forms are connected by a smooth path $\omega_t$ of symplectic forms with constant periods, i.e., for all $t\in[0,1]$, $[\omega_t]=[\omega_0]$ in $H^2_{\rm dR}(M)$. Then there exists an isotopy $\phi_t,\, t\in [0,1],$ with $\phi_t^*\omega_t=\omega_0$. In particular, $\omega_0=\phi_1^{*}\omega_1$.\qed
\end{theorem}

The so-called ``Moser trick"\index{Moser trick} works as a ``simplification by d" in the equation $di_{X_t}\omega_t=-d\alpha$ and it amounts to noticing that for a volume form $\omega$ and for an $(n-1)$-form $\beta$ the equation $i_X\omega=\beta$ has a unique solution $X$.

From symplectic structures,\index{contact form} we pass to contact forms and contact structures. 

A \emph{contact form} $\alpha$ on an $n$-dimensional manifold $M$ is a pointwise non-vanishing differential $1$-form such that at each point $p$ in $M$ the restriction of $(d\alpha)_p$ to the kernel of $\alpha_p$ is non-degenerate.  

A \emph{contact structure}\index{contact structure} on $M$ 
is a distribution of hyperplanes in the tangent bundle $TM$ given locally as a field of kernels of a contact form.

 Moser's method also works, even more simply, without ``trick"\index{Moser trick}  and without extra stronger assumption, for families of contact  structures and it gives another proof of the Gray stability theorem\index{Gray stability theorem} for contact forms \cite{Gray1959}:

\begin{theorem}[J. W. Gray \cite{Gray1959}] Let $(\alpha)_{t\in [0,1]}$ be a smooth family of contact forms on a compact manifold $M$.
Then there exists a $t$-dependent vector field $X_t$ on $M$ with flow $\phi_t$ 
and ${\rm kernel}(\phi_t^* \alpha_t)={\rm kernel}(\alpha_0)$. In particular, there exists a family of positive functions $(f_t)_{t\in [0,1]}$ such that for all $t\in [0,1]$,
$\phi_t^* \alpha_t=f_t\alpha_0$.
\end{theorem}

\begin{proof} A measure for the variation of the kernel $[\alpha_t]$ of $\alpha_t$ is the restriction 
$\dot{\alpha_t}_{|[\alpha_t]}$ of 
$\dot{\alpha_t}=\frac{d}{dt} \alpha_t$ to the kernel $[\alpha_t]$. By the non-degeneration of the restriction  
$d\alpha_{|[\alpha_t]}$, there exists a unique $t$-dependent vector field $X_t$ in the distribution (that is, the family of subspaces) $[\alpha_t]$ with 
$\dot{\alpha_t}_{|[\alpha_t]}+i_{X_t}{d\alpha_t}_{|[\alpha_t]}=0$. Hence the kernels of $\phi_t^*\alpha_t$ do not vary since
$\frac{d}{dt}\phi_t^*[\alpha_t]=\phi_t^*(\dot{\alpha_t}_{|[\alpha_t]}+i_{X_t}{d\alpha_t}_{|[\alpha_t]})=0$.\qed

\end{proof}

For more applications, see \cite{Mc-S1998}. The use of a proper exhaustion allows us to extend the above theorem  to pairs of  volume\index{volume form} forms on connected non-compact manifolds of equal finite or infinite total volume.

Furthermore, the above proofs work also in a relative version: if the forms coincide on a closed subset $A$, then
the time-dependent vector field
$X_t$ vanishes along the subset $A$ and generates a flow that fixes the subset $A$. The Darboux type rigidity theorems for volume,\index{Darboux rigidity!volume form} \index{Darboux rigidity!symplectic form}\index{Darboux rigidity!contact form} symplectic and contact forms follow:

\begin{theorem}[Local Darboux rigidities]
Let $\omega$ be a volume or a symplectic form, and let $\alpha$ be a contact form on an $n$-, $2n$- or $(2n+1)$-manifold $M$ respectively.
Then at each point of $M$ there exists a coordinate chart $(x_1, \cdots  ,x_n)$ or $(x_1, \cdots ,x_n,y_1, \cdots ,y_n)$ or $(x_1, \cdots ,x_n,y_1, \cdots ,y_n,z)$ respectively such that the volume form is expressed by
$\omega=dx_1\wedge \cdots \wedge dx_n$, the symplectic form by
$\omega=dx_1\wedge dy_1+ \cdots +dx_{n}\wedge dy_{n}$ and the contact form by $\alpha=dz-y_1dx_1- \cdots - y_ndx_n$.\qed
\end{theorem}

\begin{remark}

 The classical Darboux theorem\index{Darboux theorem} holds in the setting of symplectic geometry, see \cite{Darboux1882}. This theorem says that any  symplectic manifold of dimension $2n$ is locally isomorphic (in this setting, it is said to be symplectomorphic) to the linear symplectic space $\mathbb{C}^n$ equipped with its canonical symplectic form $\sum dx\wedge dy$. 
   As a consequence, any two symplectic manifolds of the same dimension are locally symplectomorphic to each other.

\end{remark}

\subsection{Almost complex structures} 

An\index{almost complex structure} almost complex structure $J$ on a differentiable surface $S$ is an endomorphism of the tangent bundle of $S$ satisfying $J^2=-\mathrm{Id}$. More precisely, $J=\{J_p\mid p\in S\}$  is a
smooth family of endomorphisms of tangent spaces $J_p: T_pS\to T_pS$ such that at each point $p\in S$, we have $J_p^2=-\mathrm{Id}_{T_{p}S}$. The standard example is $(\mathbb{R}^2, J)$ where $J$ is the constant family of endomorphisms given by the matrix 
$(\begin{smallmatrix} 0 & -1\\ 1 & 0\\  \end{smallmatrix})$. This corresponds to the plane $\mathbb{C}$ equipped with multiplication by $i$.

The following proof is not based upon the above method.\index{almost complex structure}

\begin{theorem} \label{th:J} {\bf Local $J$-rigidity in real dimension 2.} Let $J$ be an almost complex structure on a surface $S$. Then at each point $p\in S$ there exists a coordinate chart $(x,y)$ such that
$J(\frac{\partial}{\partial x})=\frac{\partial}{\partial y}$ holds.
\end{theorem}

\begin{proof} {\bf (Sketch)} First construct, using a partition of unity, an almost complex structure $J_0$ on the torus $T=\mathbb{R}^2/\mathbb{Z}^2$ such that the structures $J$ and $J_0$ are isomorphic when restricted to open neighborhoods $U$ of $p$ on $S$ and $U_0$ of $0$ on $T$. Let $\omega$ be a volume form on $T$ and let $g_{\omega,J_0}$ be the associated Riemannian metric $g_{\omega,J_0}(u,v)=\omega(u,J_0(v))$. Let $f$ be the real function on $T$ satisfying $f(0)=0$ and solving the partial differential equation 
$$d(df\circ J_0)=-k_{g_{\omega,J_0}}\omega,$$
where $k_{g_{\omega,J_0}}$ is the Gaussian curvature of the metric $g_{\omega,J_0}$.
By the Gauss--Bonnet Theorem, $\int_T k_{g_{\omega,J_0}}\omega=0$, therefore the equation admits a solution by  Fourier theory.
Now use the Gauss curvature formula:
$$k_{g_{e^{2f}\omega,J_0}}\omega=k_{g_{\omega,J_0}}\omega+d(df\circ J_0)=0.$$
The metric $g_{e^{2f}\omega,J_0}$ has constant curvature $0$, therefore $(T,J_0)$ is bi-holomorphic to $\mathbb{C}/\Gamma$ for some lattice $\Gamma$ (a 2-generator discrete subgroup), which shows the statement for a local chart at $0\in (T,J_0)$, and hence also for a local chart at any $p\in (S,J)$.\qed

\end{proof}

For a detailed proof of Theorem \ref{th:J}, see \cite[p. 114-117]{A2021}. This theorem shows that every almost complex structure on a differentiable surface $S$ determines in a unique way a holomorphic structure in the usual sense (that is, a structure defined by an atlas of local charts with values in $\mathbb{C}$ and holomorphic local changes).

\begin{exer}
Give a proof of Theorem \ref{th:J} using Moser's trick.\index{Moser trick} 
\end{exer}

{\bf Historical note}
The first definition of an almost complex structure is due to Charles Ehresmann\index{Ehresmann, Charles} who addressed the question of the existence of a complex analytic structure on a topological (resp. differentiable)  manifold of even dimension,
from the point of view of the theory of fiber spaces; cf. Ehresmann's talk at the 1950 ICM \cite{Ehresmann-ICM}. Ehresmann mentions the fact that H. Hopf addressed the same question from a different point of view. He notes in the same paper that by a method proper to even-dimensional spheres he showed that the 4-dimensional sphere does not admit any almost complex structure, a result which was also obtained by Hopf using different methods. 
See also McLane's review of Ehresmann's work \cite{MacLane}. 
Ehresmann and MacLane also refer to the work of Wen-Ts\"un Wu \cite{Wu1, Wu2}, who was a student of Ehresmann in Strasbourg.

\bigskip

\begin{remark}
The Nijenhuis tensor is an obstruction to local integrability of $J$-fields in higher dimensions, where Theorem \ref{th:J} does not hold in the general case, see \cite[p. 124-125]{A2021}. Real dimension $2$ is very special!
\end{remark}

\subsection{Almost complex structures on $n$-spheres}

The existence and integrability\index{J@$J$-structure!high dimensions}\index{J@$J$-structures on spheres}  of $J$-structures in dimension 2 is very special. We mentioned that the $4$-sphere $S^4$ does not admit any $J$-field (Ehresmann and Hopf), but the $6$-sphere does.

Clearly only spheres of even dimension can carry $J$-fields.\index{J@$J$-field} Adrian 
Kirchhoff, in his PhD thesis (ETH Z\"urich 1947) \cite{Kirchhoff1} established a relationship between two non-obviously related structures on spheres $S^{2n}$ and $S^{2n+1}$ of different dimensions; we report on this now.

Recall that a \emph{parallelism}\index{parallelism} on a smooth $n$-manifold is a global field of frames, that is, a field of $n$ tangent vectors which form a basis of the tangent space at each point.

%
%
%
%
%
%

%

%
%

\begin{theorem}[Kirchhoff  \cite{K}]   The sphere $S^n, \,n\geq 0,$ admits a $J$-field if and 
only if the sphere $S^{n+1}$ admits a parallelism.
\end{theorem}

\begin{proof}
The case $n=0$ is special: the tangent space $TS^0$ is of dimension $0$, therefore  
$J={\rm Id}_{TS^0}$ is a $J$-field and $S^1$ admits a parallelism.

 ``Only if" part for $n>0$: In $V=\mathbb{R}^{n+2}$ with the standard basis $e_0,e_1, \cdots ,e_{n+1}$, let $S^n$ be the unit sphere in the span
$[e_1, \cdots ,e_{n+1}]$. Let $S^{n+1}$ be the unit sphere of $V$. Assume that $J$ is a $J$-field on $S^n$. Let $L:S^{n+1} \to {\rm GL}(V), v\mapsto L_v,$ be the continuous map satisfying $L_v(e_0)=v,\, v\in S^{n+1},$ defined as follows:

\begin{itemize}
\item First, for $v\in S^n$, seen as the equator of $S^{n+1}$, we set \newline $L_v(v)=-e_0, \, L_v(e_0)=v$,\newline $L_v(u)=v+J_v(u-v), \, u\in 
[v,e_0]^\perp$.

\item
 For $v\in S^{n+1}$, we can write $v=\sin(t)e_0+\cos(t)v',\,v'\in S^n,\, t\in ]-\pi,\pi[.$ We then set \newline $L_v=\sin(t){\rm Id}_V+\cos(t)L_{v'}$.
 \end{itemize}

We have $L_v\circ L_v=-{\rm Id}_V$ for $v\in S^n$, hence $L_v=\sin(t){\rm Id}_V+\cos(t)L_{v'}\in {\rm GL}(V)$ for $v\in S^{n+1}$, since the eigenvalues are $\sin(t)\pm \cos(t)i$.
Observe that $T_{e_0}S^{n+1}=e_0+[e_0]^\perp$ and $[e_0]^\perp=[e_1, \cdots ,e_{n+1}]$. The differential $(DL_v)_{e_0}:T_{e_0}S^{n+1}\to V$ at $e_0$ of $L_v$ maps the  space $T_{e_0}S^{n+1}$ onto an affine space of dimension $n+1$ in $T_{L_v(e_0)}V$ that intersects transversely the ray $[v]$. Then, for $v\in S^{n+1}$, the images $(DL_v)_{e_0}(e_1), \cdots ,(DL_v)_{e_0}(e_{n+1})\in T_vV$ define a frame in $T_vS^{n+1}=v+[v]^\perp$ by the projection parallel to $[v]$ onto $v+[v]^\perp$.

``If" part:  Work backwards.\qed

\end{proof}

Ehresmann in his ICM talk  \cite{Ehresmann-ICM}  mentions Kirchhoff's results \cite{Kirchhoff1}.

In fact, by a celebrated result of Jeffrey Frank Adams \cite{A}, only the spheres $S^1, S^3, S^7$ admit a parallelism. This implies that $S^4$ does not admit any $J$-field\index{J@$J$-field} and $S^6$ does. Adams' result was obtained several years after Kirchhoff's result.

The question of the existence of a complex structure on $S^6$ is still wide open. How the Nijenhuis integrability condition for a $J$-field\index{J@$J$-field!integrability} on $S^6$ translates into a property of framings on $S^7$ is the subject of a recent paper \cite{LR}.

 We end this section on rigidity by a word on exotic spheres: Any two differentiable manifolds of the same dimension are locally diffeomorphic. But such manifolds may be homeomorphic without being diffeomorphic. The first examples of such a phenomenon are Milnor's exotic 7-spheres \cite{Milnor1956}. In later papers, Milnor constructed additional examples.

\section{The first compact Riemann surface}\label{s:First}
A Riemann surface is a complex 1-dimensional real manifold, or a 2-dimensional manifold equipped with a complex 1-dimensional structure, that is, an atlas whose charts take values in the Gaussian plane $\mathbb{C}$, with holomorphic transition functions. 

In this section, we shall deal with the simplest Riemann surface, the Riemann sphere.

\subsection{The Riemann sphere}\label{s:Riemann-shadow}
The familiar round sphere in 3-space, together with its group of rigid motions, can be seen as a holomorphic object:  its motions are angle-preserving. It is also a one-point compactification of the field of complex numbers. We shall see that this construction as a one-point compactification can be generalized to an arbitrary field. 

First we ask the question:

Why do we need the Riemann sphere? 

The statement: ``Every sequence of complex numbers has a convergent subsequence" is very true, indeed true for bounded sequences. The statement is salvaged without this assumption if we introduce a wish object $w$ with the property that every sequence of complex numbers, for which no subsequence converges
to a complex number,
converges to $w$. In this way, from the familiar Gaussian plane $\mathbb{C}$, we gain a new space,
$\mathbb{C}\cup \{w\}$ in which the above statement improves from very true to true.
This topological construction is the familiar one-point compactification of non compact but locally compact spaces.

Very true is also the statement: ``The ratio $\frac{a}{b}$ is well defined as long as
$(a,b)\not=(0,0)$". Again the statement becomes true if we introduce by wish a new object $w\notin k$ with $\frac{a}{0}=w, a\not=0.$
This algebraic construction applies to any field $k$, and not only $\mathbb{C}$.

In both constructions the object $w$ appears as a newcomer, an immigrant, with a special restricted status.

It is Riemann who gave an interpretation of the new set $X=\mathbb{C}\cup \{w\}$ together with a very rich structure $\Sigma$ on it, for which the new element gains unrestricted status. In short, the automorphism group of $(X,\Sigma)$ acts transitively on this space, that is, the space $X$ is homogeneous.

The above topological construction also shows that the newcomer $w$ is above any bound, so from now on we use the symbol $\infty$ for $w$.

Here is another construction of an infinity, valid for any field.

Let $k$ be a field. An element $\lambda\in k$ can be interpreted as a linear map
$a\in k \mapsto \lambda a\in k$. Its \emph{graph} $G_\lambda\subset k\times k$ is the vector subspace $\{(a,\lambda a) \mid a\in k \}$ of dimension $1$ in $k\times k$. So we get an embedding $\iota:\lambda \in k \mapsto \mathbb{P}^1(k)$ of the field $k$ in the projective space $\mathbb{P}^1(k)$ of all $1$-dimensional vector subspaces in $k\times k$. The vector subspace $G=\{(0,b)\mid b\in k\}$ is the only one which is not in the image of the embedding $\iota$. 

The element $\lambda$ can be retrieved from $G_\lambda$ as a slope: indeed, for any $(a,b)\in G_\lambda$, if $(a,b)\not= (0,0)$ then $a\not= 0$ and  $\lambda=\frac{b}{a}$. 

So the missing vector subspace $G$ corresponds to the forbidden fraction $\frac{1}{0}=\infty$ and can be called $G_\infty$.

In the case where $k=\mathbb{R}$, this is the well-known embedding of $\mathbb{R}$ in the circle of \emph{ directions up to sign}.
Extending $\iota:k\cup \{\infty\}\to \mathbb{P}^1(k)$ by $\iota(\infty)=G_\infty$ gives the interpretation of $k\cup \{\infty\}$ as the projective space $\mathbb{P}^1(k)$. Each linear automorphism $A$ of the $k$-vector space $k^2$
induces a self-bijection $G_A$ of $k\cup\{\infty\}=\mathbb{P}^1(k)$. If the matrix of $A$ is the $2\times 2$-matrix $(\begin{smallmatrix} a & b\\ c & d\\  \end{smallmatrix})\in \mathrm{GL}(2,k)$, 
then $G_A(G_\lambda)=G_{\lambda'}$
with $\lambda'=\frac{a\lambda+b}{c\lambda+d}$.
The transformation $G\in \mathbb{P}^1(k) \mapsto G_A(G)\in \mathbb{P}^1(k)$ or $\lambda\mapsto \frac{a\lambda+b}{c\lambda+d}$ is called a fractional linear or M\"obius transformation.
Note that in particular $G_A(G_\infty)=\frac{a}{c}$.

Given a general field $k$, an important structure on $\mathbb{P}^1(k)$ is provided by a $4$-point function which we shall study in \S \ref{s:shadow}.

At this stage, we restrict to the case $k=\mathbb{C}$. The above construction of $\mathbb{P}^1(k)$  for an arbitrary field $k$ gives the familiar construction of $\mathbb{P}^1(\mathbb{C})= (\mathbb{C}^2-\{0\})/\mathbb{C}^*$, where $\mathbb{C}^*$ denotes the multiplicative group of nonzero complex numbers.

The set $\mathbb{C}\cup\{\infty\}=\mathbb{P}^1(\mathbb{C})$ carries many structures. First, there is the structure of a differentiable manifold given by the following atlas: We set
$U_0=\mathbb{P}^1(\mathbb{C})\setminus \{G_\infty\}$ and 
$U_\infty=\mathbb{P}^1(\mathbb{C})\setminus 
\{G_0\}$. Observe that $U_0=\{G_\lambda\mid \lambda\in\mathbb{C}\}$ and that  every $G\in U_\infty$ is of the type $G'_\sigma=\{(\sigma b,b)\mid  b\in \mathbb{C}\}$ for $\sigma\in \mathbb{C}$.

Define maps $z_0:U_0\to \mathbb{C}$ by
$z_0(G_\lambda)=\lambda$ and $z_\infty:U_\infty\to \mathbb{C}$ by
$z_\infty(G'_\sigma)=\sigma$. 
Both maps are bijections. For $G\in U_0\cap U_\infty$ the two maps are related; indeed, 
$z_0(G)z_\infty(G)=1$. It follows that the system $((U_0,z_0),(U_\infty,z_\infty))$ is
an atlas for a manifold structure with coordinates functions $(z_0,z_\infty)$. Its quality is hidden in the quality of the coordinate change. For  $G\in U_0\cap U_\infty$, from the above implicit relation it follows that
$z_\infty(G)=1/z_0(G), \,z_0(G)=1/z_\infty(G) $. This coordinate change is differentiable; therefore $\mathbb{C}\cup\{\infty\}=\mathbb{P}^1(\mathbb{C})$ is a smooth manifold with charts in the Gaussian plane $\mathbb{C}$.

The smooth $2$-dimensional real manifold $\mathbb{C}\cup\{\infty\}=\mathbb{P}^1(\mathbb{C})$ is diffeomorphic to the unit sphere in the three-dimensional real vector space $\mathbb{R}^3$. More precisely, the coordinate change
$\phi_{\infty,0}:\mathbb{C}^*=z_0(U_0)\to z_\infty(U_\infty)=\mathbb{C}^*$ is given in terms of the natural coordinate $z$ on $\mathbb{C}^*$, by $\phi_{\infty,0}(z)=1/z$. The smooth map $\phi_{\infty,0}:\mathbb{C}^*\to \mathbb{C}^*$ is moreover holomorphic, so the above atlas provides $\mathbb{C}\cup\{\infty\}=\mathbb{P}^1(\mathbb{C})$ with the structure of a Riemann surface. This Riemann Surface is the {\it Riemann Sphere}.\index{Riemann sphere}

Let $U$ be an open subset of the Gaussian plane $\mathbb{C}$. Riemann defined a map $\phi:U\to \mathbb{C}$ to be holomorphic without using an expression that evaluates the map at given points. The idea is the following. The real tangent bundles $TU$ and $T\mathbb{C}$ come with a field $m_i$ of endomorphisms.  (The notation $m_i$ stands for ``multiplication by $i$".) The value $m_{i,p}$ of the field $m_i$ at the point $p$
is the linear map $m_{i,p}:T_pU\to T_pU, \, u\mapsto iu$. To be holomorphic by Riemann's definition is given by the following property of the differential:
$$(D\phi)_p(m_{i,p}(u))=m_{i,\phi(p)}((D\phi)_p(u)).$$
In words, this means that the differential $D\phi$ is $\mathbb{C}$-linear.

Riemann's characterization of holomorphic maps together with the local $J$-Rigidity Theorem
\ref{th:J}
allows us to define a Riemann surface\index{Riemann surface}\index{J@$J$-structure} $(S,J)$ as a real 2-dimensional  differentiable manifold $S$ equipped with a smooth field of endomorphisms $J:TS\to TS$ of its tangent bundle satisfying $J\circ J=-{\rm Id}_{TS}$.

The Riemann Sphere is the first example of a compact Riemann surface. The most familiar non-compact Riemann surface is the Gaussian plane $\mathbb{C}$. Another most important Riemann surface is the unit disc in $\mathbb{C}$. This is also the image of the southern hemisphere by the stereographic projection from the North pole onto a plane passing through the equator. This projection is holomorphic. The importance of the unit disc stems from the fact that it is equipped with the Poincar\'e metric, which makes it a model for the hyperbolic plane.

 \subsection{The group ${\rm SU}(2)$ and its action on the Riemann sphere.}\label{ss:Hermitian}
 
 Now that we are familiar with the Riemann sphere, we study a group action on it.

Let $<u,v>_{\rm Herm}$ be the usual Hermitian product on $\mathbb{C}^2$. This is the complex bilinear form on $\mathbb{C}$ defined by $<u,v>_{\rm Herm}=u_1\bar{v_1}+u_2\bar{v_2}$. The Hermitian perpendicular $L^\perp$ to a complex vector subspace $L$ is again a complex vector subspace.

The group of determinant $1$
linear transformations of $\mathbb{C}^2$ that preserve $<u,v>_{\rm Herm}$ is the group ${\rm SU}(2)$\index{SU2@${\rm SU}(2)$} consisting of all matrices of the form  $(\begin{smallmatrix} a & b\\ 
-\bar{b}& \bar{a}\end{smallmatrix}),\, (a,b)\in \mathbb{C}^2,\, a\bar{a}+b\bar{b}=1$. This group acts on the Riemann sphere by M\"obius transformations, in fact, by rotations. The map $(\begin{smallmatrix} a & b\\ 
-\bar{b}& \bar{a}\end{smallmatrix})\mapsto (a,b)$ defines a diffeomorphism  ${\rm SU}(2)\to S^3$ and induces a Lie group structure on the sphere $S^3$.

The group ${\rm SU}(2)$ acts transitively by conformal automorphisms on the Riemann sphere $\mathbb{P}^1(\mathbb{C})$. The stabilizer of $L=\{(\lambda,0)\mid \lambda\in \mathbb{C}\}$ in ${\rm SU}(2)$ is the group 
$(\begin{smallmatrix} a & 0\\ 
0 & \bar{a}\end{smallmatrix}),\, a \in \mathbb{C}, a\bar{a}=1$, which is isomorphic to the group of complex numbers of norm $1$. The quotient construction induces a Riemannian metric on $$\mathbb{P}^1(\mathbb{C})={\rm SU}(2)/{\rm Stab}_{{\rm SU}(2)}(L).$$

A \emph{marked element} in $\mathbb{P}^1(\mathbb{C})$ is a pair $(L,u)$ where
$u=(a,b)\in \mathbb{C}^2, a\bar{a}+b\bar{b}=1$ and $L=[u]=\{\lambda u\mid \lambda\in \mathbb{C}\}$. Note that this representation is redundant since $u$ determines $L=[u]$.

The group ${\rm SU}(2)$ acts simply transitively on marked elements in 
$\mathbb{P}^1(\mathbb{C})$. 

The involution $L \mapsto L^\perp$ extends to marked elements: map $(L,u)=(L,(a,b))$ first to $u^\perp=(\bar{b},-\bar{a})$ and next to
$(L^\perp,u^\perp)$ with $L^\perp=[u^\perp]$.

A marked element $(L,u)$ determines a path in $\mathbb{P}^1(\mathbb{C})$ by
$L_u:t\in [0,\pi]\mapsto L_u(t)=[\cos(t)u+\sin(t)u^\perp]$, which in fact is a simple closed curve.  Its velocity at $t=0$ is a length $1$ tangent vector $V_u\in T_{[u]} (\mathbb{P}^1(\mathbb{C}))$. Observe that $V_u=V_{-u}$ and $V_{iu}=-V_{u}$. The path $L_u$ lifts to $H^\perp_u:t\in [0,\pi] \mapsto H^\perp_u(t)=\cos(t)u+\sin(t)u^\perp \in S^3$, which is a geodesic from $u$ to $-u$  perpendicular to the foliation on $S^3$ by the Hopf circles $H_v=\{v'\in S^3\mid v'=\lambda v\},\, v\in S^3$. Hopf circles $H_v$ map to points, and geodesics $H^\perp_u$ map to simple closed geodesics in $\mathbb{P}^1(\mathbb{C})$.

 The map $\pm u\in S^3/\{\pm {\rm Id}\}=\mathbb{P}^3(\mathbb{R}) \mapsto V_u\in T(\mathbb{P}^1(\mathbb{C}))$ induces a bijection onto the length $1$ vectors  to $\mathbb{P}^1(\mathbb{C})$. Observe that ${\rm SU}(2)$ acts almost simply transitively on length $1$ tangent vectors to $\mathbb{P}^1(\mathbb{C})$. The quotient group ${\rm PSU}(2)={\rm SU}(2)/\{\pm {\rm Id}\}$ acts simply transitively on length $1$ tangent vectors.

\section{All three planar geometries and hyperbolic 3-space simultaneously} \label{s:Models}

\subsection{A stratification of the Riemann sphere arising from algebra} \label{s:ring}
Bernhard  Riemann was aware of the (Riemann) sphere being the complex plane union a point at infinity. His point of view on complex analysis was very geometric. In this section, we wish 
to describe an incarnation of the Riemann sphere which arises from algebra.\index{Riemann sphere!model}\index{model!Riemann sphere} For more details on this model, see \cite[Chap. 3, \S 3.1]{A2021} and \cite[Chap. 1, \S 8.3]{ACP2012}.

The starting object is the set $\Sigma$ of surjective  ring homomorphisms \index{model!Riemann sphere} from the ring $\mathbb{R}[X]$ of polynomials in one unknown $X$ with real coefficients to a field $F$. On the set $\Sigma$ we introduce two equivalence relations. The first relation, $\sim$, declares $f\colon \mathbb{R}[X]\to F$ and $f'\colon \mathbb{R}[X]\to F'$ to be equivalent if there exists a field isomorphism $\phi\colon F\to F'$ with $f'=\phi\circ f$.

The relation $f\sim f'$ holds if and only if 
the ideals ${\rm kernel}(f),\,{\rm kernel}(f')$ in $\mathbb{R}[X]$ are equal.

The second relation, $\sim_X$, requires $f\sim f'$ and moreover $f(X)=f'(X)$ holds in $\mathbb{R}[X]/{\rm kernel}(f)=\mathbb{R}[X]/{\rm kernel}(f')$.

Up to field isomorphism, there are only three fields, $F$, that are
 hit by a surjective ring homomorphism
$f: \mathbb{R}[X]\to F$, namely, the fields
$\mathbb{C}$, $\mathbb{R}$ and $\mathbb{F}_1$ where
$\mathbb{F}_1$ is the field with one element, that is, the field where $0=1$ holds. The field $\mathbb{F}_1$ corresponds to the ideal $\rho=\mathbb{R}[X]$ consisting of the whole ring, which is prime, maximal but not proper. 

\begin{exercise}
The two fields $\mathbb{R},\, \mathbb{F}_1$ have only the identity as automorphism and the field 
$\mathbb{C}=\{a+bi\mid \,a,b\in\mathbb{R}\}$ only two automorphisms as $\mathbb{R}$-algebra, but as many field automorphisms as the power set of the real numbers.
\end{exercise}

In the following we will describe the quotient sets $\Sigma/\sim,\,\Sigma/\sim_X$ together with natural structures on these sets.

All ideals in $\mathbb{R}[X]$ are principal, that is, any such ideal is generated by a single element (it is obtained by multiplication of such an element by an arbitrary element of the ring).
Kernels of $f\in \Sigma$ are prime ideals, that is, the quotient of  $\mathbb{R}[X]$ by such an ideal is an integral domain (the product of any two nonzero elements is nonzero).
Thus, we have three kinds of kernels of $f$, namely,
$\rho=(1)=\mathbb{R}[X]$, $(X-a), a\in \mathbb{R}$, and
$((X-a)^2+b^2),a,b\in \mathbb{R},b>0$. Therefore the set
$\Sigma/\sim$ is identified with $\mathbb{C}_+\cup \mathbb{R}\cup \{\rho\}$. 
Here we use the notation $\mathbb{C}_\pm=\{a+bi\mid a,b\in \mathbb{R},\pm b>o\}$ for the upper/lower half planes.

The kernel of the ring homomorphism $f$ is not sufficient in order to describe its class in $\Sigma/\sim_X$ if ${\rm kernel}(f)=((X-a)^2+b^2)$. One needs moreover to specify a root $a+bi\in \mathbb{C}_+$ or $a-bi\in \mathbb{C}_-$. Thus the set $\Sigma/\sim_X$ is a disjoint union of $4$ strata $\Sigma/\sim_X=\mathbb{C}_+\cup\mathbb{C}_-\cup\mathbb{R}\cup \{\rho\}$.

The fields $\mathbb{C},\mathbb{R}, \mathbb{F}_1=\{0\}$ are realized as sub-$\mathbb{R}$-algebras in $\mathbb{C}$, so an alternative description of the set $\Sigma/\sim_X$ is the set of $\mathbb{R}$-algebra homomorphism from $\mathbb{R}[X]$ to $\mathbb{C}$.

 We shall see that the set $R=\Sigma/\sim_X$ and its strata carry a rich panoply of structures. The set $R=\mathbb{C}_+\cup\mathbb{C}_-\cup\mathbb{R}\cup \{\rho\}$  is identified with the Riemann sphere $\mathbb{C}\cup \{\infty\}$.\index{Riemann sphere!model}\index{model!Riemann sphere} Structures, such as the Chasles three point function (defined below) on $\mathbb{R}\subset R$,  or the hyperbolic geometry on $\mathbb{C}_+$, will appear naturally. Naturally means here that the construction that leads to the structure commutes with 
 the $\mathbb{R}$-algebra automorphisms of $\mathbb{R}[X]$. For instance, it commutes with the substitutions that consist in translating $X$ to $X-t,\, t\in \mathbb{R},$ or with stretching $X$ to $\lambda X,\, \lambda\in \mathbb{R}^*$.
The ideal $(X-a)$ maps to the ideal $(X-a-t)$ by translation and to $(X-\frac{a}{\lambda})$ by stretching. 
 
A first example is the Chasles $3$-point function\index{Chasles $3$-point function} $\mathrm{Ch}(A,B,C)$ on the stratum 
$\mathbb{R}$ 
 consisting of the ideals $ (X-a),a\in \mathbb{R},$ defined as follows: Given three distinct such points, $A=(X-a),B=(X-b),C=(X-c)$, define
 $\mathrm{Ch}(A,B,C)=\frac{b-a}{c-a}$. 
 
 In words, $\mathrm{Ch}(A,B,C)$ is the ratio of the monic generators of $B$ and $C$ evaluated at the zero of the monic generator of $A$. 
 
The next example is the $4$-point function cross ratio\index{cross ratio} $\mathrm{cr}(A,B,C,D)$: for $4$ distinct points $A=(X-a),B=(X-b),C=(X-c),D=(X-d)$, define $\mathrm{cr}(A,B,C,D)=\mathrm{Ch}(A,B,C)\mathrm{Ch}(D,C,B)=\frac{b-a}{c-a}.\frac{c-d}{b-d}$. In words, this is Chasles evaluated at the first three points times Chasles evaluated at the last three points in the reverse order. It is truly a remarkable fact that the cross ratio function extends to a $4$-point real function on $\mathbb{P}^1(\mathbb{R})$ and to a $4$-point complex function on $\mathbb{P}^1(\mathbb{C})=\mathbb{C}\cup \{\rho\}$ if one transfers the above wishful calculus with $w=\infty$ to $\rho$.

The multiplicative monoid $\mathbb{R}^*[X]$ of polynomials which do not vanish at $0$ has also automorphisms that do not directly fit with the interpretation as polynomials with unknown $X$. In particular, they are not  ring automorphisms, but monoid automorphisms. A main example is the \emph{twisted palindromic symmetry}:\index{twisted palindromic symmetry} 
perform on a polynomial $P(X)$ the substitution $X \to  \frac{-1}{X}$, followed by stretching with factor $(-X)^{{\rm degree}(P)}$. (The palindromic symmetry\index{palindromic symmetry} is said to be twisted, because of the minus signs.) Then the ideal $(X-a)$ maps to the ideal $(-X(\frac{-1}{X}-a))=(1+aX)=(X+\frac{1}{a})$. The Chasles function restricted to $\mathbb{R}^*$ does not commute with the symmetry $a\mapsto \frac{-1}{a}$, but the cross ratio commutes. (This property is among the ones that make the cross ratio more natural than the Chasles $3$-point function.)\index{Chasles $3$-point function} This symmetry, which is an involution, extends to a fixed point free involution $\sigma_\mathbb{P}$ of $\mathbb{R}\cup \{\rho\}=\mathbb{P}^1(\mathbb{R})$.  Remarkably, the  symmetry $\sigma_\mathbb{P}$ commutes with the cross ratio $\mathrm{cr}$.
In this sense, $\mathrm{cr}$ is more natural than $\mathrm{Ch}$.

The above operations of real translation and stretching, i.e., substituting $X-t$ for $X$ or $\lambda X$ for $X$ with $t\in \mathbb{R},\,\lambda\in \mathbb{R}^*$, together with the twisted palindromic symmetry induce bijections of the set $\{(X-u)(X-\bar{u})\mid u\in \mathbb{C}_+\}$ of monic polynomials of degree $2$ without real roots. Composing these bijections
generates a group $G$. It is a remarkable fact that this group is, as an abstract group, isomorphic to the group
${\rm PGL}(2,\mathbb{R})$. It is also a remarkable fact that the abstract group ${\rm PGL}(2,\mathbb{R})$ carries a unique structure of Lie group. So there is also a topology on $G$, which allows us to define the subgroup 
$G_0\subset G$ as the connected component of the 
neutral element in the Lie group $G={\rm PGL}(2,\mathbb{R})$. The group $G_0$ is isomorphic to the group ${\rm PSL}(2,\mathbb{R})$. 

The fixed point free involution $\sigma_\mathbb{P}$
on $\mathbb{P}^1(	\mathbb{R})=\mathbb{R}\cup \{\rho\}=\partial\bar{C}_+$ extends to $\mathbb{C}_+$ by putting
$\sigma_\mathbb{P}(u)=\frac{-1}{u}$ for an involution with $i$ as unique fixed point. 

The group $G_0$ acts transitively and faithfully on the above strata $\mathbb{C}_\pm$ and on $\mathbb{R}\cup \{\rho\}$. From this action one gets a topology on the strata and also, as we will explain, a geometry on 
$\mathbb{C}_+$. It is also remarkable that this geometry, in fact, the planar hyperbolic geometry, can also  be explained in a more elementary way in term of the interpretation as ideals.

The action of $G_0$ on $\mathbb{C}_+$ is the so-called modular action\index{modular action}  of ${\rm PSL}(2,\mathbb{R})$ on $\mathbb{C}_+$. Thinking of an element $g\in G_0$ as a real $2\times 2$ matrix of determinant $1$ up to sign, $\pm( \begin{smallmatrix}
a & b \\
c & d 
\end{smallmatrix} )$, the action on $u\in \mathbb{C}_+$ is given by $(g,u)\mapsto \frac{au+b}{cu+d}$. 

The modular action of $G_0$ on $\mathbb{C}_+$ extends to the projective action of $G={\rm PGL}(2,\mathbb{R})$ on $\partial \bar{\mathbb{C}}_+=\mathbb{P}^1(\mathbb{R})$ and also to the projective action of the complex group ${\rm PGL}(2,\mathbb{C})$ on $\mathbb{P}^1(\mathbb{C})$.

The hyperbolic geometry on $\mathbb{C}_+$ can also be defined in terms of ideals in the following rather elementary way. 

Points are the elements of $\mathbb{C}_+$. The monic generator of the corresponding ideal may be denote by $P_u(X)= (X-u)(X-\bar{u})$ . 

The notion of line is introduced using convex combinations.
More precisely, given two distinct points $u,v$, the line $L_{u,v}$ through them is defined as follows: 

Denote by
$P_{t,u,v}(X)=tP_u(X)+(1-t)P_v(X),\, t\in [0,1]$ the convex combination of the polynomials $P_u(X),P_v(X)$. Note first that any polynomial obtained in this way is monic. Let
$I_{u,v}=]m_{u,v},M_{u,v}[$ be the maximal interval on which $P_{t,u,v}(X)$ has no real roots. Note that $[0,1]\subset I_{u,v}\not= \mathbb{R}$ (one may start by checking that for $t=0$ and $t=1$ there are no real roots).
This implies that the cross ratio $\delta(u,v)=\mathrm{cr}(m_{u,v},1,0,M_{u,v})$ is a positive real number $>1$. If ${\rm Re}(u)={\rm Re}(v)$ define $L_{u,v}=\{w\in \mathbb{C}_+\mid {\rm Re}(w)={\rm Re}(u)={\rm Re}(v)\}$. If ${\rm Re}(u)\not={\rm Re}(v)$ let $c\in \mathbb{R}$ be the zero of
$P_u(X)-P_v(X)$ and define $L_{u,v}=\{w\in \mathbb{C}_+\mid |w-c|^2=P_u(c)\}$. 

With such a definition, the axiom of hyperbolic geometry saying that for any two distinct points there is a unique line passing through them is trivially satisfied.

\begin{exercise}
For a better understanding of this construction of the hyperbolic plane, please check that the root of $P_{t,u,v}(X)$ travels along  $L_{u,v}$ for $t\in I_{u,v}$ in case ${\rm Re}(u)\not={\rm Re}(v)$. Note that $L_{u,v}$ is a half circle in $\mathbb{C}$ with centre 
in $\mathbb{R}\cup \{\rho\}$.
\end{exercise}
The hyperbolic Riemannian metric $g_u$ is the Hessian at $u$ of the function $v\in \mathbb{C}_+ \mapsto D(u,v)^2$.

In this model, we can
define the hyperbolic distance between $u$ and $v$ as $D(u,v)=\frac{1}{2}\log(\delta(u,v))$. We can make a relation with the model of the hyperbolic plane that uses the Hilbert metric.
In this way we have all the ingredients of hyperbolic geometry based on elementary algebra, that can be taught at high-school.

\begin{remark} The above twisted palindromic map $\sigma\colon \mathbb{R}[X]\to \mathbb{R}[X]$
induces an involution on $\mathbb{R}^*[X]$, the complement of the hypersurface of polynomials that vanish at $0$. 
Therefore, $\sigma$ is a birational involution and the above constructions would generate subgroups $G$, etc. in the Cremona group of birational transformations of $\mathbb{P}^1(\mathbb{C})$. Thus, the artefact of defining $\sigma_\mathbb{P}$ can be avoided. Even better, the geometry of Cremona groups is hyperbolic. For the Cremona group, see  \cite{Blanc, Blanc1, Cantat, Deserti, Serre, UZ, Z}.
\end{remark}

\subsection{ A note on the field with one element}

The importance of the field with one element\index{field with one element} $\mathbb{F}_1$, which was interpreted in the above construction of the Riemann sphere as the point at infinity of the complex plane, was first highlighted by Jacques Tits in his paper \cite{Tits}. In this paper, Tits proposed the development of a   geometry over the field $\mathbb{F}_1$  which would be the limit of geometries over the finite fields $\mathbb{F}_q$, where $\mathbb{F}_q$ denotes the field of cardinality $q=p^n$ where $p$ is a positive prime. Note that  if we make $n$ tend to $0$ in the notation $\mathbb{F}_q$, $q=p^n$, we obtain $\mathbb{F}_1$, which is an indication of the fact that the field  $\mathbb{F}_q$ may be considered as a deformation of the field $\mathbb{F}_1$. One may also make an analogy with the fact that a limit $n\to 0$ is used in quantum topology  for a geometric interpretation of the TQFT calculus of Turaev, Viro, Kauffman, etc. (The analogy is vague, but this is the nature of analogies.) 

After Tits introduced his idea of studying  a geometry over the field $\mathbb{F}_1$, many works were done on this theme. For instance, 
Manin, in his lectures on the zeta function \cite{Manin}, proposed the study of a ``Tate motive over a one-element field". In the paper \cite{Connes} titled \emph{Fun with  $\mathbb{F}_1$} by Connes, Consani and Marcolli, the field $\mathbb{F}_1$ becomes an actor in an approach to the Riemann hypothesis.

\subsection{The Riemann sphere and shadow numbers} \label{s:shadow}

Let\index{shadow number} $V$ be a $2$-dimensional vector space over a field $k$.  If $k$ has $3$ or more elements, then the space $\mathbb{P}^1(k)$ of lines through the origin in $V$ has at least $4$ elements. 
We wish to define in a geometric way the cross ratio of
a figure $L_1,L_2,L_3,L_4$ consisting of $4$ distinct lines in $V$. This is done in the form of two exercises.

\begin{exer}\label{ex:1}
Do the following: Identify $V=L_1\oplus L_4$ with the Cartesian product of $L_1$ and $L_4$, and let $\phi_i$ be the linear map from $L_1$ to $L_4$ that has $L_i$ as graph, $i=2,3$. Define
$\mathrm{cr}(L_1,L_2,L_3,L_4)$ as the stretch factor of the linear map $\phi_3^{-1}\circ \phi_2\colon L_1 \to L_1$. This is represented in Figure \ref{cross1}, in which we start with the point $X$ on $L_1$, with $Y$ its image by $\phi_2$ and $Z$ the image of $Y$ by $\phi_3^{-1}$.

Please check now that if the lines are defined using a basis
$e,f$ and ``numbers" $a_i\in k\cup \{\infty\},i=1,2,3,4$ such that $e+a_if\in L_i$, then the formula
$\mathrm{cr}(L_1,L_2,L_3,L_4)=\frac{a_2-a_1}{a_3-a_1}\frac{a_4-a_3}{a_4-a_2}$ holds.

In this way, we recover the usual formula for the cross ratio.
\end{exer}

   \begin{figure}
\begin{center}
\includegraphics[width=12 cm]{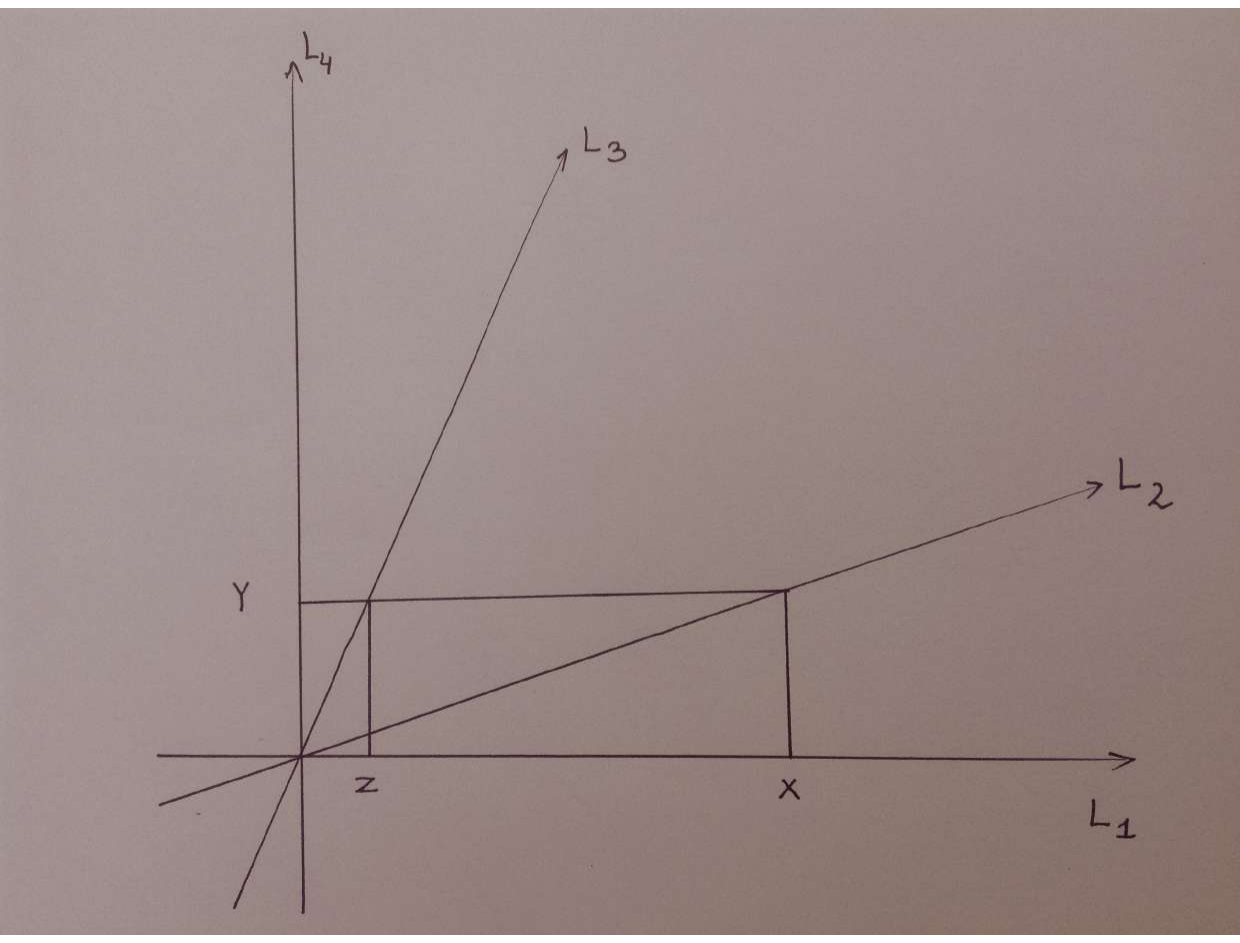} 
\caption{Figure for Exercise \ref{ex:1}}
\label{cross1}
\end{center}
\end{figure}

In the above construction of the cross ratio, the use of a Cartesian product  and graphs of maps suggest that parallel lines are essential. This is not the case. For instance, the construction also works
inside a Euclidean triangle with vertices $ABC$, as in the following exercise, which asks for a construction of the cross ratio which is a projective version of the construction in Exercice \ref{ex:1}, which is affine (it uses the notion of parallels):

\begin{exer}\label{ex:2}
Let $f_2,f_3$ be interior points of the side $BC$ of a triangle $ABC$, put $f_1=B,\,f_4=C$. Let
$L_i,\, i=1,2,3,4,$ be the segments connecting the vertex $A$ with $f_i$. 

Define $\phi_2\colon L_1\to L_4$ as follows: Connect $X\in L_1$ by a segment with $C$ which insects $L_2$ in $X'$. The half-line
$[B,X')$ intersects $L_4$ in $Y=\phi_2(X)$. Define  
in the same way $\phi_3$. Now $\phi_3^{-1}\circ \phi_2\colon L_1 \to L_1$ is not linear, but with fixed point $A$. Define the cross ratio $\mathrm{cr}(L_1,L_2,L_3,L_4)$ to be the stretch factor of the differential at $A$ of the map $\phi_3^{-1}\circ \phi_2$. The construction is represented in Figure \ref{cross2}, where, like in  Figure \ref{cross1},  $X$ is a point in $L_1$,  $Y$ its image by $\phi_2$ and $Z$ the image of $Y$ by $\phi_3^{-1}$.

 \end{exer}

   \begin{figure}
\begin{center}
\includegraphics[width=12 cm]{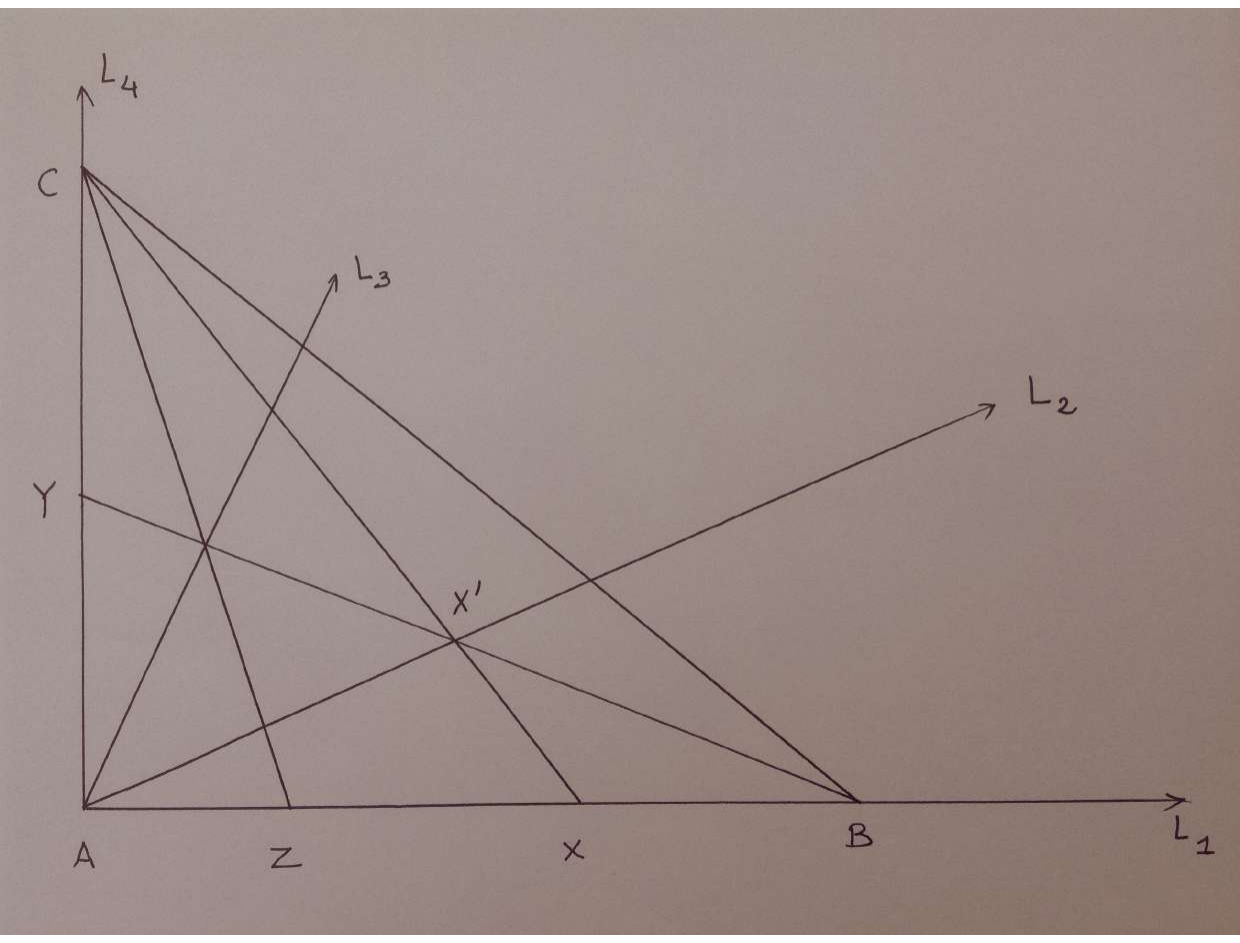} 
\caption{Figure for Exercise \ref{ex:2}}
\label{cross2}
\end{center}
\end{figure}

It is important to observe that $\mathrm{cr}(L_1,L_2,L_3,L_4)$,  in planar Euclidean geometry,
does not depend upon the position of the side $BC$. This property does not hold in spherical or hyperbolic geometry. In spherical geometry there is a preferred choice for the side $BC$, namely, such that the triangle becomes bi-orthogonal. This was already done by Menelaus of Alexandria!\index{Menelaus theorem}

This construction also works for a configuration in general position of $4$ linear subspaces $E_1,E_2,E_3,E_4$ of dimension $2$ in a real vector space $V$ of dimension $4$. The invariant is now an element  $\Lambda(E_1,E_2,E_3,E_4)\in {\rm GL}(E_1)$. One may use this to prove the following:

\begin{theorem} \label{th:four}  {\bf The four Complex Lines Theorem.}  Let $A=(E_1,E_2,E_3,E_4)$ be a generic labelled configuration of planes in a real $4$-dimensional vector space $V$.
There exists a linear complex structure $J\colon V\to V$ with $J(E_j)=E_j$ (which means that each $E_j$ is a complex line)  for every $i=1,2,3,4$ 
if and only if 
$${\rm Trace}(\Lambda(A))^2<4{\rm Det}(\Lambda(A))$$ 
or if 
$\Lambda(A)$ is a multiple $\lambda {\rm Id}_{E_1}$ of the identity ${\rm Id}_{E_1}$, for some $\lambda\not= 1$.
\end{theorem}

The proof is contained in p. 73 of \ref{A2021}

 Theorem \ref{th:four} should be seen in  the setting of the following general question:
 Given a $4$-dimensional real vector space $V$ and a quadruple of 2-dimensional planes $(E_1,E_2,E_3,E_4)$ in $V$ satisfying some properties, does there exist an almost complex structure $J:V\to J$ making $(V,J)=\mathbb{C}^2$ such that $(E_1,E_2,E_3,E_4)$ are complex lines?

The usual formula for the cross ratio with numbers is memo-technically speaking a headache. A more geometrically-rooted name would be more satisfying. We propose {\it shadow number}:\index{shadow number} the shadow issued from a light bulb of four concurrent lines lying in a plane on a plane shows the same number. The proof uses the above special property in Euclidean geometry and becomes simplified if one uses the above geometric definition. If the name should remember a person, then perhaps Leonardo da Vinci number,\index{Leonardo da Vinci} for Leonardo  studied central projections and perspectivities\index{perspectivity}  about $500$ years ago \cite{D}, or Menelaus number,\index{Menelaus number} for Menelaus studied shadows in spherical geometry about\index{Menelaus of Alexandria} $2000$ years ago. See 
 \cite{2012-Menelaus1} for the use  by Menelaus of the invariance of the cross ratio in spherical geometry. Menelaus used this result in the proof of Proposition 71 of the \emph{Spherics},\footnote{Menelaus was extremely concise in his \emph{Spherics}, and  the proofs of some of the propositions in this work are very difficult to follow. This is why several results of the \emph{Spherics} \cite{RP2018} were later explained and commented on by Arab mathematicians of the Middle Ages, after the Greek mathematical schools were desintegrated. Regarding this particular proposition, see the two papers \cite{2012-Menelaus1, 2012-Menelaus2}} and later medieval commentators, in an effort to provide full proofs of some of this and other difficult propositions in Menelaus'  \emph{Spherics}, highlighted this invariance as a new proposition, in order to explain a proof in the \emph{Spherics};\index{Menelaus of Alexandria!\emph{Spherics}} see Proposition 3.2 in \cite{2012-Menelaus1} and \cite[p. 356-360]{RP2018}.

The group ${\rm PGL}(2,\mathbb{C})$ of linear transformations of $\mathbb{C}^2$ acts on the Riemann sphere $R$ (see \S \ref{s:ring}), since this group acts on complex lines through the origin of $\mathbb{C}^2$. From the geometric definition of the shadow number,\index{shadow number} it follows that this number is ${\rm PGL}(2,\mathbb{C})$-invariant:
$$\lambda_{pqts}=\lambda_{Ap,Aq,At,As},\,A\in{\rm PGL}(2,\mathbb{C}).$$ 

The shadow number is a $4$-point function defined on the complement of the general diagonal in the Riemann sphere $R$:
$$\lambda:R^4\setminus {\rm Diag(R)}\to \mathbb{C}$$
where ${\rm Diag(R)}$ is the subset of quadruples $(p,q,t,s)\in R^4$ of points in $R$ with $\#\{p,q,t,s\}<4$. Using the above formula one checks that the function $\lambda$ is holomorphic with meromorphic extension to $R^4$.  

\begin{exer}
For $p,q,t$ a triple of distinct points, study the partial function $f:s\mapsto \lambda_{pqts}$. What are the level sets of $f,|f|,f\pm\bar{f}$?
\end{exer}

\begin{exer}
Reconstruct the complex structure $J$ of $R$ from the shadow function\index{shadow function} $\lambda$.
\end{exer}

\begin{exer}
For a $4$ point function $f:R^4\setminus {\rm Diag(R)}\to \mathbb{C}$, define ${\rm Aut}(R,f)$ to be the group of bijections $A:R\to R$ satisfying $f\circ A=f$.
Consider the cases $f=\lambda$ or $f=|\lambda|$ or $f=\lambda\pm\bar{\lambda}$.
Show  that
${\rm PGL}(2,\mathbb{C})={\rm Aut}(R,\lambda)$ and $ {\rm Aut}(R,|\lambda|)={\rm Aut}(R,\lambda\pm\bar{\lambda})$.
\end{exer}

\subsection{$J$-compatible metrics} \label{s:compatible}

In this section, $(S,J)$ is a differentiable compact connected surface equipped with a complex structure $J:TS\to TS$. A Riemannian metric $g$ on $S$ is called conformal\index{J@$J$-compatible metric}\index{J@$J$-conformal metric} with $J$ if for every point 
$p\in S$ the map $J_p:T_pS\to T_pS$ is $g_p$-orthogonal. Thus, two metrics\index{J@$J$-conformal metric}\index{J@$J$-compatible metric} $g, g'$ conformal with $J$ differ pointwise by a positive factor: $g'=fg$ for some function $f$.
A  $J$-calibrated\index{J@$J$-calibrated volume form}\index{volume form!J@$J$-calibrated} volume form $\omega$ on $S$, i.e.,  a differential 2-form satisfying $\omega_p(u,J_pu)>0,\, p\in S, \, u\in T_pS,\, u\not= 0,$ gives by $g_p(u,v)=\omega_p(u,J_pv)$ a Riemannian metric conformal with $J$ whose associated volume form is precisely $\omega$. All metrics $g$ conformal with $J$ are obtained in a unique way from such a construction by taking for $\omega$ the oriented volume form $\omega(u,Ju)=g(u,u)$ of the Riemannian metric $g$.

In conclusion, given $J$, the space of Riemannian metrics\index{J@$J$-conformal metric}\index{J@$J$-compatible metric} conformal with $J$ is parametrized by the cone of $J$-calibrated volume forms.\index{volume form!J@$J$-calibrated}\index{J@$J$-calibrated volume form}  In this way, the uniqueness is not only up to multiplying by a constant factor. 

We wish to strengthen the notion of being ``conformal with $J$"\index{J@$J$-compatible metric}\index{J@$J$-conformal metric} for metrics and gain back uniqueness or controlled non-uniqueness up to multiplying by a constant factor. This will be achieved  first for Riemann surfaces $(S,J)$ that are homogeneous, i.e., whose group ${\rm Aut}(S,J)$ acts transitively on $S$, such that moreover the stabilizers of points are compact.

The group ${\rm Sim}(M,g)$\index{similarity group} of a connected Riemannian manifold is the group of diffeomorphisms that multiply the metric by a constant factor. Such diffeomorphisms are called similarities.\index{similarity}

\begin{theorem} Let $(S,J)$ be a homogeneous Riemann surface with commutative stabilizers. Then there exists a Riemannian metric $g$ conformal with\index{J@$J$-conformal metric}\index{J@$J$-compatible metric} $J$ such that
$\mathrm{Sim}^+(S,g)=\mathrm{Aut}(S,J)$. The Riemannian metric $g$ is unique up to multiplication by a constant factor.
\end{theorem}

\begin{proof}[Sketch of proof] 
The list of homogeneous connected Riemann surfaces
up to bi-holomorphic equivalence  is short: 
\begin{enumerate}
\item The Riemann sphere $R$; 
\item $R\setminus\{\infty\}$; 
\item $R\setminus\{0,\infty\}$;
\item the infinite family of elliptic curves $\mathbb{C}/\Gamma$;
\item $\mathbb{C}_+=\{f_{a+bi}\mid a,b\in \mathbb{R},\,b>0\}$.
\end{enumerate}

In this list, only the Riemann sphere\index{Riemann sphere!model}\index{model!Riemann sphere} $R$ has non-compact and non-commutative stabilizers---the stabilizer of a point is the affine group of $\mathbb{C}$. Thus, we do not need to consider this surface.

The second surface in the list is bi-holomorphic to $\mathbb{C}$. The Euclidean metric $g_z(u,u)=u\bar{u}$ is up to multiplication by a constant factor the only metric whose bi-holomorphic equivalences are similarities.
Note that its group of isometries is a strict subgroup of the group of similarities. It is not commutative. So we also do not consider this surface.

The third surface in the list is the image of the covering map
$z\in \mathbb{C} \mapsto e^{2\pi i z}	\in \mathbb{C}^*$ with deck transformations $z\mapsto z+k,\,k\in \mathbb{Z}$. These deck transformations are isometries of the Euclidean metric, therefore they define a metric $g$ on $\mathbb{C}^*$ which is locally Euclidean. This  metric $g$ is defined by $g_z(u,u)=\frac{u\bar{u}}{z\bar{z}}$.
Note the triple equality ${\rm Sim}(\mathbb{C}^*,g)={\rm Iso}^+(\mathbb{C}^*,g)={\rm Aut}(\mathbb{C}^*,J)$.

A similar discussion holds for the family of elliptic curves. In this case, instead of the mapping $\mathrm{exp}:\mathbb{C}\to \mathbb{C}^*=\mathbb{C}/\mathbb{Z}$, one considers a mapping $\mathrm{exp}_{\omega}:\mathbb{C}\to\mathbb{C}/\mathrm{Gr}(1,\omega)$ where $\mathrm{Gr}(1,\omega)$ is the group of translations generated by the complex numbers $1$ and $\omega$.

The group $G={\rm PSL}(2,\mathbb{R})$ acts  transitively on $\mathbb{C}_+$ by
$z\in \mathbb{C}_+ \mapsto \frac{az+b}{cz+d}\in \mathbb{C}_+$. The stabilizer of $i$ is the compact group 
${\rm Stab}(i)={\rm PSO}(2)$. The Killing form $k(H,K)={\rm Trace}(HK)$ on the Lie algebra  ${\rm Lie}(G)$ induces on 
$\mathbb{C}_+=G/{\rm Stab}(i)$ a Riemannian metric which is conformal with\index{J@$J$-conformal metric}\index{J@$J$-compatible metric} $J=m_i$ (``multiplication by $i$").
Again a triple equality of groups holds.

The interpretation of $z=a+bi\in \mathbb{C}_+$ as a ring homomorphism $f_z$ and as an ideal $((X-a)^2+b^2))$ (see \S \ref{s:Riemann-shadow}) gives an elementary insight
for the above explanation by Lie group theory.

Given $z\not=z'$, define the curve $L_{z,z'}=z_t,\, t\in ]a^-(z,z'),a^+(z,z')[$ by the following:
put $z=a+bi,z'=a'+b'i$, then $z,z'$ are the roots in $\mathbb{C}_+$ of the polynomials $P_z=x^2-2ax+a^2+b^2,\,P_{z'}=x^2-2a'x+a'^2+b'^2$. Let $]a^-(z,z'),a^+(z,z')[$ be the maximal open interval containing the interval $[0,1]$ such that for all $t\in ]a^-(z,z'),a^+(z,z')[$
the convex combinations of polynomials $Q_t=tP_z+(1-t)P_{z'}$
have no real roots. Define $L_{z,z'}$ as the path of roots in $\mathbb{C}_+$ of the polynomials $Q_t$.

If $a=a'$, then the curve $L_{z,z'}$ is the real half-line perpendicular to the real axis through  $z$ and $z'$. If $a\not=a'$, then the curve $L_{z,z'}$ is the semi-circle through $z,z'$ with center on the real axis, hence perpendicular to the real axis.

The quantity $D(z,z')=\frac{1}{2}\log(\lambda_{a^-(z,z'),a^+(z,z'),0,1})$ defines a metric on $\mathbb{C}_+$, which is ${\rm PSL}(2,\mathbb{R})$-invariant. The curves $L_{z,z'}$ are length minimizing, hence geodesics.
More precisely, the expansion of $D(z,z+u)^2$ at the point $z$ gives
the leading second order term $g_z(u,u)=\frac{u\bar{u}}{{\rm Im}(z)^2}$ and defines a Riemannian metric. From this, we can  recover the hyperbolic metric on $\mathbb{C}_+$, see \cite[\S 8.3]{ACP2012}.\qed
\end{proof}

In conclusion, from the formulae $\omega(u,v)=g(u,Jv)$ and $\omega(u,Jv)=g(u,v)$, each of the volume form $\omega$ and the metric $g$ is determined from the other one and from the almost complex structure $J$. The same formulae determine $J$ from $\omega$ and $g$ because the 2-forms $\omega$ and $g$ are non-degenerate.

In fact, more generally, the motto is that in the triple of structures on surfaces (volume form $\omega$, Riemannian metric $g$, almost complex structure $J$) on $\mathbb{R}^{2n}$, two elements determine the third one. The result is that there are fruitful interactions between symplectic geometry, Riemannian geometry and complex geometry.  This is expressed in a spectacular manner in Gromov's 1985 paper in which he introduced pseudo-holomorphic curves \cite{Gromov}. This introduced also notions of ``positivity" and of ``compactness" in the three geometrical settings. One result is that the stabilizers of $w$ and $g$ in $\mathrm{GL}(2n,\mathbb{R})$, if they are related by a $J$ as above, coincide and form a maximal compact subgroup.

\subsection{Spherical geometry}\label{ss:spherical}

The differential sphere is present as the Riemann sphere $R$. Spherical geometry\index{spherical geometry} is still missing. The group ${\rm PGL}(2,\mathbb{C})$ acts
transitively on $(R,J_R)$. The stabilizer ${\rm Stab}(\infty)$
is the group $z\mapsto \lambda z+ t, \lambda\in \mathbb{C}^*,\,t\in \mathbb{C}$. Again by Lie theory, each maximal compact subgroup $U$
of  ${\rm Aut}(R,J_R)$ defines a spherical metric $g_U$ on $R$ that is conformal with $J_R$.

A perhaps more elementary approach is the following:
 Let $I_R$ be the space of fixed-point free involutions that preserve the shadow function\index{shadow function} $\lambda:R^4\setminus {\rm Diag}\to \mathbb{C}$. For two distinct involutions $A,B$,
the composition $C=A\circ B$ has two fixed points $p,q$. Let $C_p,C_q$ be the determinants of the differentials $D_pC,D_qC$ at $p,q$. Define the distance 
$D(A,B)$ by 
\[D(A,B)=\frac{1}{2}\log((C_p+C_q)/2).\]
The space $(I_R,D)$ is a model for hyperbolic $3$-space.\index{hyperbolic 3-space!model}\index{model!hyperbolic 3-space} The infinitesimal version $g_{I_R}$ of the metric $D$ is a Riemannian metric on $I_R$. For each $A\in I_R$, the half-rays from $A$ define a diffeomorphism
$\phi_A$ from the infinitesimal sphere $S_A$ with center $A$ to $R$. The map $\phi_A$ is conformal and carries the spherical metric of $S_A$ to a metric
$g_A$ on $R$ which is conformal with\index{J@$J$-conformal metric}\index{J@$J$-compatible metric} $J_R$. Think of $S_A$ as the unit sphere in $(T_AI_R,g_{I_R,A})$.
The group ${\rm Iso}(R,g_A)$ is a maximal compact subgroup  in ${\rm Aut}(R,J_R)={\rm PGL}(2,\mathbb{C})$. Moreover the equality ${\rm Sim}(R,g_A)={\rm Iso}(R,g_A)$ holds.

We think of the Riemann sphere\index{Riemann sphere!model}\index{model!Riemann sphere} $R$ as the space of $1$-dimensional vector subspaces in $V=\mathbb{C}^2$ (\S \ref{s:Riemann-shadow}). A positive Hermitian form\index{positive Hermitian form}\index{Hermitian form!positive} on
$V$ is a real bilinear map
$h:V\times V\to \mathbb{C}$ satisfying 
for $u,v\in V,\lambda \in \mathbb{C}$

\begin{itemize}

\item $h(\lambda u,v)=\lambda h(u,v)$,

\item $h(u, v)=\bar{h(v,u)}$,

\item $h(u,u)>0$ for $u\not= 0$.

\end{itemize}

  A positive Hermitian form on
$V$ defines by $L\in R \mapsto L^\perp\in R$
a fixed-point free involution $A_h$ of the Riemann sphere  $R$ that preserves the shadow function\index{shadow function} $\lambda$. Here,  associated with the complex line
$L=tu,t\in \mathbb{C}, u\in V,u\not= 0,$ is the complex line $L^\perp=\{v\in V \mid h(v,u)=0\}$. Two positive Hermitian forms that differ by a positive factor give the same involution. 

Moreover the stabilizer in ${\rm Aut}(R)={\rm PGL}(2,\mathbb{C})$ is a maximal compact subgroup in ${\rm Aut}(R)$. Similar forms have the same stabilizers and all maximal compact subgroups correspond to a unique form.

If two lines $L,L'$ in $R$ are neither perpendicular nor equal, they 
give by $L,L',L^\perp,L'^\perp$ a quadruple of complex lines. The 
expression
$D(L,L')=\lambda_{L,L',L'^\perp,L^\perp}$ defines a function  
on $R\times R$ with values in $\mathbb{R}_+\cup \{+\infty \}$. The 
preimage of $0$ is the diagonal, the set of pairs $(L,L')$ with $L'=L$, and 
the 
preimage of $+\infty$ the set  of pairs with $L'=L^\perp$. Its infinitesimal 
version along the diagonal, i.e., at $L\in R$ the 
Hessian of $L'\mapsto D(L,L')$ at $L'=L$, defines a Riemannian metric on 
$R$ which  is similar, even isometric, to the 
spherical geometry of Gaussian curvature $+1$ in dimension 2.

\subsection{Models for hyperbolic 3-space and metrics of constant positive curvature on the sphere} \label{s:Model-spaces}

Even though the 3-dimensional hyperbolic space\index{hyperbolic 3-space!model}\index{model!hyperbolic 3-space} $\mathbb{H}^3$ does not admit a complex structure, its automorphism group is that of a complex manifold
and is itself a complex Lie group, since we have
$$\mathrm{Iso}^+(\mathbb{H}^3)
= \mathrm{Aut}_{\mathrm{hol}}(\mathbb{P}^1(\mathbb{C}))={\rm PGL}(2,\mathbb{C}).
$$

We have the following four models of the hyperbolic 3-space\index{hyperbolic 3-space!model}\index{model!hyperbolic 3-space}
$\mathbb{H}^3$  that arise from the complex geometry of surfaces; 
the first two models are geometric (they are defined in terms of group actions) whereas the other two use algebra:

\begin{enumerate}
\item
The space of fixed-point free shadow-preserving involutions of $R$.
\item  The space of Riemannian metrics on $R$ that are conformal with\index{J@$J$-conformal metric}\index{J@$J$-compatible metric} $J_R$ and similar to the spherical metric.
\item The space of similarity classes of positive Hermitian forms on $\mathbb{C}^2$.
\item
The space of maximal compact subgroups in ${\rm PGL}(2,\mathbb{C})$ (which is the automorphism group of the oriented hyperbolic 3-space). 
\end{enumerate}

The first model already appeared in \S \ref{ss:spherical}, where we constructed the spherical geometry of the Riemann sphere. This is the sphere we called $I_R$ there. To get another point of view on this model,\index{hyperbolic 3-space!model}\index{model!hyperbolic 3-space} consider first the well-known Poincar\'e model of the hyperbolic space $\mathbb{H}^3$ as a unit ball sitting in $3$-space, with boundary the Riemann sphere $R$. For each point $p$ in $\mathbb{H}^3$, an involution of $R$ is defined by assigning to each point in $R$ the intersection with the sphere of the line passing through this point and $p$. To see that we get a model of  $\mathbb{H}^3$, use the Hilbert metric model of this space.

For the second model, we also consider the Poincar\'e unit ball model of the hyperbolic space, with boundary the Riemann sphere $R$.  For each point of $\mathbb{H}^3$, we take the diffeomorphism which sends the infinitesimal round sphere (or a sphere in the tangent space) at that point to the boundary at infinity of the space, using the geodesic rays that start at this point. This is a conformal mapping. (One may prove this by trigonometry.) The mapping sends the conformal metric of the infinitesimal sphere to a metric on the sphere sitting at the boundary of the hyperbolic space. This construction commutes with the isometries of $\mathbb{H}^3$. It gives a 3-dimensional space of metrics in the same conformal class on $\mathbb{P}^1(\mathbb{C})$.  Thus, the hyperbolic $3$-space $\mathbb{H}^3$ appears as a space of metrics of curvature $+1$ on $\mathbb{P}^1(\mathbb{C})$.  In other words,  $\mathbb{H}^3$ appears as a space of special metrics, i.e., as a moduli space, namely, the space of Riemannian metrics $g$ of Gaussian curvature $+1$ on $R$ that are moreover in the conformal class of $J_R$. Alternatively, the hyperbolic $3$-space  $\mathbb{H}^3$ appears as the space of volume forms
$\omega$ on $R$ such that the Riemannian metric defined by $g(u,v)=\omega(u,J_Rv)$ is a 
metric of Gaussian curvature $+1$.

For the third model, recall that a positive Hermitian form on $\mathbb{C}^2$ is a Riemannian metric on $\mathbb{R}^4$ such that one can measure the angle between two lines in this space. This model was explained in \S \ref{ss:Hermitian} above.

 The fourth model is equivalent to the 3rd because the group $\mathrm{PGL}(2,\mathbb{C})$ acts on the space in question with stabilizer the automorphism group of the oriented hyperbolic 3-space, that is, the group $\mathbb{PGL}(2,\mathbb{C})$ quotiented by its maximal compact subgroup.

An explicit way of interconnecting the above models for the hyperbolic three-space $\mathbb{H}^3$ is as follows.

Think of $S^2$ as the unit sphere in $\mathbb{R}^3$ with its induced spherical metric together with the corresponding conformal structure, the space $\mathbb{R}\oplus\mathbb{C}$, for the Euclidean norm
$||(t,z)||^2=t^2+z\bar{z}$.

The stereographic projection with pole $(0,-1)$
maps by $$(0,z)\mapsto (\frac{1-z\bar{z}}{1+z\bar{z}},\frac{2z}{1+z\bar{z}})$$ the factor $\{0\}\times \mathbb{C}$ conformally to $S^2$. This projection extends to an (oriented) diffeomorphism 
$$St:R=\mathbb{C}\cup \{\infty\} \to S^2$$
 and equips by pullback the Riemann sphere with a Riemannian metric of Gaussian curvature $+1$. 

Given an ordered triple of distinct points $(a,b,c)$ in $R$, the shadow function gives by 
$$p\in R\setminus\{c\}\mapsto Sh_{abc}(p)=(0,cr(a,p,b,c))\in 
\{0\}\times\mathbb{C}\cup\{\infty\}$$
 a mapping such that the composition $St_{abc}\circ Sh$ extends to a conformal diffeomorphism  $\mu_{abc}:R \to S^2$.

By pullback we obtain a family $g_{abc}$ of Riemannian metrics of curvature $+1$ on $R$. Two such metrics $g_{abc}$ and $g_{a'b'c'}$ are related by an oriented isometry if and only if 
the composition $\mu_{abc}\circ \mu_{a'b'c'}^{-1}$ is an oriented isometry of $S^2$.

Now remember that the group of holomorphic automorphisms  of $R$ is isomorphic to ${\rm PGL}(2,\mathbb{C})$ and acts simply transitively on triples of distinct points. So each triple $(a,b,c)$ is obtained in a unique way as $(a,b,c)=M_{abc}(a_0,b_0,c_0)$ from a chosen triple $(a_0,b_0,c_0)$ by applying a well-defined element $M_{abc}\in {\rm PGL}(2,\mathbb{C})$. The above condition that the metrics $g_{abc}$ and $g_{a'b'c'}$ are related by an oriented isometry
translates into the fact that $M_{a',b',c'}\circ M_{a,b,c}^{-1}\in {\rm PU}(2,\mathbb{C})$ and so it proves that the
hyperbolic 3-space $\mathbb{H}^3$ is parametrized by the symmetric space ${\rm PGL}(2,\mathbb{C})/{\rm PU}(2,\mathbb{C})$.

From the above discussion, some basic ingredients of the geometry of $\mathbb{H}^3$ can be seen, such as the following:

\begin{itemize}
\item Ideal points are the points of $R$.

\item Given $a\not=c $ in $R$, the points on the line $L_{ac}$ from $a$ to $c$
are the metrics $\mu_{a,b,c}$ where $b$ varies over $R\setminus \{a,c\}$. 

\item For $p\not= q \in H^3$, represented respectively by metrics $\mu_p,\mu_q$ and volume forms $\omega_p, \omega_q$ on $R$, these points lie on $L_{ac}$ where
the points $a,c$ are the extrema of $\frac{\omega_p}{\omega_q}$.

\end{itemize}

 \subsection{Models for the hyperbolic plane}

The hyperbolic plane\index{hyperbolic plane}\index{model!hyperbolic plane} $\mathbb{H}^2$, like the 3-dimensional hyperbolic space $\mathbb{H}^3$, has also several interpretations in terms of the complex geometry of surfaces. 
A standard interpretation  of $\mathbb{H}^2$ is that it is the space of marked elliptic curves.
We propose three other incarnations of this plane:

\medskip

1. The hyperbolic plane\index{model!hyperbolic plane}\index{hyperbolic plane} $\mathbb{H}^2$ is the space of ideals $\mathbb{H}_I$
with its geometry naturally given by a family of lines $L_{((X-a)^2+b^2),((X-a')^2+b'^2)}$. Indeed, the plane $\mathbb{H}_I$ comes with a natural complex coordinate $z:\mathbb{H}_I\to \mathbb{C}_+=\{a+ib, \ b>0\}$, defined by \[z((X-a)^2+b^2))=a+bi.\]
The coordinate $z$ is a bijection with
 \[Z:a+bi\mapsto ((X-a)^2+b^2)\] as inverse. Thus,  $Z(a+bi)$ is the interpretation of the number $a+bi$ as an ideal. See the details in \S \ref{s:compatible} above.

\medskip

2. The hyperbolic plane\index{model!hyperbolic plane}\index{hyperbolic plane} is the space $\mathbb{H}_J$ of\index{almost complex structure}  almost complex structures $J$ on $\mathbb{R}^2$ (that is, endomorphisms of $\mathbb{R}^2$ satisfying $J^2=-\mathrm{Id}$) equipped with coordinates $(x,y)$, and a volume form $\omega=dx\wedge dy$, and where $J$ is calibrated by the inequality $\omega(u,Ju)>0$ for all $u\in\mathbb{R}^2$.

Another way of seeing this is the following. We start by defining a map: $\mathrm{Fix}:H_J\to \mathbb{C}_+$, so that a point in the space $H_J$ has a coordinate $z(J)={\rm Fix}(J)\in \mathbb{C}_+$ where ${\rm Fix}(J)$
is the fixed point of $J$ for its homographic action on $\mathbb{C}_+$.

 More explicitly, the matrix of $J$ in the basis $(1,i)$ is of the form $(\begin{smallmatrix} a/b & *\\ 1/b &-a/b\\  \end{smallmatrix})$ where $h\in \mathbb{R}$ and $b>0$ (use the fact that the trace is zero), and $*$ is obtained from the fact that the determinant of the matrix is equal to 1. Such a matrix acts on the upper half-plane with exactly one fixed point, and we assign to $J$ this fixed point. It is possible to write explicitly the coordinates of this fixed point in terms of $a$ and $b$.

 Consider the map
\[z((\begin{smallmatrix} h & *\\ k &-h\\  \end{smallmatrix}))=\frac{h}{k}+\frac{i}{k}\]
where $k>0$ and $*=\frac{1+h^2}{-k}$ is a complex coordinate. This is again a bijection, with inverse 
\[Z:a+bi\mapsto (\begin{smallmatrix} a/b & *\\ 1/b &-a/b\\  \end{smallmatrix}), \, b>0,\] and there is an interpretation of the number $a+bi,\, b>0$ as a linear complex structure.

\medskip

The incarnation $\mathbb{H}_J$ is very helpful for the construction of hyperbolic trigonometry. Only High-school/Freshman knowledge is required in this construction.

This incarnation $\mathbb{H}_J$ is also very helpful for the study of the space of complex structures\index{space of complex structures}\index{complex structure!space of} $\mathbb{J}(TS)$ on a given surface. 
As a first example, the surface $\mathbb{H}_J$ itself
has a complex structure. A tangent vector at $J\in \mathbb{H}_J$ is an endomorphism $K$ of $\mathbb{R}^2$ such that the equation
$$(J+K)^2=J^2+J\circ K+ K\circ J+ K^2=-{\rm Id}_{\mathbb{R}^2}$$ holds at first order in $K$. Therefore the tangent space $T_J\mathbb{H}_J$ at $J$ of $\mathbb{H}_J$ is the vector space of endomorphisms $K$ that anti-commute with $J$. The map $K\in T_J\mathbb{H}_J	\mapsto 
J\circ K \in T_J\mathbb{H}_J $ is a canonical complex structure $J_{\mathbb{H}_J}$ on $\mathbb{H}_J$ which induces a complex structure on the infinite-dimensional manifold $\mathbb{J}(TS)$.
This space $\mathbb{J}(TS)$ is canonically path-connected. Indeed two structures $J_0,J_1$ are connected by the path that for $p\in S$ is a geodesic in the space of linear complex structures on the oriented real vector space $T_pS$. The fact that the space $\mathbb{J}(TS)$ is contractible now follows, since these canonical paths depend continuously upon endpoints.

\medskip

3. The bijective coordinates $z$  on 
$\mathbb{H}_I,\mathbb{H}_J$ transport the geometries to a common geometry on 
$\mathbb{C}_+$. The coordinate $z:\mathbb{H}_J\to \mathbb{C}_+$ is $(J_{\mathbb{H}_J},m_i)$-holomorphic. The model 
$\mathbb{C}_+$ with its conformal structure and hyperbolic metric, equipped with its action by the modular group 
${\rm PSL}(2,\mathbb{R})$, is
a very appreciated model of the hyperbolic plane.

\section{Uniformisation}\label{s:Uni}
\subsection{Riemann's uniformisation of simply connected domains}

\begin{theorem}[The Riemann Mapping Theorem \cite{Riemann1851}]

  Let $\Omega$ be a non-empty open connected and simply connected subset of the complex plane which is not the entire plane.   Let $z_0$ be a point in $\Omega$ and $v$ a nonzero vector at $z_0$.  Then, there is a unique holomorphic bijection $f$ from $\Omega$ to the unit disc $\mathbb{D}$ such that $f(z_0)=0$ and $df(v)$ is a complex number which is real and positive.

\end{theorem}

Note that the inverse of a holomorphic bijection is also holomorphic (think of a holomorphic function as an angle-preserving map). Therefore the map $f$ is biholomorphic.

\begin{proof}[Outline, after Riemann's proof, see \cite{Julia}] This proof works in the case where $\Omega$ 
  is an open subset of the plane bounded by a Jordan curve (a simple closed curve). We outline this proof.

We wish to find a biholomorphic mapping $f:\Omega\to \mathbb{D}$ which sends $\gamma=\partial \Omega$ to $\mathbb{S}^1=\partial \mathbb{D}$.

Suppose that such a mapping $f$ exists. Separating the real and imaginary parts, we set

\[f(z)=u(x,y)+iv(x,y).\]
 Applying a translation, we may assume that $z_0=0$, therefore,   $f(0)=0$. Applying  a rotation, we may also assume that the vector of coordinate $(0,1)$ at $0$ is sent by $f$ to a vector pointing towards the positive real numbers.
 
 Such a map $f$, it is exists, is unique, since any holomorphic automorphism of the disc which fixes the origin is a rotation, and a rotation that fixes a direction is the identity.

Since $f$ is a bijection between $\Omega$ and the unit disc, the point $z=0$ is the unique solution of the equation $f(z)=0$. Furthermore, the differential (or the complex derivative) of $f$ at $0$ is nonzero, otherwise $f$ would be non-injective in the neighborhood of zero. Thus, the function $f(z)/z$ is holomorphic on $\Omega$. Since it does not take the value $0$, a branch of its complex logarithm can be defined. Hence, there exists a holomorphic function $H(z)$ on $\Omega$ such that 
\[\frac{f(z)}{z}=e^{H(z)}.\]
In polar coordinates we can write $z=r(x,y)e^{i\theta (x,y)}$, which gives

\[f(z)=ze^{H(z)}= r(x,y)e^{i\theta(x,y)} e^{H(z)}.\]

Separating the real and imaginary parts of $H$ as $H(z)=P(x,y)+iQ(x,y)$, we get

\[\log f(z) = \log r(x,y)+i\theta (x,y) +P(x,y)+iQ(x,y),\]
which we can write as

\[\log f(z) = \log r(x,y) +P(x,y)+i(\theta (x,y) + Q(x,y)).\]

Notice that the function $\log \vert f\vert$ is identically zero on $\partial \Omega$, since points $z$ on this curve satisfy $\vert f(z)\vert =1$.

Furthermore, the function $\log \vert f\vert$ is the real part of the holomorphic function $\log f $, therefore it is harmonic\index{harmonic function} at every point of $\Omega$.

Now given a harmonic function\index{harmonic function} defined on the simply connected subset $\Omega$ of the plane, there exists a unique holomorphic function which has this function as a real part.  (The usual way to find this function is to use Cauchy's theory of path integrals.)

Reasoning backwards, the problem of mapping $\Omega$ to the unit disc can be solved if we can find a function $P$ which is harmonic\index{harmonic function} on $\Omega$ and which takes the values $-\log r$ on $\gamma=\partial \Omega$.

In this way, Riemann reduced the problem of finding a holomorphic mapping from $\Omega$ to $\mathbb{D}$ to a problem of finding a harmonic function\index{harmonic function} on $\Omega$ with a prescribed value on $\partial \Omega$. To find such a function, he appealed to the so-called Dirichlet method,\index{Dirichlet method} which consists of considering the functional $I$ defined on the space of differentiable functions $P$ on $\Omega$ by:
\[I(P)= \int_{\Omega}
(\frac{\partial P}{\partial x}^2 + \frac{\partial P}{\partial y}^2
)dxdy
.\]
A function which realizes this infimum is harmonic.\qed

\end{proof}

 \begin{remarks}

    1.--- Riemann also used the Dirichlet principle\index{Dirichlet principle} in his existence proof of meromorphic functions on general Riemann surfaces, obtaining them from their real parts (harmonic functions). It was realized later that the Riemann Mapping Theorem and the existence of meromorphic functions are equivalent.

 \medskip
 
2.--- A modern proof of the Riemann Mapping Theorem  goes as follows (see  Thurston's notes\cite{Thurston83}): Given a simply connected open strict subset $\Omega$ of the complex plane and a point $z_0$ in $\Omega$, consider the set of functions 
\[\mathcal{F}=\{f:\Omega\to \mathbb{D},   \, \mathrm{  holomorphic, injective, }\  f(z_0)=0 \ \mathrm{ and } \ 
f'(z_0)>0\},\] equipped with the topology of uniform convergence on compact sets. The desired conformal mapping is the mapping $f$ in this family that maximizes the modulus of the derivative $\vert f'(z_0)\vert$.  The existence of such a mapping uses the fact that for any family of holomorphic functions which is uniformly bounded on compact sets, every sequence has a subsequence which converges uniformly on compact sets to a holomorphic function  (this is Montel's theorem).

 \end{remarks}
 
We shall call a \emph{Riemann mapping}\index{Riemann mapping} a biholomorphic mapping that sends  the unit disc to a simply-connected subset of the plane (which we sometimes call a domain). This is the inverse of the mapping $f:\Omega\to\mathbb{D}$ that we just considered.

For a general domain $\Omega$, the Riemann mapping\index{Riemann mapping} is not explicit. In some special cases, there are explicit formulae, e.g., when $\Omega$ is the upper half-plane, or a band in the plane (a region bounded by two parallel lines), or an angular sector (a connected component of the union of two intersecting Euclidean lines in the plane), or a domain bounded by two arcs of circle, or a sector of a disc, or a rectangle in the Euclidean plane, or a regular convex polygon, or a regular star polygon, or the interior of an ellipse or a parabola, and a few other cases. These examples and others are studied in the book by G. Julia \cite{Julia}.

There are also results that concern the boundary behavior of the Riemann mapping.\index{Riemann mapping!boundary behavior} The first name associated with such a boundary theory is Carath\'eodory. He  proved that if the boundary of the domain $\Omega$  is a Jordan curve, then the Riemann mapping\index{Riemann mapping!extension} extends continuously to a homeomorphism from the unit circle to the boundary of the domain. He also proved that in the general case the Riemann mapping extends continuously to the unit circle if and only if  the boundary of the domain is locally connected (but this extension is generally not a homeomorphism). There are also measure-theoretic results and questions. For instance, in the case where the domain is bounded by a rectifiable curve $\gamma$,  an extension of the Riemann mapping\index{Riemann mapping!extension} exists and sends any subset of measure zero (resp. of positive measure) of $\gamma$  to a subset of measure zero (resp. of positive measure) of the circle. This was proved by Luzin and Privaloff \cite{LP} and by  F. and M. Riesz \cite{Riesz}.  One can also mention here the following result of Fatou \cite{Fatou1906}: If a function $f$  defined on the unit disc is holomorphic and bounded, then, at almost every point $t$ on the unit circle $\vert z\vert =1$, the value of $f$ at a sequence of points converging to that point on a path which is not tangential to the unit circle has a limit. For an exposition of  related results, see Chapter II of Lavrentieff's book \cite{Lavrentieff-Sur}. These results are useful in the modern developments of complex dynamics, in particular in the study of the conformal representation of the complement of the Julia set of a polynomial.

\subsection{Uniformization of multiply connected domains}

It is natural to search for\index{uniformization!multiply connected domains} generalizations of  the Riemann Mapping Theorem for multiply connected open subsets of the plane. The first natural question is whether there exist canonical domains onto which they can be mapped biholomorphically, in analogy with the fact that the disc and the plane are canonical domains for simply connected domains. It is known that two open subsets of the complex plane that have the same connectivity are not necessarily conformally equivalent. For instance, two circular annuli (that is, annuli bounded by two concentric circles) are conformally equivalent if and only if they have the same modulus, that is, if and only if the ratio of their outer radius to their inner radius is the same. Thus, there are no canonical domains for multiply connected open  subsets of the plane. However, there are figures which may be considered as ``standard" domains for such subsets, 
and we now discuss some of them.

One of the oldest works on this question was done by Koebe, who studied in his papers  \cite{Kobe1918}, \cite{ Kobe1918a} and others conformal mappings of multiply connected open subsets of the plane onto \emph{circular domains},\index{circular domain} that is, multiply connected domains whose boundary components are all circles (which may be reduced to points). He proved that every finitely connected domain in the plane is conformally equivalent to a circular domain. The question of whether there is an analogous result for open subsets of the plane with infinitely many boundary components is still open. This question is known under the name  \emph{Kreisnormierungsproblem}, or the \emph{Koebe uniformization conjecture}. Koebe\index{Koebe uniformization conjecture}\index{uniformization conjecture!Koebe}  formulated it in his paper \cite{Kobe1908}. The reader interested in this question may refer to the paper \cite{Bowers} by Ph. Bowers in which the author surveys several developments of this conjecture, including related works on circle packings by Koebe, Thurston and others. See also \ref{Br1992}.

 H.  Gr\"otzsch, in his paper \cite{Groetzsch1928},  works with two classes of domains which were considered by Koebe:
\begin{enumerate}
 \item  \emph{Annuli with circular slits}.\index{annulus with circular slits} These are circular annuli from which a certain number ($\geq 0$) of circular arcs (which may be reduced to points) centered at the center of the annulus, have been deleted. 
 \item \emph{Annuli with radial slits}.\index{annulus with radial slits}  These are circular annuli from which a certain number ($\geq 0$) of radial arcs (which may be reduced to points) have been removed. Here, a radial arc is a Euclidean segment in the annulus which, when extended, passes through the center of the annulus. 
 \end{enumerate}

%
%

The unit disc slit along an interval of the form $[0,r]$ ($r<1$)  was employed by Gr\"otzsch in  \cite{Groetzsch1928} as a standard domain and is known  under the name \emph{Gr\"otzsch domain}.\index{Gr\"otzsch domain}
    Teichm\"uller, in his paper  \cite{T13}, calls a \emph{Gr\"otzsch extremal region}\index{Gr\"otzsch extremal region} the complement of the unit disc in the Riemann sphere $\mathbb{C}\cup\{\infty\}$  cut along a segment of the real axis joining a point $P>0$ to the point $\infty$.
    
     Another ``standard model" for doubly connected domains is the Riemann sphere slit along  two intervals of the form $[-r_1,0]$ and $[r_2,\infty)$ where $r_1$ and $r_2$ are positive numbers. Such a domain is called  a \emph{Teichm\"uller extremal domain} by Lehto and Virtanen \cite[p. 52]{LV73}. In his paper \cite{T13},  Teichm\"uller uses these domains in the study of extremal properties of conformal annuli. In the same paper, 
he  works with other ``standard" domains, e.g.,  the circular annulus $1<\vert z\vert <P_2$ cut along the segment joining $z=P_1$ to $z=P_2$, where $P_1$ and $P_2$ are points on the real axis satisfying $1<P_1<P_2$ \cite[\S 2.4]{T13}; see the discussion in Ahlfors' book \cite[p. 76]{AhC}.

\subsection{Uniformization of simply connected Riemann surfaces}\label{s:unif-simply}
It is natural to consider conformal representations of surfaces that are not subsets of the complex plane.
A wide generalization of the Riemann Mapping Theorem is the following, proved independently by Poincar\'e and Koebe:  
\begin{theorem}[Uniformization theorem] \label{t:uniformization}
Every simply connected Riemann surface can be mapped\index{uniformization theorem} biholomorphically either to 
 
 \begin{enumerate}
 
\item the Riemann sphere (elliptic case);\index{elliptic type}\index{type!elliptic}
 
\item the complex plane (parabolic case);\index{parabolic type}\index{type!parabolic}
 
\item the unit disc (hyperbolic case).\index{hyperbolic type}\index{type!hyperbolic}
\end{enumerate}
  
  \end{theorem}
  
 In each of these cases, the conformal biholomorphism from the model space onto the open Riemann surface is usually called a \emph{uniformizing map}.\index{uniformizing map} Such a map associated to a Riemann surface is well defined up to composition by a conformal homeomorphism of the source. This implies that the mapping is determined in the first case by the images of three distinct points, in the second case by the image of two distinct points, and in the third case by the image of one point and a direction at that point.
 
 If the Riemann surface is a simply connected open subset $\Omega$ of the complex plane which is not the entire plane, then it is necessarily of hyperbolic type (this is the Riemann Mapping Theorem). Uniformizing maps for such surfaces have been particularly studied. 
 
 For a proof of the uniformization theorem\index{uniformization theorem} based on a historical point of view, see the collective book \cite{StGervais}. For another point of view on uniformization (``uniformization by energy"), see \cite[Chapter 10]{A2021}.
  
   There are several open questions on the behavior of the uniformizing maps of simply connected surfaces equipped with Riemannian metrics.  Lavrentieff, in his paper \cite{Lavrentieff1935}, adresses the question of the Lipschitz behavior of  a uniformizing map $f:\mathbb{D}\to \Omega$. More precisely, he asks for conditions on the surface $\Omega$ so that for any pair of points $x,x' \in \mathbb{D}$, the ratios  $\frac{d_\mathbb{D}(x,x')}{d_S(f(x),f(x'))}$ and $\frac{d_S(f(x),f(x'))}{d_\mathbb{D}(x,x')}$ are bounded or unbounded (the distance in the disc being the Euclidean distance).
To the best of our knowledge, this problem is still not settled.

\section{Branched coverings}\label{s:Branched}
\subsection{The Riemann--Hurwitz formula}\label{s:Hurwitz}
The notion of branched covering\index{branched covering} between surfaces was conceived by Riemann.\index{Riemann--Hurwitz formula} It will be used at several places in the rest of this chapter. We shall review in particular a theorem of Thurston on branched coverings of the sphere in relation with the question of the topological characterization of rational maps\index{rational map!topological characterization} (\S \ref{s:Thurston}). Branched coverings also appear in the study of the type problem\index{type problem}\index{problem!type} (\S \ref{s:Nevanlinna}) and in the theory of dessins d'enfants (\S \ref{s:dessins}).

  A map $f: S_1\to S_2$ between two topological oriented surfaces is said to be a \emph{branched covering}\index{branched covering} if it satisfies the following:
  \begin{enumerate}
  \item  for some discrete subset $\{a_i\}_{i\in I}\subset S_2$,  the map $f$ restricted to  $f^{-1}\setminus \{a_i\}$  is a covering map in the usual sense;
  \item for each point  $x\in S_1$ which is the inverse image by $f$ of a point $a_i\in S_2$ for some $i\in I$, there exists an orientation-preserving homeomorphism $\phi: U\to D$  between an open  neighborhood $U$ of $x$ and an open neighborhood $D$ of $0$ in $\mathbb{C}$ satisfying  $\phi (x)= 0$, and an orientation-preserving homeomorphism  $\psi: V\to D$  between an open neighborhood $V$ of $a_i$ and $D$ satisfying  $\phi (a_i)= 0$, such that the mapping $\psi\circ f\circ\psi^{-1}$ defined on $D$ coincides with the restriction of the map $z\mapsto z^k$ from the complex plane to itself, for some integer $k\geq 1$.
  \end{enumerate}
  
 The integer $k$ in the above statement depends only on the point $x$ and is called the \emph{branching order} of $f$ at $x$.  
 If $k=1$, then $f$ is a local homeomorphism at $x$. If $k>1$, then we say that $x$ is a \emph{critical point}\index{critical point} of the covering and that $f$ is \emph{branched} at $x$. A point $a_i$ which is the image of a critical point  is called a \emph{critical value}.\index{critical value} The \emph{degree} of the branched covering\index{branched covering} $f$ is the usual degree of the covering induced by $f$ on the surface $f^{-1}\setminus \{a_i\}$, that is, the cardinality of the pre-image by $f$ of an arbitrary point in $S_2\setminus\{a_i\}$. This degree is also equal to the sum of the branching orders of the points in the pre-image of an arbitrary point $a_i$. The degree of a branched covering\index{branched covering} is 1 if and only if this covering is a homeomorphism.

 A meromorphic function on a Riemann surface defines a branched covering from this surface to the Riemann sphere. (Here, the adjective ``meromorphic" can be replaced by ``holomorphic" since the value $\infty$ is, from the complex structure, a point like any other point on the Riemann sphere.) Every connected compact Riemann surface $S$ admits a branched covering\index{branched covering} $f:S\to R$ over the Riemann sphere $R$, with the  conformal structure on $S$ induced by lifting by $f$ the conformal structure of the sphere. (This is one form of the so-called Riemann's Existence Theorem.)\index{Riemann existence theorem}
 Special cases of branched coverings that we shall discuss below are the Belyi maps\ref{Bel1980}:\index{Belyi map} holomorphic maps from a compact Riemann surface to the Riemann sphere with at most $3$ critical values (see \S \ref{s:dessins}).

 The question of the topological characterization of analytic functions among the maps from a surface onto the Riemann sphere was known under the name \emph{Brouwer problem} and it has many aspects, see the comments in \cite[p. xv]{Stoilov}. A theorem of Sto\"\i lov says that if $S$ is a topological surface, then a continuous mapping from $S$ to the Riemann sphere is a branched covering\index{branched covering} if and only if $f$ is open (i.e., images of open sets are open) and discrete (i.e., inverse images of points are discrete subsets of $S$, that is, each point is isolated in $S$). Furthermore, for each continuous mapping satisfying this property, there exists a unique complex structure on $S$ such that this mapping  is conformal \cite[Chap. V]{Stoilov}.

 Ren\'e Thom was interested in such questions. In his paper \cite{Thom}, he solves the following problem:
 
  \emph{Given $n$ complex numbers (which are not necessarily distinct), does there exist a polynomial of degree $n+1$ with these numbers as critical values?}
  
  Thom also considers the real analogue.
  At the end of his paper, he formulates the general problem:
  
 \emph{Given a $C^{\infty}$ function $f:\mathbb{R}^2\to\mathbb{R}^2$ with only isolated critical points, can we transform it, by diffeomorphisms of the source and of the target, into a holomorphic function?}
 
 Thurston\index{branched covering!postcritically finite!Thurston theorem} proved a fundamental result on the topological characterization of certain rational maps of the sphere, more precisely, he gave a characterization of postcritically finite branched coverings of the sphere by itself (branched coverings in which the forward orbits of the critical points are eventually periodic) that are topologically conjugate to rational maps. See \cite{Thurston83} and the exposition in \cite{BCL}.
   
If $S_1$ and $S_2$ are compact surfaces, there is a classical combinatorial relation that is satisfied by any branched covering $S_1\to S_2$, namely, the following:\index{Riemann--Hurwitz formula}

\begin{proposition}[Riemann--Hurwitz formula] 
For any branched covering $S_1\to S_2$  between compact surfaces, we have 
\[\chi(S_1)=d\cdot \chi(S_2)-\sum (k(p)-1)
,\] 
where $\chi$ denotes Euler characteristic, and where the sum is taken over the critical points in $S_1$, $k(p)$ being the branching order at $p$, for each critical point $p$.
\end{proposition}

\begin{proof}[Sketch of proof] Consider a triangulation of $S_2$ whose set of vertices contains the set of critical values, and lift it  to a triangulation of $S_1$. The formula follows from an Euler characteristic count.
\end{proof}

A famous problem, called the Hurwitz problem,\index{Hurwitz problem} asks for a characterization of pairs of compact surfaces equipped with data satisfying  the Riemann--Hurwitz formula that can be realized by branched coverings, see e.g., \cite{PP1, PP2}.  

A different kind of realization problem for branched coverings was formulated and solved by Thurston, see \S \ref{s:Thurston} below. It uses an object we call a \emph{Speiser curve}\index{Speiser curve} which we introduce in the next subsection.

       \subsection{Speiser curves, Speiser graphs and line complexes}\label{s:Speiser}
 
  Let $S$ be a Riemann surface which is a branched covering $\psi:   S\to R$ of the Riemann sphere $R$  with finitely many critical values $a_1,\ldots,a_q\in R$, $q\geq 2$.   It is understood in such a setting that the complex structure of $S$ is induced from that of $R$, that is, the map $\psi$ is holomorphic.      
    We draw on $R$ a Jordan curve $\gamma$ passing through the points $a_1,\ldots,a_q$ in some cyclic order chosen arbitrarily and we rename these points accordingly. We equip this curve with the natural induced orientation. We call this curve a \emph{Speiser curve}.\index{Speiser curve} (About the name, see the historical note below on Andreas Speiser.) The curve $\gamma$ decomposes the sphere into two simply connected regions $U_1$ and $U_2$, each having a natural polygonal structure with vertices $a_1,\ldots,a_q$  and sides $(a_1a_2), (a_2a_3),\ldots (a_qa_1)$. The pre-image of $\gamma$ by the covering map $\psi$ is a graph $\Gamma$ which we call a \emph{Speiser graph}.\footnote{\label{n:Speiser}In the literature, sometimes the name ``Speiser graph"\index{Speiser graph} is used for what we call here ``line complex" (which is the name Nevanlinna uses). We reserve the name ``Speiser graph" for this lift of the Speiser curve,\index{Speiser graph} that we shall use below.} It divides the surface $S$ into cells that are also equipped with natural polygonal structures, each one projecting either to $U_1$ or to $U_2$.      A point in the pre-image of some point $a_i\in R$, $i\in I$, may be unbranched; at such a point, there are only two polygons glued. At a branch point of order $m-1$ in the pre-image of some $a_i$ (with $m\geq 2$), there are  $2m$ polygons glued.    
                     
Now, we construct a graph in $R$ which is dual to the Speiser curve $\gamma$. We start by choosing two points $P_1$ and $P_2$ in the interior of $U_1$ and $U_2$ repectively. Then, for each edge $(a_ia_{i+1})$ of $\gamma$, we join $P_1$ to $P_2$ by a simple arc that crosses $\gamma$  at a single point in the interior of this edge, in such a way that any two such simple arcs intersect only at their extremities. The union of these arcs  constitutes a graph dual to $\gamma$. Its lift by $\psi^{-1}$ is a graph $G$ embedded in $S$, called the \emph{line complex}\footnote{See Footnote \ref{n:Speiser}.} of the covering.

 The line complex $G$ is a homogeneous graph of degree $q$: each vertex is adjacent to $q$ edges. Choosing a black/white coloring for the two points $P_1$ or $P_2$, every vertex of $G$ is naturally colored using these two colors (depending on whether it is a lift of $P_1$ or of $P_2$). Thus, the line complex is bicolored,\index{bicolored graph} that is, its vertices can be colored using only two colors and no edge joins two vertices of the same color. A common notation for black and white vertices, for line complexes, is $\times, \circ$ (see Figures \ref{N3} and \ref{N4}, borrowed from \cite{N1970}).

 The connected components of the complement in $S$ of the line complex $G$ are called the \emph{faces} of $G$.  Around each vertex of the line complex $G$ there is always the same number of faces, appearing with their indices, in the cyclic order, $a_1,\ldots, a_q$. The image of each face of $G$ by the covering map $\psi$ contains a unique branch value $a_i$ in its interior. Such a face has a natural polygonal structure which it inherits from its image in $R$. This polygon may  be of three types:

    \begin{enumerate}
   \item \label{f1} A polygon with an even number $2m$ of sides: such a polygon contains a critical point of finite order equal to $m-1$  (the example in Figure \ref{N4} contains polygons with 2 and 4 sides);
        \item  \label{f2} A polygon with an infinite number of sides: such a polygon is unbounded in $S$  (the three examples in Figure \ref{N3} contain such polygons);   
                   \item a bigon: such a polygon contains an unbranch point over an $a_i$ (the example in Figure \ref{N4} contains bigons).
                        \end{enumerate}
        
         In case \ref{f1}, the face is said to be algebraic, and in case \ref{f2} it is said to be logarithmic (the terminology comes from the theory of Riemann surfaces associated with holomorphic functions).
         
Topologically, the branched covering $S$ of the sphere is uniquely determined by the points $a_1,\ldots,a_q$, the Jordan curve $\gamma$ joining them and the line complex $G$.   The line complex in the universal cover of the punctured surface encodes the way the various lifts of the polygons $P_1$ and $P_2$ (the complementary components of the Speiser curve) fit together.

The three line complexes represented in Figure \ref{N3}, with their coloring, correspond to the universal coverings of the sphere with 2, 3 and 4 punctures.  In these cases, the surface $S$ is simply connected and the associated line complex is regular. More generally, the line complex associated with the universal covering of the sphere punctured at $n\geq 2$ points is an $n$-valent regular graph. In terms of holomorphic functions, the line complex on the left hand side of the figure is associated with the Riemann surface of the exponential function. Figure \ref{N4} represents the line complex of the Riemann surface of the function $f(z)=ze^{-z}$, which is branched over three points of the sphere: the points $0$ and $\infty$, each of infinite order, and the point $1/e$, of order one.

\begin{figure}[h]
\centering
\includegraphics[scale=1.3]{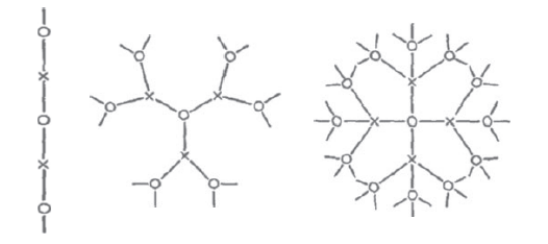}
\caption{\small Three examples of regular line complexes with their coloring. The associated branched coverings, restricted to the complement of the branch points, are universal coverings of the sphere with 2, 3 and 4 punctures respectively}\label{N3}
\end{figure}

\begin{figure}[h]
\centering
\includegraphics[scale=1.3]{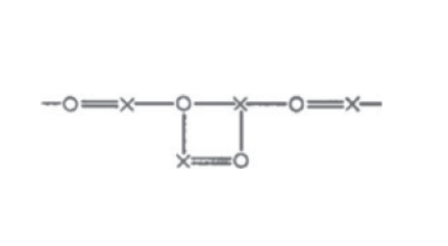}
\caption{\small The line complex of the Riemann surface of the function $f(z)=ze^{-z}$}\label{N4}
\end{figure}

There are sections on line complexes in Nevanlinna's book \cite{N1970}, in Sario and Nakai's book \cite{SN} and in L. I. Volkovyskii's comprehensive survey on the type problem \cite{Volkovyskii}. 
We shall return to line complexes in \S \ref{s:Nevanlinna}.
   \medskip
 
\noindent {\bf Historical note}  Andreas Speiser (1885--1970)\index{Speiser, Andreas} was a Swiss mathematician who studied in G\"ottingen, first with Minkowski and then with Hilbert. The latter became his doctoral advisor for a few months, after the death of Minkowski in 1909 (at age 44). Speiser defended his PhD thesis, in G\"ottingen, in 1909 and his habilitation in 1911, in Strasbourg. He worked in group theory, number theory and Riemann surfaces, and was the main editor of Euler's \emph{Opera Omnia}. He edited 11 volumes and collaborated to 26 others of this huge collection.  Ahlfors  writes in his comments to his \emph{Collected works} \cite[vol. 1, p. 84]{Ahlfors-Collected}, that the type\index{type problem}\index{problem!type} problem, which we shall review in \S \ref{s:Type}, was formulated for the first time by Andreas Speiser.\index{Speiser, Andreas} The latter introduced line\index{line complex} complexes (in a version slightly different from the one we use here) in his papers \cite{Speiser1929} and \cite{Speiser1930}.
 In his paper \cite{Ahlfors1982}, Ahlfors writes: ``Around 1930 Speiser had devised
a scheme to describe some fairly simple Riemann surfaces by means of a graph
and had written about it in his semiphilosophical style".  Ahlfors adjective is related to the fact that Speiser's articles do not contain  formulae.

\subsection{Thurston's realization theorem}\label{s:Thurston}

A branched covering from the sphere to itself can be iterated (composed with itself). The \emph{postcritical set}\index{postcritical set} of such a branched covering is the union of the set of critical points with its forward images by the mapping. A branched covering is said to be \emph{postcritically finite}\index{postcritically finite} if its postcritical set is finite.
                           In 1983, Thurston obtained a characterization of branched coverings  $f:S^2\to S^2$ of a topological 2-sphere that are topologically equivalent to postcritically finite rational maps of the Riemann sphere, that is, quotients of two polynomials (see \cite{Thurston83} and the exposition in \cite{BCL}). In 2010, while he was considering again the question of understanding holomorphic mappings from the topological point of view, Thurston obtained a realization theorem\index{branched covering!realization theorem} for branched coverings of the sphere which we present now. The result involves the
      Speiser graph\index{Speiser graph} associated with a branched covering of the sphere that we defined in \S \ref{s:Speiser}. We now state this result.

Let $f:S^2\to S^2$ be a generic degree $d$ branched covering of the topological  sphere $S^2$. Here, the word generic means that the cardinality of the set of critical values of $f$ is $2d-2$ (which is the largest possible). This implies that the cardinality of the set of critical points is also $2d-2$. Every critical point of such a map is simple (of order 1).
 
Let $\gamma$ be a Speiser curve  associated with this covering, equipped with its orientation,  and consider its lift $\Gamma= f^{-1} (\gamma)$. This is a graph we called Speiser graph.\index{Speiser graph} Its vertices are all of valence 4 (we do not take into account lifts of the critical values that are not critical points).

The Speiser curve\index{Speiser curve} $\gamma$ divides the sphere into two connected components which are  naturally equipped with polygonal structures. We color these components in blue and white. Each connected component of the complement of   $\Gamma$ is sent by $f$ to a connected component of $\gamma$, and therefore it carries naturally a color and a polygonal structure. We call such a component a \emph{face} of $S^2\setminus\Gamma$. Thurston addressed the following question:

 \emph{To characterize the oriented 4-valent graphs in $S^2$ that can be obtained by the above construction, that is, as a Speiser graph\index{Speiser graph} associated with some branched covering $f:S^2\to S^2$.} 
 
 He proved the following:

\begin{theorem}[Thurston, see \cite{Koch-Lei}]\label{th:Thurston}
Let $\Gamma$ be an oriented graph on $S^2$, with $2d-2$ vertices, all of valence 4. Then $\Gamma$  is a Speiser graph\index{Speiser graph} of a degree-$d$ branched covering of the sphere onto itself if and only if:
\begin{enumerate}
\item each face  of $S^2\setminus\Gamma$ is simply connected;
\item \label{balance1}  for any alternating blue-white coloring of the faces of $S^2\setminus\Gamma$, there are $d$ white faces and $d$ blue faces;
\item \label{balance2}  for every oriented Jordan curve embedded in $\Gamma$ and bordered by only blue faces on the left and only white faces on the right (excluding the corners), the Jordan disc bounded on the left by this curve contains strictly more blue faces than white faces.
\end{enumerate}
\end{theorem}

Conditions \ref{balance1} and \ref{balance2} are considered as \emph{balance conditions.}

Thurston's proof of Theorem \ref{th:Thurston} given in \cite{Koch-Lei} is based on the so-called ``marriage theorem"\index{marriage theorem} (due to Philip Hall) from combinatorics, a theorem dealing with finite subsets of a given set; it gives a necessary and sufficient condition on such a family that guarantees the choice of one element from each set such that these elements are all distinct (monogamy is the rule). The authors in \cite{Koch-Lei} ask for a better proof which would give an algorithm for constructing arbitrary rational functions whose critical values are on the unit circle
\cite[p. 255]{Koch-Lei}.

The Riemann--Hurwitz formula,\index{Riemann--Hurwitz formula} which we recalled in \S \ref{s:Hurwitz}, becomes in the case at hand, i.e., the case where $f$ is a degree-$d$ branched cover of the sphere by itself:
\[2d-2= \sum(k(p)-1).\]
 In this formula, the sum is taken as before over the critical points $p$, and $k(p)$ is the branching order at $p$. (This formula follows from the fact that the Euler characteristic of both surfaces is 1.) Written in this form, the Riemann--Hurwitz formula\index{Riemann--Hurwitz formula} says that the branching information at the various points forms a partition of the integer $2d-2$. The question of what are the partitions that are realized by a branched covering of the sphere to itself is still open. (Cases where the map does not exist are known to occur.)

 Thurston's result is part of his project of understanding the shapes of rational functions, where by ``shape" he means the evolution of the critical levels of such a function (see his comments on this problem in a MathOverflow thread he started in 2010, titled \emph{What are the shapes of rational functions?}). 
 
 The paper \cite{A2017} which contains a stratification of spaces of monic polynomials of a give degree is inspired by Thurtson's question. See also \ref{BE2009}

\section{The type problem}\label{s:Type}

   The Uniformization\index{type problem} Theorem\index{uniformization theorem} (Theorem \ref{t:uniformization}) says that every simply connected Riemann surface $S$ is either conformally equivalent to   the Riemann sphere, or to the complex plane, or to the unit disc. In the first (respectively, second, third) case, $S$ is said to be of elliptic (respectively, parabolic, hyperbolic) type.\index{parabolic type}\index{type!parabolic}\index{elliptic type}\index{type!elliptic}\index{hyperbolic type}\index{type!hyperbolic} The names  stem from the fact that the Riemann sphere, the complex plane and the unit disc are ground spaces for elliptic, parabolic and hyperbolic geometry respectively (elliptic and parabolic geometries are  alternative names for spherical and Euclidean geometries). The first case is distinguished from the other two by topology, since in this case $S$ is compact, whereas in the two other cases it is not. 
   
   The \emph{type problem} asks for a practical way to decide whether a given simply connected non compact  Riemann surface is of parabolic or elliptic type. 
   
   The answer  depends on the context in which the surface is defined. For instance, the surface may be embedded in Euclidean 3-space and equipped with the induced metric, or it may be equipped with some abstract Riemannian metric, or it may be obtained by pasting smaller surfaces (Euclidean or non-Euclidean polygons, etc.). It may also be given as an infinite branched covering of the sphere with some data at each branching point, or it may be  obtained by analytic continuation, like the universal cover of the Riemann surface of some holomorphic function. One can also imagine other situations. The multiplicity of these cases explains the variety of ways and techniques in which the type problem\index{type problem}\index{problem!type} has been addressed; we shall see some examples below.
   
   \subsection{Ahlfors on the type problem}
       Ahlfors\index{type problem!Ahlfors}\index{problem!type} emphasized at several occasions the importance of the type problem, which was one of his main research topics during the period 1929--1941. In his paper \cite{Ahlfors1932}, talking more precisely about the question of deducing the type of a Riemann surface associated with a univalent function, from the distribution of its singularities, he writes: ``This problem is, or ought to be, the central problem in the theory of functions. It is evident that its complete solution would give us, at the same time, all the theorems which have a purely qualitative character on meromorphic functions.''       
           Nevanlinna,  in the introduction of his book \emph{Analytic functions} (1953), ), writes (p. 1 of the English edition \cite{N1970}): ``[\ldots] Value distribution theory is thus integrated into the general theory of conformal mappings. From this point of view the central problem of the former theory is the \emph{type problem},\index{type problem}\index{problem!type} an interesting and complicated question, left open by the classical uniformization theory.''\index{uniformization}   
           
                 One of Ahlfors' earliest results on the type problem is a theorem contained in his paper 
  \cite{Ahlfors1931}. It gives the following condition for a simply connected surface $S$ which is a branched covering of the sphere whose branch values are all of finite degree, to be parabolic.\index{type!parabolic}\index{parabolic type} See also \ref{Ahlforsthesis}.
  
  \begin{theorem}[Ahlfors, \cite{Ahlfors1931}] \label{th:Ahlfors} Let $S$ be a simply connected surface which is a branched covering of the sphere with all branch values of finite degree, let $p$ be an arbitrary point in $S$ and let $n(t)$ $t\geq 0$ be a function defined on $S$ which associates with $t$ the number of branch points, counted with multiplicity, at distance $\leq t$ from $p$. If
  \[\int_a^\infty\frac{dt}{tn(t)}=\infty\] for some (or equivalently for all) $a>0$, then $S$ is of parabolic type.
  \end{theorem}
  
  Thus, the theorem says that if the amount of branching is relatively small, the surface is parabolic.   
  We shall see below other results on the type of branched coverings of the sphere that are also expressed in terms of the growth of the amount of branching. Heuristically, one can consider that the degree of branching measures a certain degree of negative curvature concentrated at the branch points, and  consequently, the larger the amount of branching is, the more negative curvature is concentrated at that point; this explains the fact that a surface with a large amount of branching is hyperbolic.
  
   In the setting of function theory, one also admits branch values of infinite degree, modeled on the ones that occur in the theory of Riemann surfaces of multi-valued analytic functions. The complex logarithm function  is an example of a multivalued function which gives rise to a branch value of infinite degree.

  In the next 3 subsections, we shall describe different approaches used by various authors to the study of the type problem.\index{type problem}\index{problem!type} They involve combinatorial, analytical and geometrical methods.

\subsection{Nevanlinna on the type problem}\label{s:Nevanlinna}
 
 Nevanlinna,\index{type problem!Nevanlinna}\index{problem!type} in the papers \cite{N1932a, N1932, N1970} considered in detail the type problem for simply connected surfaces that are branched coverings of the sphere with a finite number of critical values. Soon after, Teichm\"uller dealt with the same problem in his papers \cite{T9} and \cite{T13}, disproving a conjecture made by Nevanlinna. The fundamental tool that is used in these works is the line complex\index{line complex} that we recalled in \S \ref{s:Speiser}.    
 
  Nevanlinna formulated a principle which amounts to formalizing the fact that the type of a simply connected Riemann surface $S$ which is a branched covering of the Riemann sphere may be deduced from information on the associated  line complex.\index{line complex} The faces of this graph (i.e., the connected components of its complement) are in one-to-one correspondence with the branch points of the covering, if one discards bigons. A polygon with $2m$ sides is associated with a branch point of order $m-1$. Roughly speaking, Nevanlinna's claim is that if the amount of branching of the covering is small, the surface is of parabolic type, and if the amount of branching  is large, the surface is of hyperbolic type. He showed how the amount of branching can be deduced from the properties of the line complex.\index{line complex}  Nevanlinna writes, in  \cite[p. 308]{N1970}: ``It is thus natural to imagine the existence of a critical degree of branching that separates the more weakly branched parabolic surfaces from the more strongly branched hyperbolic surfaces." To make this precise, he introduced the notions  of mean branching and mean excess relative to the line complex.\index{line complex} These are analogues of the notion of global mean curvature of a surface. We shall recall his definition.
    
    To motivate the above principle, Nevanlinna first considered the case of a surface $S$ which, instead of being an infinite-sheeted  covering of the Riemann sphere, is finite-sheeted. In this case, the \emph{mean branching} is simply the sum of the orders of the  branch points divided by the degree of the covering.   
         
     In this case, the mean branching can also be computed using the line complex $G$.\index{line complex} 
 The number of vertices of the line complex is equal to twice the number of sheets. Indeed, if $d$ is the number of sheets in the covering,   then, using the notation in \S \ref{s:Speiser}, there are $d$ vertices which are lifts of  the point $P_1$ and $d$ others which are lifts of the point $P_2$. Each branch point of order $m-1$ is in the interior of a polygonal component of $S\setminus G$ having $2m$ sides and $2m$  vertices. Considering that twice the order of this branch point, that is, $2m-2$, is evenly distributed among these $2m$ vertices, each vertex of this polygon receives from the polygon itself a branching contribution equal to $\frac{2m-2}{2m}=1-\frac{1}{m}$. Taking the sum over the contributions coming from all the adjacent polygons, each vertex of the line complex\index{line complex} receives a total branching contribution equal to $\sum (1-\frac{1}{m})$, where the sum is over all the polygons adjacent to this vertex. Dividing by the total number of vertices, which is twice the number of sheets,  gives the  mean branching. Nevanlinna uses this second way of counting the mean branching for the definition of the mean branching in the case of an infinite covering. 
   
 Nevanlinna considered first the case of a \emph{regularly} ramified surface. This is a surface whose group of covering transformations acts transitively on the set of branch points (therefore it acts transitively on the line complex).\index{line complex} In this case, the mean ramification is equal to the ramification of an arbitrary polygonal component in the complement of the line complex. 
  He proved:
     
     \begin{theorem}[Nevanlinna \cite{N1970}, p. 311]
     For a regularly ramified surface, the following holds:
     \begin{enumerate}
   
     \item the surface is parabolic if the mean ramification is $2$;
    \item he surface is hyperbolic if the mean ramification is $>2$;
      \item The surface is elliptic if the mean ramification is $<2$;
        
     \end{enumerate}
     
     \end{theorem}
     \begin{proof} In his proof, Nevanlinna introduces a metric on each polygonal component of the complement of the line complex in which the measure of the angle at each vertex of such a polygon corresponding to a branch point of order $m-1$ is equal to $\pi/m$. 
     
     The excess of a polygon is equal to $\sum\frac{1}{m}-q+2$ where $q$ is as before the cardinality of the set of critical values, which is also the number of sides of the polygon. Since we are in the case of a regularly  ramified surface, the mean ramification is equal to the order of ramification at each branched point.

     In case 1, the angle excess of each polygon is zero and we can realize this polygon as a regular polygon in the Euclidean plane. 
     In case 2, the angle excess of each polygon is negative and we can realize it as a regular polygon in the hyperbolic plane. Case 3 corresponds to spherical geometry.   
     \qed
     \end{proof}
     
     In the case of an infinitely-branched covering, the average of the branching over all the vertices of $G$ is obtained by taking an exhaustion of this graph by an infinite number of finite sub-graphs, and taking the lower limit of a mean of the sum of the total branchings of the vertices \cite[p. 309ff]{N1970}.
     The conjecture is formulated by Nevanlinna as a question in  \cite[p. 312]{N1970} in the following terms: 
   \emph{Is the surface parabolic or hyperbolic according as the mean excess is zero or negative?} After formulating this question,     Nevanlinna notes that  Teichm\"uller in his paper \cite{T13}, disproved the conjecture by exhibiting a hyperbolic simply connected Riemann surface branched over the sphere with a line complex\index{line complex!mean excess} 
 whose mean excess\index{mean excess} is zero.

\subsection{Quasiconformal mappings and Teichm\"uller's work on the type problem}\label{Teich}
 
 In the\index{type problem!Teichm\"uller}\index{problem!type} course of proving the above result, Teichm\"uller used quasiconformal mappings, and we next recall this notion.
 
 To  a differentiable complex-valued function $f(z)=u(x,y)+iv(x,y)$ defined on an open subset $\Omega$ of the complex plane (or on a Riemann surface), one asssociates its \emph{complex dilatation} by taking first the partial derivatives 
 $f_x=u_x+iv_x$ and $f_y=u_y+iv_y$, setting 
 \[f_z=\frac{1}{2}(f_x-if_y)\]
 and 
 \[f_{\overline{z}}=\frac{1}{2}(f_x+if_y)\] and defining the \emph{complex dilatation} of $f$
 as
 \[\mu(z)=\frac{f_{\overline{z}}(z)}{f_z(z)}
.\]

 The (real) \emph{dilatation} of $f$ is then given by the real function
 \[K(z)=\frac{1+\vert \mu(z)\vert}{1-\vert \mu(z)\vert}
 \]
and $f$ is said to be  \emph{$K$-quasiconformal}, for some $K>0$, if this function satisfies
\[K(z)<K\]
for all $z$ in $\Omega$.

  In practice, one allows $f$ to be only  absolutely continuous on almost every line which is   parallel to the real or imaginary axis. This is sufficient for  the partial derivatives $\frac{\partial f}{\partial x}$ and  $\frac{\partial f}{\partial y}$ to exist almost everywhere.
  
  The dilatation is then defined a.e., and the last inequality needs only to be satisfied a.e.
  
A computation shows that the \emph{dilatation} of $f$ at an arbitrary point is equal to the ratio of the major axis to the minor axis of the  ellipses that are images by the differential of the map at that point, of circles in the tangent space centered at the origin. Therefore, a diffeomorphism $f$ between two Riemann surfaces is $K$-quasiconformal if it takes an infinitesimal circle to an infinitesimal ellipse of eccentricity uniformly bounded by $K$.

  The map $f$ is quasiconformal if it is $K$-quasiconformal   for some finite $K$.

  This leads us to a notion of quasiconformally equivalent Riemann surfaces and \emph{quasiconformal structure},\index{quasiconformal structure} a weakening of the notion of conformal structure: A quasiconformal structure on a 2-dimensional manifold is a family of conformal structures such that the identity map between any two of them is quasiconformal. We say that two  conformal structures in such a family are quasiconformally equivalent.

 Teichm\"uller wrote a paper dedicated to applications of quasiconformal mappings the type problem, see \cite{T9}, titled \emph{An application of quasiconformal mappings to the type problem}. In the introduction to this paper, he writes:
 ``Recently, the problem of determining the properties of the schlicht mapping of
[a simply connected Riemann surface] $\mathfrak{M}$ from its line complex has been frequently studied. So far in the foreground
has been the type problem: How can one determine, from the given line complex,
if the corresponding surface $\mathfrak{M}$  can be mapped one-to-one and conformally onto
the whole plane, the punctured plane, or the unit disk? One is still very far from
finding sufficient and necessary conditions."

In this work on the type problem,\index{type problem}\index{problem!type}  Teichm\"uller started by showing that two branched coverings of the sphere whose line complexes are equal are quasiconformally equivalent. He gave a criterion for hyperbolicity\index{type!hyperbolic}  and he proved that
 the type of a simply connected open Riemann surface is a quasiconformal invariant \cite{T9}. The latter result is a direct consequence of the fact that the complex plane and the unit disc cannot be quasiconformally mapped onto each other, a result which he proves in the same paper. In principle, this result reduces the type problem\index{type problem!Teichm\"uller}\index{problem!type} to a simpler problem, since quasiconformal mappings\index{quasiconformal mapping}\index{mapping!quasiconformal} exist much more abundantly than conformal mappings. To show that a surface $S$ is parabolic (respectively hyperbolic), it suffices to exhibit a quasiconformal homeomorphism between $S$ and the complex plane (respectively the unit disc).

To end this subsection on Teichm\"uller's work on the type problem, let us mention the survey \cite{ABP} titled \emph{Teichm\"uller’s work on the type problem}.

                       \subsection{Lavrentieff on the type problem}\label{s:almost}

Lavrentieff's\index{type problem!Lavrentieff}\index{problem!type} work on the type problem is based on  the notion of ``almost analytic function"\index{almost analytic function}\index{function!almost analytic} which he introduced, another  weakening of the notion of conformal mapping which is close to (but slightly different from) the notion of quasiconformal mapping.\index{quasiconformal mapping}\index{mapping!quasiconformal}

To state this property, let us first recall that there exist several equivalent conditions for a differentiable function $f$ from the complex plane to itself (or, more generally, between surfaces) to be holomorphic, and weakening each of them gives a weakening of a function to be holomorphic. Among these conditions, we mention the following:
\begin{enumerate}  
\item   $f$ is angle-preserving.
 \item the real and imaginary parts of $f$ satisfy the Cauchy--Riemann equations.
 \item $f$ satisfies $\frac{\partial f}{\partial \overline{z}}=0$.
 \item The Taylor series expansion of $f$ around every point is convergent in some open disc around this point;

\item The following notion of holomorphicity is stated in  terms of almost complex structures:\index{almost complex structure}
 Given a surface $S$ equipped with a $J$-holomorphic function,\index{J@$J$-holomorphic} a differentiable function $f:(S,J)\to \mathbb{C}$ is said to be \emph{$J$-holomorphic}\index{J@$J$-holomorphic} if  for every point $p$ on $S$ and for every tangent  vector $u$ at $p$, we have 
\[(df)_p(J_pu)=i(df)_p(u).\]

\item \label{prop:circles} The function $f$  possesses a complex derivative, that is, infinitesimally, it acts by multiplication by a complex number. 
 
\end{enumerate}

   Lavrentieff's  almost analyticity\index{almost analytic function}\index{function!almost analytic} property is a weakening of Property \ref{prop:circles}, a property equivalent to the fact that at each point the Jacobian matrix of $f$ acts on the tangent space as a rotation followed by a homothety. This is also expressed by the fact that the map sends infinitesimal circles to infinitesimal circles. (Thus, this property is close to the notion of quasiconformal mapping which we recalled in \S \ref{s:Nevanlinna}.)

 In fact, in the most general case, the Jacobian matrix of a differentiable map between tangent spaces, since it acts linearly, takes circles centered at the origin to ellipses centered at the origin, and Lavrentieff's property is formulated in terms of  the boundedness of the eccentricty and the directions of these ellipses. We now state this condition in precise terms.
 
\begin{definition} {\bf (Lavrentieff \cite{Lavrentieff1935})}\label{def:AA} A function $f:\Omega\to \mathbb{C}$ defined on a domain $\Omega\subset \mathbb{C}$ is said to be \emph{almost analytic}\index{almost analytic function}\index{function!almost analytic} if the following three properties hold:

 \begin{enumerate}
\item $f$ is continuous.

\item $f$ is orientation-preserving and a local homeomorphism on the complement of a countable closed subset of  $\Omega$.

\item There exist  two real-valued functions $p:\Omega\to [1,\infty[$ and $\theta: \Omega\to[0,2\pi]$, called the \emph{characteristics of $f$}, defined as follows:
\begin{itemize}
\item There is subset $E$ of $\Omega$ which consists of finitely many analytic arcs such that $p$ is continuous on $\Omega\setminus E$ and $\theta$ is continuous at each point $z$ satisfying $p(z)\not=1$. ($E$ might be empty.)

\item On every domain $\Delta\subset \Omega\setminus E$ whose frontier is a simple analytic curve, $p$ is uniformly continuous. Furthermore, if such a domain $\Delta$ and its frontier do not contain any point satisfying $p(z)=1$, then $\theta$ is also uniformly continuous on $\Delta$.

\item For an arbitrary point $z_0$ in $\Omega\setminus E$, let $\mathcal{E}$ be an ellipse  centered at $z_0$, let $\theta(z)$ be the  angle between its major axis and the real axis of the complex plane, and let $p(z_0)=\frac{a}{b}\geq 1$ be the ratio of the major axis $a$ to its minor axis $b$. Let $z_1$ and $z_2$ be two points on the ellipse $\mathcal{E}$ at which the expression $\vert f(z)-f(z_0)\vert$ attains its maximum and minimum respectively. Then,
\[\lim_{a\to 0}\left| \frac{f(z_1)-f(z_0)}{f(z_2)-f(z_0)}\right| =1.
 \]
\end{itemize}
\end{enumerate}
\end{definition}

 One goal of  Lavrentieff's article \cite{Lavrentieff1935} is to show that almost analytic\index{almost analytic function}\index{function!almost analytic} functions share several properties of analytic functions. This includes a compactness result for families of  almost analytic functions with uniformly bounded characteristics, which is a generalization of a known compactness property that holds for families of analytic functions. It also includes a generalization of Picard's big Theorem \cite[\S 3]{Lavrentieff1935}. The generalization of  the same theorem for quasiconformal mappings was already obtained by Gr\"otzsch, see \cite{Groetzsch1928a}.
 
 The question of finding an almost analytic\index{almost analytic function}\index{function!almost analytic} function from its characteristics, which Lavrentieff solves in the same paper and which we state now, is a form of what was later called the ``Measurable Riemann Mapping Theorem". This  is a wide generalization of the Riemann Mapping Theorem:

\begin{theorem}[Lavrentieff \cite{Lavrentieff1935}] \label{th:MRMT}
Given any two functions $p(z)\geq 1$ and $\theta(z)$ defined on the closed unit disc $\vert z\vert \leq 1$ such that $p$     and $\theta$ satisfy the properties in the above definition and such that $p(z)<M$ for some constant $M$, there exists an almost analytic function\index{almost analytic function}\index{function!almost analytic} $f$ realizing a self-homeomorphism of the closed disc $\vert z\vert \leq 1$ and having the characteristics $p$ and $\theta$.
\end{theorem}

 In the same paper, Lavrentieff proves the following result on the type\index{type problem!Lavrentieff}\index{problem!type} problem. Like Ahlfors' theorem (Theorem \ref{th:Ahlfors}) and like the theorem of Milnor that we shall quote below (Theorem \ref{t:Milnor}), Lavrentieff's theorem reduces the question of proving parabolicity to that of  proving that a certain integral diverges.

\begin{theorem}[Lavrentieff \cite{Lavrentieff1935}]Let $S$ be a surface in the 3-dimensional Euclidean space which is the graph of a differentiable function $$
t = f(x,y), \quad x^2 + y^2 <\infty, 
$$
whose partial derivatives $\frac{\partial f}{\partial x}$ and $\frac{\partial f}{\partial y}$ are continuous. For each $r\geq 0$, let 
\[M(r)=\max_{x^2+y^2=r^2}1+\left| \mathrm{grad} \,f(x,y) \right|.\] If $\int_{1}^{\infty}\frac{dr}{rM(r)}=\infty$,  then $S$ is of parabolic type.

\end{theorem}

 Lavrentieff, in his paper, reduced the type\index{type problem!Lavrentieff}\index{problem!type} problem to Theorem \ref{th:MRMT}, that is, to the problem of finding an  almost analytic function\index{almost analytic function}\index{function!almost analytic} having  given characteristics $p$ and $\theta$, such that for $z=(x,y)$, 
 $$
p(z) = \frac{1}{\cos\alpha(z)},
$$
where 
 $\alpha(z)$ is the angle formed by the $(x,y)$-plane and the tangent plane of $S$ at $(x,y,t)$ and where $\theta(z)$ is the angle formed by $\mathrm{grad} \,f(x,y)$ and the $x$-axis.  
\subsection{Milnor on the type problem}

Milnor,\index{type problem!Milnor}\index{problem!type} in his paper \cite{Milnor1997}, studied the type problem for the case of a surface $S$ embedded in Euclidean 3-space which is complete and whose Gaussian curvature depends only on the distance to some fixed point $p$. The Riemannian metric induced on this surface can be written in the polar coordinates $(r,\theta)$ as $dr^2+g(r)^2d\theta^2$ where $2\pi g(r)$ is a smooth positive function that denotes the length of the geodesic circle of radius $r$ centered at $p$. The Gaussian curvature is then given by $K(r)=-(d^2g/dr^2)/g$. 

Milnor proves the following:
 
 \begin{theorem}\label{t:Milnor}
 For a surface $S$ as above we have:
\begin{enumerate}
\item \label{M1}  $S$ is parabolic if and only if for some (or equivalently for any) $a>0$, the integral $\int_a^\infty dr/g(r)=\infty$.

\item  \label{M2} If for an arbitrary $\epsilon >0$ we have $K(r)\geq -(r^2\log r)$ for $r$ large enough, then $S$ is parabolic.
 
\item  \label{M3} For any $\epsilon >0$, we have $K\leq -(1+\epsilon)/r^2\log r)$ for large enough $r$,  and if the function $g(r)$ is unbounded, then $S$ is hyperbolic.
 \end{enumerate}
 \end{theorem}

Items \ref{M2} and \ref{M3} are not surprising if we consider them again from the point of view that enough negative curvature corresponds to a hyperbolic type and enough positive curvature corresponds to a parabolic type. An interesting fact here is that the difference between the two cases is made by an $\epsilon$ which can be chosen to be arbitrarily small, and in this sense the result is the best possible in the given situation.

 \begin{proof}[Milnor's proof] The proofs of Items \ref{M2} and \ref{M3} follow from Item \ref{M1} which, like the results of Ahlfors and Lavrentieff that  we mentioned above, characterizes parabolicity in terms of the divergence of a certain integral.  

 For the proof of \ref{M1}, Milnor introduces a new coordinate $\rho=\int_a^r ds/g(s)$. In the coordinates $(\rho,\theta)$, the metric becomes
 \[g^2(d\rho^2+d\theta^2),\] 
and is thus conformally equivalent to the Euclidean metric. (Coordinates in which the metric is conformal to the Euclidean metric are called isothermal, another term which comes from the theory of heat transfer.) The exponential map $(\rho,\theta)\mapsto e^{\rho +i\theta}$ maps $S$ conformally onto $\mathbb{C}$. (The mapping is a priori defined in the complement of the point $p$, but it extends to $p$ because it is bounded, differentiable and conformal in a neighborhood of $p$, therefore this singularity is removable.) 

\end{proof}\qed

Milnor, at the end of his paper, adresses the question of finding an effective criterion for deciding, for a simply connected open surface embedded in 3-space, whether it is parabolic or hyperbolic.

\subsection{Probablistic approaches} 
There are other approaches to the type problem\index{type problem}\index{problem!type} that we shall not touch upon here. In particular, there is a probabilistic point of view. A classical result in this direction says that a surface equipped with a Riemannian metric is parabolic if and only if the associated Brownian\index{Brownian motion}\index{type problem!Brownian motion}\index{problem!type} motion is recurrent. There is a large literature on this question, see e.g., Kakutani's paper \cite{Kakutani}, which appeared the same year as Teichm\"uller's paper \cite{T9} on the type problem. Kakutani gives the following criterion: 
  Let $S$ be a simply-connected open Riemann surface which is an
infinite cover of the sphere and let 
$D$ be a simply connected subset  
of $S$ bounded by a Jordan curve $\Gamma$. For a point $\zeta$ in $S \setminus D$, let $u(\zeta)$ denote the probability that the Brownian motion on $S$ starting at 
$\zeta$ enters into $\Gamma$ without getting out of $S\setminus D$ before. Then, one of the
following two cases holds:
\begin{enumerate}
\item  $u(\zeta) < 1$ everywhere in $S\setminus D$ and in this case $S$  is hyperbolic;

\item  $u(\zeta)$ is identically equal to 1   $S\setminus D$ , and in this case $S$ is
parabolic.
\end{enumerate}

Kakutani wrote  later several papers related to the type problem, see, e.g.,
\cite{Kakutani1937, Kakutani1944, Kakutani1945}.

Let us also mention Z. Kobayashi's papers on Riemann surfaces in which he discusses the type problem; see in particular his paper \cite{Kobayashi1} in which he gives sufficient conditions under which a surface is of parabolic type, and his paper \cite{Kobayashi2} in which he presents a theory based on Kakutani's paper \cite{Kakutani} and Teichm\"uller's paper \cite{T13}, generalizing both works. 
Figure \ref{fig:Kobayashi} is extracted from Kobayashi's paper \cite{Kobayashi1}. It is reproduced here for the purpose of showing the nobility of this kind of drawing compared to computer-drawing everybody uses today in mathematical papers.

\subsection{Electricity}

P. Doyle and J. L. Snell expanded on the relation between the type problem, Brownian motion and the propagation of an electric current on the surface, see the paper \cite{DS} where the main underlying idea is that the unit disc has finite electrical resistance while the plane has infinite resistance. 
See also Doyle's review in the Bulletin of the AMS, \cite{Doyle}, in which the author reports on the type problem for a covering  of the Riemann sphere
with $n$ punctures using the properties of the Speiser graph of the covering, surveying the important  work of Lyons--McKean--Sullivan \cite{LMS} and of McKean--Sullivan \cite{MS} on this topic. Incidentally, the author seems not to be aware of the work of Teichm\"uller on these questions.

Introducing electricity in the type problem is completely in the tradition of Riemann who, at several places in his work on Riemann surfaces uses arguments from electricity. This occurs in particular in  the two main papers where he deals with Riemann surfaces, namely, his doctoral dissertation \cite{Riemann1851} and his paper on Abelian functions \cite{Riemann1862}.  For instance, in the Dirichlet principle\index{Dirichlet principle} that he used in his proof of the Riemann Mapping Theorem\index{Riemann Mapping Theorem} (see \S \ref{s:unif-simply} above), Riemann considered  the function $f$ defined on the boundary of the domain $\Omega$ as a time-dependent electrical potential, and his conclusion in the solution to the Dirichlet problem\index{Dirichlet problem} consisted in letting the system evolve and reach an equilibrium state. The function obtained necessarily satisfies a mean value property and therefore is harmonic.\index{harmonic function} 

\medskip

\noindent {\bf Historical note} In the summer of 1858, Riemann gave a course titled \emph{The mathematical theory of gravitation, electricity and magnetism} \cite{Sellien}. Besides his mathematical papers in which he used electricity, Riemann also wrote papers on electricity. See his unfinished paper \emph{Gleichgewicht der Electricität auf Cylindern
mit kreisförmigen Querschnitt und parallelen Axen} (On the equilibrium of electricity
on cylinders with circular transverse section and whose axes are parallel) \cite{Riemann1857} (1857)  published posthumously in the second edition of his \emph{Collected
works} which concerns the distribution of
electricity or temperature on infinite cylindrical conductors with parallel generatrices.
The reader interested  in more details may refer to the survey titled  \emph{Physics in Riemann's mathematical papers} \cite{P2017}. 

\medskip

   \begin{figure}
\begin{center}
\includegraphics[width=13 cm]{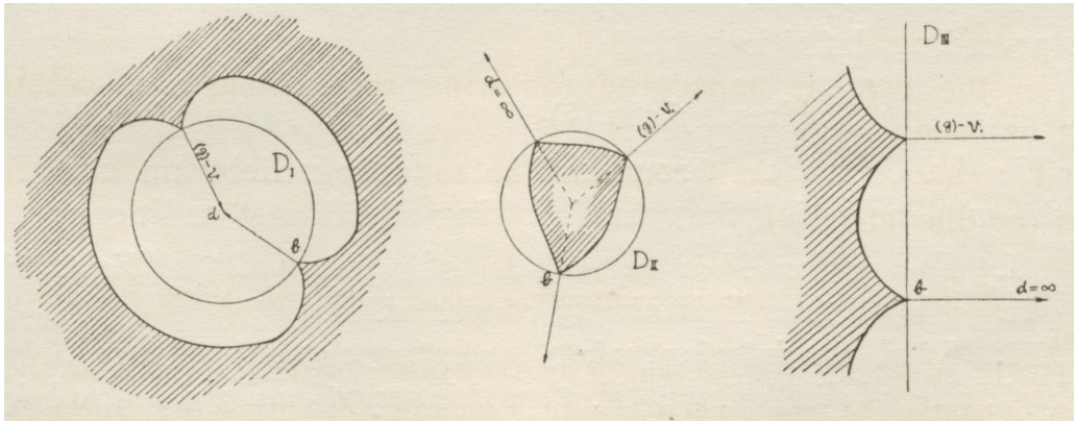} 
\caption{From Kobayashi's paper \cite{Kobayashi1} on the type problem}
\label{fig:Kobayashi}
\end{center}
\end{figure}

\begin{remark} ({\bf Higher-dimensional generalizations}) One advantage of introducing random walks and probabilistic methods in the theory of parabolicity/hyperbolicity of Riemann surfaces is that it allows this theory to be generalized to graphs, discrete groups and higher-dimensional manifolds. There are several works in this direction. I mention Marc Troyanov's  paper \cite{Troyanov-Type} in which he discusses parabolicity for $n$-dimensional manifolds. More precisely, in this paper, Troyanov surveys an invariant of Riemannian manifolds which is related to the non-linear potential theory of the $p$-Laplacian  and which determines a property which generalizes parabolicity or hyperbolicity of surfaces. The key notion that this leads to is that of $p$-parabolicity and $p$-hyperbolicity. 
In the generalized setting of the paper \cite{Troyanov-Type}, 
the classical dichotomy parabolic/hyperbolic becomes $2$-parabolic/$2$-hyperbolic, i.e., $p=2$.
The author focuses in particular on the relationship between the asymptotic geometry of a manifold and its parabolicity.

\end{remark}

\section{Uniformization: Geometry and combinatorics}\label{s:combinatorics}

\subsection{Dessins d'enfants}\label{s:dessins}

A {\it map}\index{dessin d'enfants} on a compact oriented surface $S$ is a bicolored finite graph $\Gamma$ embedded in $S$ such that each connected component of $S\setminus \Gamma$ is contractible. The vertices of the graph are colored by $0,1$ and each edge joins vertices of different colors. A connected component of $S\setminus \Gamma$ is called a \emph{cell} of the map.  A map $(S,\Gamma)$ can be upgraded to a triangulation $(S,\Delta_\Gamma)$ by choosing a point in each cell and connecting the chosen point by a star-shaped system of curves to the vertices of $\Gamma$ that are on the boundary of the cell. The triangulation $\Delta_\Gamma$ is special in that it admits a checkerboard\index{checkerboard coloring} black-white coloring (that is, triangles meeting from opposite sides along an edge have opposite colors). To see this, we use the symbols $0,1,\infty$ to color the vertices of the triangulation $\Delta_\Gamma$ in such a way that the chosen points in the cell gets the color $\infty$. Each triangle has vertices of different colors, hence becomes oriented
by imposing the cyclic order $0,1,\infty$.
Finally we color white those triangles for which the induced orientation from $S$ coincides with the cyclic orientation of the vertices.

The minimal example of a map is $M=(\mathbb{C}\cup \{\infty\},[0,1])$. Here, the graph is $\Gamma=[0,1]$ with two vertices and one edge embedded in the Riemann sphere $\mathbb{C}\cup \{\infty\}$. This map has only one cell. Now choose the point $\infty$ connected by $[1,\infty]$ and $[\infty=-\infty_\mathbb{R},0]$ as an upgrading to a triangulation with two triangles, the white one being the so-called upper half-plane $\mathbb{C}_+$. We denote by $M_\infty$ this minimal map together with the above upgrading to a triangulation.

 Speiser graphs discussed in \S \ref{s:Speiser} are maps in the above sense, except that the surfaces on which they are defined are not necessarily compact.

Let it be given the combinatorial data consisting of a triple $D=(S,\Gamma,\Delta_\Gamma)$ of a map $\Gamma$ with its $0,1$ bicoloring  on a compact surface $S$, together with the upgrading to a checkerboard colored triangulation $\Delta_\Gamma$. 

A smooth mapping $f_\Gamma:(S,\Gamma,\Delta_\Gamma)\to  M_\infty$ is well defined up to isotopy. It sends vertices to vertices respecting colors, edges to edges, and its restriction
to the complement of the vertex set of the triangulation is a submersion.

Let $J_\Gamma$ be the complex structure on $S$ obtained by pulling back the complex structure of $\mathbb{P}^1(\mathbb{C})\setminus \{0,1,\infty\}$ and
 by extending this structure using the removable singularity theorem.
 
The result is again an upgrading
from the combinatorial data 
$\Gamma \subset S$ to a Riemann surface 
$(S,J_\Gamma)$ 
together with
a holomorphic function $f_D:(S,J_\Gamma)\to \mathbb{P}^1(\mathbb{C})$ having its critical values in the set $\{0,1,\infty\}$.

The set of possible combinatorial data up to isotopy is countable and the set of Riemann surfaces up to bi-holomorphic equivalence is uncountable. Thus, a natural question arises, namely, what Riemann surfaces appear as an upgrading of a map.
The answer is given by Belyi's Theorem:
\begin{theorem}[Belyi \cite{Bel1979}] For a compact Riemann surface $S$ the following are equivalent:

\begin{enumerate}
\item $S$ carries a holomorphic map $f:S\to \mathbb{P}^1(\mathbb{C})$ with at most $3$ critical values (which can be taken, without loss of generality, to be $\{0,1,\infty\}$);

\item $S$ is bi-holomorphically equivalent to an upgraded map surface;

\item $S$ is bi-holomorphically equivalent to a complex curve, that is, a Riemann surface embedded in the projective space $P^n=\mathbb{P}^n(\mathbb{C})$ defined by a set of polynomials with coefficients in a number field, that is, a finite field extension of the field of rational numbers.
\end{enumerate}

\end{theorem}

 The implication $3\Rightarrow1$ is the ``easy part"; indeed, the inclusion $S\subset P^n$ as an algebraic curve defined by a set of polynomials with coefficients in a number field can be composed (with care) with projections to lower-dimensional spaces $P^n\to P^{n-1}$ until one gets a meromorphic function on $S$. Using the fact that this function is itself defined on a number field, there is a way of reducting the number of critical values to at most three.

The strength of Belyi's\index{Belyi theorem} Theorem can already be appreciated in the case of genus $0$ surfaces and maps having only one cell, or, equivalently, in the case where  the graph $\Gamma$ is a planar tree in the Gaussian plane $\mathbb{C}$. The above construction leads in this case of a map $\Gamma\subset \mathbb{C}$ to a polynomial $f_\Gamma:\mathbb{C}\to \mathbb{C}$ having at most two critical values. Such a polynomial $f_{\Gamma,e}$ becomes well defined if one marks an edge $e$ of $\Gamma$ and requires that $f_{\Gamma,e}$ sends its extremities to $0,1$. 

The study of polynomial mappings $P:\mathbb{C} \to \mathbb{C}$ with at most two critical values was introduced by G. Shabat, see  \cite{S-V1990} and \cite{S-Z1994}.  These polynomials generalize Chebyshev polynomials\index{Chebyshev polynomial} which have only two critical values and such that these critical points are all of Morse type. Equivalently, there is a maximal number of such critical points.

This fact leads to an action of the absolute Galois group ${\rm Gal}(\bar{\mathbb{Q}},\mathbb{Q})$ on the set of isotopy classes of bicolored planar trees with marked edges. This action is faithful and a window for observing and perhaps understanding the absolute Galois group. See Geothendieck's \emph{Sketch of a program} \cite{Grothendieck}. We also refer to the review made in \cite{AJP} of the relevant parts of the \emph{Esquisse}.

The inverse image of the interval $[0,1]$ by 
$f_{\Gamma,e}$ may look like a drawing of a person by a child, hence the name ``dessin d'enfants"\index{dessin d'enfants} given by Grothendieck to this mathematical object whose study he promoted in \cite{Grothendieck}. For an introduction to dessins d'enfants we refer to the original article  \cite{Grothendieck} and to the expositions in \cite{G-G2012, Guillot, HS}.
The reader interested in the relation between dessins\index{dessin d'enfants} d'enfants and the deformation theory of Riemann surfaces may refer to \cite{H2007}.

Voevodsky and Shabat in \cite{VS1989} considered surfaces equipped with Euclidean structures with cone singularities obtained by gluing Euclidean equilateral triangles along their sides. They showed that the Riemann surface structure underlying such a metric space, as an algebraic curve, is defined over a number field. Cohen, Itzykson and Wolfart in \cite{CIW} showed that the same construction works using congruent hyperbolic triangles (triangles in the Bolyai--Lobachevsky plane), instead of Euclidean.

\subsection{Slalom polynomials}\label{s:slalom}

We shall use the notions of slalom\index{slalom polynomial}\index{polynomial!slalom} polynomial and slalom curve\index{slalom curve} and we shall recall the definitions. For a more detailed exposition we refer to the original article \cite{A1998} and the recent book \cite{A2021} by the first author.

  We start with the definition of the planar tree\index{planar tree} $\Gamma_P$ associated with a generic monic polynomial\index{generic monic polynomial} $P$ of degree $n+1$. This is the closure of the union of the flowlines $\gamma: I\to \mathbb{C}$ of $-\mathrm{grad}(\log\vert P\vert)$ that satisfy the following properties:
\begin{itemize}
\item $P$ and its derivative $P'$ do not vanish on the image of $\gamma$;
\item the image of $\gamma$ is maximal with respect to inclusion (in other words, $\gamma$ is not the restriction of a flowline defined on an interval $J$ strictly containing $I$);
\item the image of $\gamma$ is bounded in $\mathbb{C}$, except when $I$ is the positive real line and where this image converges to $\infty$.
\end{itemize}

Consider now a rooted planar tree 
$\Gamma$ with $n+1$ edges in which one vertex is connected to $+\infty$ (considered as the root) by an unbounded edge.  A \emph{slalom curve}\index{slalom curve} $Sl(\Gamma)$ of $\Gamma$ is an immersed copy of the circle  $\gamma:S^1 \to \mathbb{C}$ having $n$ transversal double points at the midpoints of the bounded edges of 
$\Gamma$ (Figure \ref{fig:root}). The curve $Sl(\Gamma)$ intersects $\Gamma$ only and transversely twice at the midpoints of the bounded edges and once at the unbounded edge  such that each bounded complementary region of the curve contains exactly one vertex.

A {\it slalom polynomial $SlP_\Gamma(z)$ for $\Gamma$}\index{slalom polynomial} is a generic\index{polynomial!slalom} monic polynomial $P$ of degree $n+1$ with $|P(s)|=1$ for each zero $s$ of $P'$
such that the colored trees $\Gamma_P$ and 
$\Gamma$ are  isotopic relative to $+\infty$. Generic means in this context that $P$ has $n+1$ distinct roots, and that $P'$ has $n$ distinct roots.

A {\it slalom polynomial}\index{slalom polynomial} is\index{polynomial!slalom} a generic monic polynomial $P$ that is a slalom polynomial for some tree $\Gamma_P$. 

Slalom polynomials\index{slalom polynomial}\index{polynomial!slalom}  $P$ share the following properties with Shabat and Chebyshev polynomials: $P'$ has $n={\rm degree}(P)-1$ distinct roots and all critical values of $P$ have absolute value $1$.
Thus, the roots of $P'$ are the double points of the curve $|P(z)|=1$, which is a slalom curve\index{slalom curve} for the tree $\Gamma_P$. 

\begin{theorem}[Existence of slalom polynomials] Given a rooted planar tree 
$\Gamma$ with $n+1$ edges as above, there exists a slalom polynomial $P$ for $\Gamma$.
\end{theorem}

\begin{proof}

There exists a continuous  mapping $\phi:\mathbb{C}\to \mathbb{C}$ that is smooth except at the midpoints of $\Gamma$ and such that
\begin{itemize}

\item the restriction of $\phi$ to the unbounded component of the complement of $Sl(\Gamma)$ is a degree $n+1$ covering map of
the complement of the closed unit disc in $\mathbb{C}$;

\item the restriction of $\phi$ to a bounded region of the complement of $S(\Gamma)$ is a diffeomorphism to the open unit disc in $\mathbb{C}$ that sends the vertex to $0$ and half edges to radial rays;

\item $\phi$ is holomorphic near $\infty$.
\end{itemize}

\begin{figure}[h]
\centering
\vskip-1cm
\includegraphics[scale=0.6]{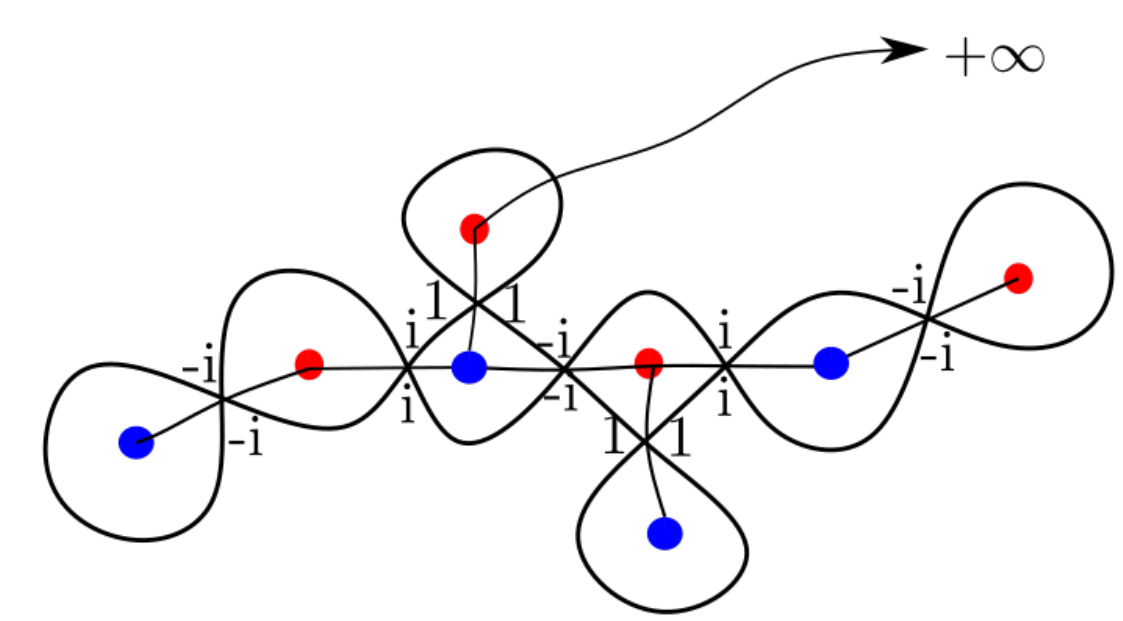}
\caption{A rooted colored tree having one edge asymptotic to $+\infty$ with a slalom curve. The double points of this curve are the critical points of the polynomial.}
\label{fig:root}
\end{figure}

Let $J$ be the conformal structure on 
$\mathbb{C}$ such that $\phi:(\mathbb{C},J)\to \mathbb{C}$ is holomorphic and let $U:\mathbb{C}\to (\mathbb{C},J)$ be a uniformisation of the structure $J$. The composition $P=\phi \circ U:\mathbb{C}\to \mathbb{C}$ is a degree $n+1$ polynomial. 
The polynomial $P$ has $n+1$ roots, the derivative has $n$ roots at the points $s$ with $U(s)$ being a midpoint of an edge. 
By adjusting the uniformisation map $U$ one achieves that $P$ is monic and generic.

The flowline from $+\infty$ ends at a root of
$P$ that we color red and which becomes the root of $E_P$. This red root can be changed by substituting $\lambda z $ to $z$. This substitution turns the picture and gives $2n$ possibilities for attaching the unbounded edge. 
Color the tree $\Gamma$ with the two colors, red and blue. Choose a disc $D_R$ that contains $\Gamma$ up to a part of the real axis.

It follows that $\Gamma$ and $\Gamma_P$ are isotopic relative to $+\infty$, showing that all the ${\rm Cat}(n)$ combinatorial data consisting of a rooted planar tree as above are realized. 
Here ${\rm Cat}(n)$ designates the $n$-th   \emph{Catalan number}, that is, the number of distinct triangulations of a convex polygon with $n+2$ sides obtained by adding diagonals.

 The polynomial $P=P_\Gamma$ is a slalom polynomial for the tree 
$\Gamma$. \qed 
\end{proof}

The slalom polynomial\index{slalom polynomial}\index{polynomial!slalom} $P_\Gamma$ is not unique up to holomorphic changes of coordinates in the range and in the target. Uniqueness can be forced  for trees with maximum valence $\leq k$ if one imposes moreover that the critical values are roots of unity of order depending on $k$. See Figure \ref{fig:root}. For $k=3$, order four works. In Figure \ref{fig:root}, the values $1, i, -i, -1$ are the values of the polynomial at the given points.

In the example of Figure \ref{fig:z4}, the polynomial $z^4+z$ is the slalom polynomial of the tree $D_4$. The Chebyshev polynomials\index{Chebyshev polynomial} are up to a real translation precisely the slalom polynomials with  critical values among $\pm 1$. See also \cite{AC2021}.

\begin{figure}[h]
\centering
\includegraphics[scale=0.25]{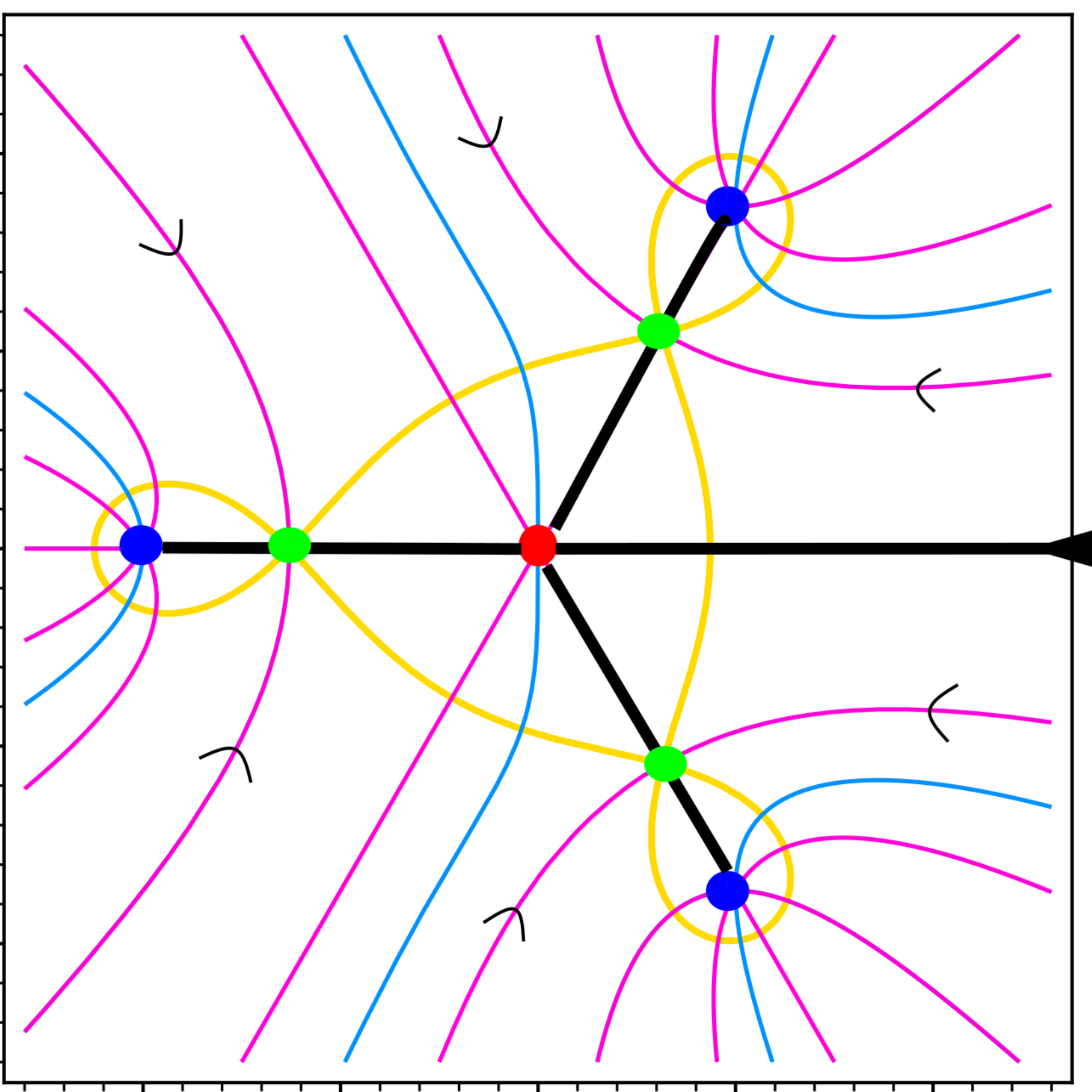}  
\caption{Flowbox decomposition for $P=z^4+z$. The green dots are the zeros of $P'$, the saddle points of the function $|P|$ and bifurcation points of the vector field $-{\rm grad}(|P|)$. The zeros of $P$, which are the red and blue dots, span a planar tree of Dynkin diagram $D_4$. The flowline from $+\infty$ (the horizontal segment on the middle right part of the figure) ends at the triple point of $D_4$. The separating flowlines are marked with arrows. The flowlines $P^{-1}(i\mathbb{R})$ are colored blue.}
\label{fig:z4}
\end{figure}

As in \cite{A2017} one obtains a cell-decomposition
of the space of monic polynomials 
${\rm Pol}_d$ of degree $d$ and of the complement ${\rm Pol}_d\setminus \Delta$ of the discriminant  $\Delta$ (the set of polynomials having at least one multiple root).

\begin{theorem}
The equivalence relation $\Gamma_P \sim \Gamma_Q$ induces a stratification with ${\rm Cat}(d-1)$ top-dimensional strata on the spaces ${\rm Pol}_d$ and ${\rm Pol}_d\setminus \Delta$. Moreover, the top-dimensional strata are cells and have a  representative by a slalom polynomial. \qed

\end{theorem}

For this theorem, see \cite{AC2021}.

\subsection{A stratification of the space of monic polynomials}\label{s:stratification}

The space of monic univariate complex polynomials ${\rm Pol}_d$ of degree $d$ is a complex affine space of dimension $d$. 
The discriminant $\Delta_d$ is an important sub-variety. A cell decomposition of ${\rm Pol}_d$ that induces a cell decomposition of 
$\Delta_d$ is of interest, for instance in the study of the braid group. The following defines an equivalence relation on ${\rm Pol}_d$, whose associated equivalence classes define a cell decomposition such that the discriminant is a union of cells, see \cite{A2017}.

Define the \emph{picture} ${\rm Pic}(P)$ of a monic polynomial to be the graph ${\rm Pic}(P)=P^{-1}(\mathbb{R}\cup i\mathbb{R})$. Call two polynomials $P,Q$ equivalent if the graphs ${\rm Pic}(P)$ and ${\rm Pic}(Q)$ are isotopic by an isotopy that preserves the asymptotes.

The picture ${\rm Pic}(P)$ of a degree $d$ monic polynomial is a planar forest without
terminal vertices, but with $4d$ edges going to $\infty$ asymptotic to the rays corresponding to $4d$ roots of unity. The precise combinatorial characterization is given in \cite{A2017}. The problem of showing that each 
possible equivalence class is realized was solved using Riemann's Uniformisation Theorem. The number of top-dimensional cells is the Fuss--Catalan number\index{Fuss--Catalan number} 
$\frac{1}{3d+1}{4d \choose d }$.

\subsection{Rational maps, Speiser colored cell decompositions, classical knots and links.}\label{ss:Rational}

Let $f:\mathbb{P}^1(\mathbb{C})\to \mathbb{P}^1(\mathbb{C})$ be a holomorphic map of degree $d>0$. Such a map is also called a rational map, since in the above coordinate $z$ its expression $f(z)=\frac{P(z)}{Q(z)}$ is a ratio of two polynomials $P(z)\not=0,Q(z)\not=0$ of degrees $d_1,d_2\geq 0$ with $d=\max\{ d_1,d_2\}$. 

Let $\Delta(f)$ be the finite set of critical values of $f$. The restriction of $f$ to the complement of the set  $f^{-1}(\Delta(f))$ is a covering map of degree $d$.
The set $f^{-1}(\Delta(f))$ contains the set of critical points of $f$. For each $a\in \mathbb{P}^1(\mathbb{C})$,
define the \emph{defect} $\delta(a)\geq 0$ by $\delta(a)=d-\#f^{-1}(a)$. Observe that $\delta(a)=0$ except for $a\in \Delta(f)$. The Riemann--Hurwitz formula becomes here $$\sum_{a\in \Delta(f)}\delta(a)=2d-2,$$ as follows from computing the Euler characteristic of the domain of $f$, $$2=\chi(\mathbb{P}^1(\mathbb{C}))=d\chi(\mathbb{P}^1(\mathbb{C}))-\sum_{a\in \Delta(f)}\delta(a).$$

Let $\beta(f)=\#\Delta(f)$ be the number of critical values of $f$. Clearly, $\beta(f)\leq 2d-2$. Recall (\S \ref{s:Speiser}) that a Speiser  curve for $f$ is an oriented simple closed smooth curve $\gamma$ in $\mathbb{P}^1(\mathbb{C})$ that passes through all the critical values $\Delta(f)$. Observe that two such curves $\gamma,\gamma'$  are smoothly isotopic, keeping $\Delta(f)$ fixed, if and only if the induced cyclic order on the elements of $\Delta(f)$ agree. So up to isotopy one has $(\#\Delta(f)-1)!$ isotopy classes of curves $\gamma$ through the points of $\Delta(f)$.
A Shabat polynomial or a Belyi map\index{Shabat polynomial}\index{polynomial!Shabat}
with three real critical values admits up to isotopy two Speiser curves, one being  $\mathbb{P}^1(\mathbb{R})=\mathbb{R}\cup\{\infty\}$.

From a Speiser curve $\gamma$ one constructs the Speiser graph as the pair $\Gamma_{f,\gamma}=(f^{-1}(\gamma), f^{-1}(\Delta(f)))$
in $\mathbb{P}^1(\mathbb{C})$. The vertices are colored by the elements of the branching set $\Delta(f)$ and the edges are oriented by lifting the orientation of $\gamma$. The Speiser curve $\gamma$ being smooth, the curve $f^{-1}(\gamma)\subset \mathbb{P}^1(\mathbb{C})$ is the image of the immersion $\iota:\dot{\cup}_k S^1 \to \mathbb{P}^1(\mathbb{C}), \, k\leq d,$ of finitely many copies of $S^1$  ($\dot{\cup}$ denotes disjoint union). The immersion  
$\iota$ has no tangencies and equi-angular multiple points at the critical points of $f$ (see Figure \ref{fig:Speiser-Belyi}). Each complementary region of $P_{f,\gamma}$ is a polygon with $\beta(f)$ vertices. We color such a polygon $R$ by $\pm$: by $+$ only if the induced orientation on $\partial R$ agrees with the orientation of $\Gamma_{f,\gamma}$, in which case the colors of the vertices of $R$ appear in the cyclic order induced by $\gamma$. The $\pm$ coloring of the polygons is a checkerboard coloring.

The Speiser graph $\Gamma_{f,\gamma}$ is in fact an immersion of circles without tangencies in $\mathbb{P}^1(\mathbb{C})=S^2$. The immersion $\iota$ is not a generic immersion, since self-intersection of $3$ or more local branches occur. But, at all self-intersections the local branches intersect pairwise transversely. 

We recall that a \emph{divide}\index{divide} in the sense of \cite{A'C, A1999} is the image of a relative generic immersion of a finite union of copies of the unit interval $(I, \partial I)$ in the unit disc $(D,\partial D)$. 
From a divide $P$, one defines a  link $L(P)$ in $S^3$ by
$$
L(P)=\{(x,u)\in T(P) \subset S^3 \hbox{ such that } \Vert (x,u)\Vert = 1\}
$$
where the 3-sphere $S^3$ is seen here as the unit sphere in the tangent bundle $T(\mathbb{R}^2)$ of $\mathbb{R}^2$.

By a result of  \cite{A'C}, if the image of the immersion defining a divide is connected, then the associated link is fibered. 
The theory of links associated with divides is closely related to singularity theory \cite{A1974, G-Z}. In the paper  \cite{A'C},  the monodromy of the fibered link associated with a connected divide is described in terms of the combinatorics of the divide.

Many constructions for divides still apply to Speiser graphs.

Let $\dot{\Gamma}_{f,\gamma}$ be the set of length $1$ tangent vectors to the Speiser graph$\Gamma_{f,\gamma}$.
Let $\dot{\Gamma}^+_{f,\gamma}$ be the set of oriented length $1$ tangent vectors to $\Gamma_{f,\gamma}$. 

Recall that the 3-sphere $S^3$ and the Lie group ${\rm SU}(2)$ are diffeomorphic and cover by a $2\to 1$ map $c$ the space of length $1$ tangent vectors to $\mathbb{P}^1(\mathbb{C})$. Hence $\dot{\Gamma}_{f,\gamma}$ and $\dot{\Gamma}^+_{f,\gamma}$ can be considered as subsets in $\mathbb{P}^3(\mathbb{R})$ and lift to
 the classical knots and links $c^{-1}(\dot{\Gamma}_{f,\gamma}), c^{-1}(\dot{\Gamma}^+_{f,\gamma})$ in $S^3$, which we call Speiser links.\index{Speiser link} See Figure \ref{fig:Speiser-Belyi} for an example of a Speiser graph and \ref{fig:Speiser-link} for its Speiser link.
 
 \begin{figure}[h] 
\centering
\includegraphics[scale=0.25]{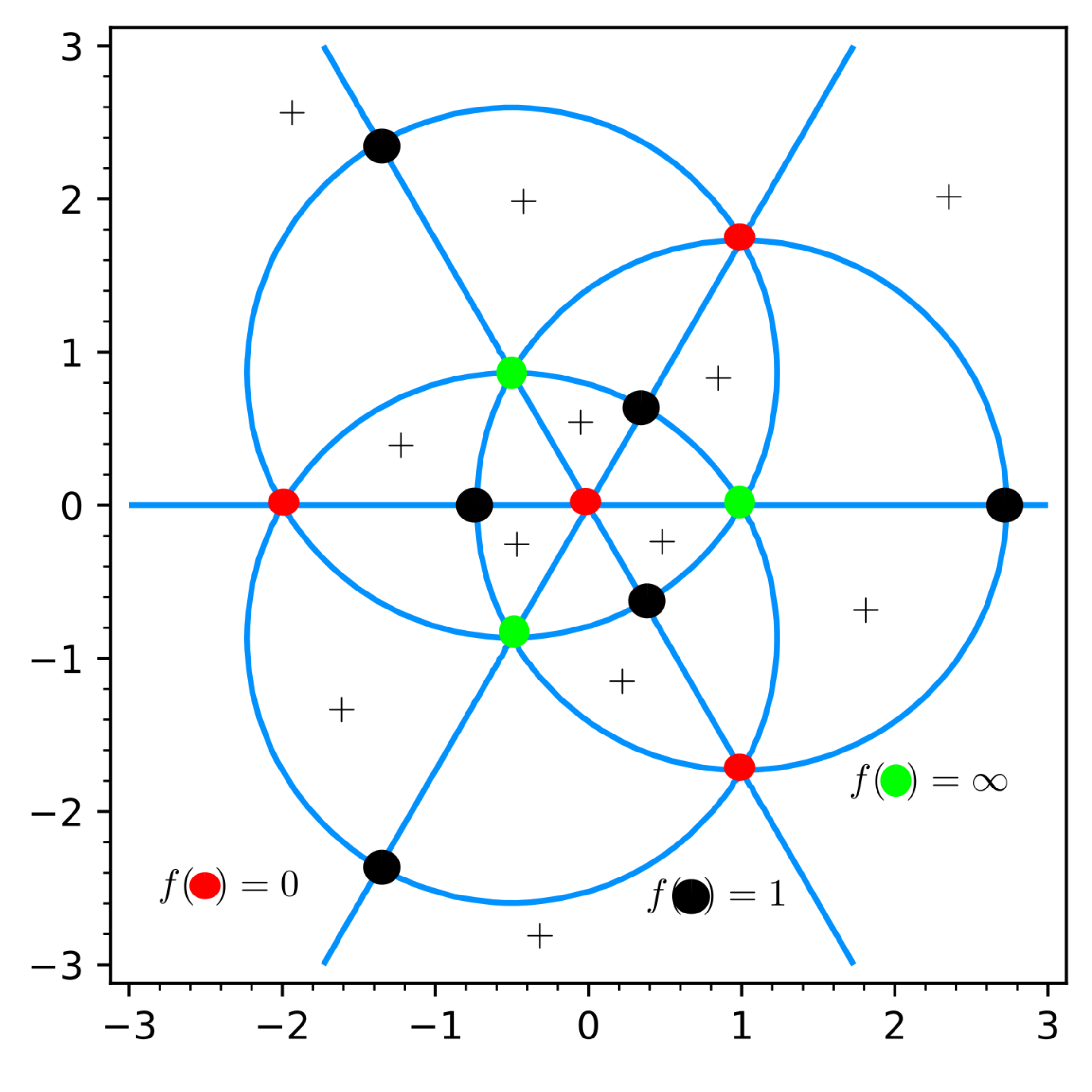}
\caption{The Speiser graph $\Gamma_{f,\gamma}$ for the Belyi map $f=\frac{z^3(z^3+8)^3}{64(z^3-1)^3}$ of degree $12$ and $\gamma=\mathbb{R}\cup\{\infty\}$. The covering has three critical values, $0, 1, \infty$, and $\gamma$ is the Speiser curve passing through these points. The twelve regions marked with $+$ map to $\mathbb{C}_+$.}
\label{fig:Speiser-Belyi}
\end{figure}

\begin{figure}[h]
\centering
\includegraphics[scale=1.1]{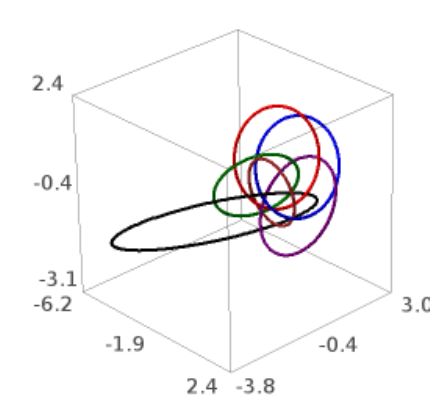}
\caption{The Speiser link $c^{-1}(\dot{\Gamma}_{f,\gamma})$ for the Belyi map\index{Belyi map} $f=\frac{z^3(z^3+8)^3}{64(z^3-1)^3}$ of degree $12$ consisting of $6$ components. Each component is a great circle on the round $S^3$.}
\label{fig:Speiser-link}
\end{figure}

\begin{theorem} The subsets $\dot{\Gamma}_{f,\gamma}$ and $\dot{\Gamma}^+_{f,\gamma}$
are links in $\mathbb{P}^3(\mathbb{R})$.  The lifts to $S^3$ are classical links. Moreover, the link $c^{-1}(\dot{\Gamma}_{f,\gamma})\subset S^3$ is fibered.
\end{theorem}

\begin{figure}[h]
\centering
\includegraphics[scale=0.60]{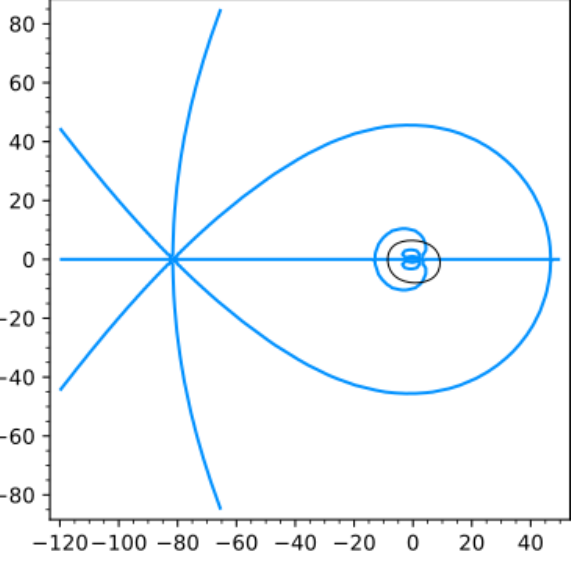}\includegraphics[scale=0.16]{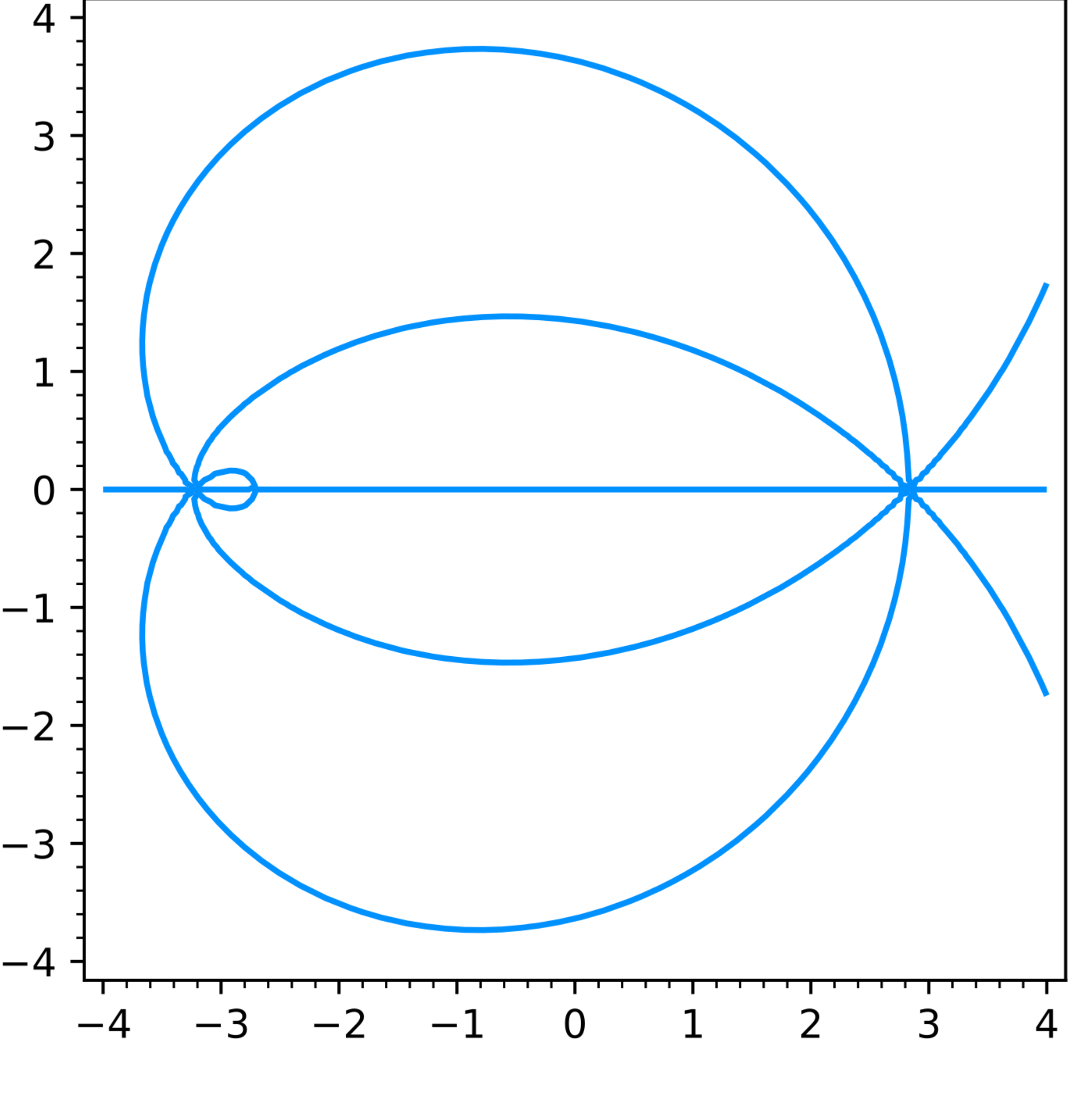}
\caption{The divide $\Gamma_{f,\gamma}$ for the Belyi map $f=\frac{(z^2+30\sqrt{8}z+264)^4}{216(z-\sqrt{8})^4(z+\sqrt{8})}$ of degree $8$ and $\gamma=\mathbb{R}\cup\{\infty\}$. The black encircled region in the picture on the left is  zoom enlarged and translated in the picture  on the right.}
\end{figure}

\begin{proof} Let $\theta:\gamma \to T\gamma\subset T_{l=1}\mathbb{P}^1(\mathbb{C})$ be a continuous oriented
vector field of length $1$ along $\gamma$. Each component of $\gamma\setminus \Delta(f)$ together with the restriction of the vector field $\theta$ lifts by $c$ to $d={\rm degree}(f)$ smooth arcs in $T_{l=1}\mathbb{P}^1(\mathbb{C})$.  The $d\beta(f)$ lifted arcs are pairwise disjoint. At a critical point $p$ of $f$,
$2k$ arcs join. The local germ of $\dot{\Gamma}^+_{f,\gamma}$ at $p$ is diffeomorphic to the germ at $0$ of the set
$\{a\in \mathbb{C}\mid {\rm Im}(a^k)=0\}$.
So, after a small $C^1$ perturbation of $\dot{\Gamma}_{f,\gamma}$ locally near the critical points of $f$, the graph  $\dot{\Gamma}_{f,\gamma}$ 
can be deformed to a divide with checkerboard coloring $P_{f,\gamma}$ on $S^2=\mathbb{P}^1(\mathbb{C})$. The connected components of the complement $P_{f,\gamma}$ are contractible and moreover $P_{f,\gamma}$ has no self-intersection other than double points. The construction in \cite{A'C}
associates the knot or link $c^{-1}(\dot{\Gamma}_{f,\gamma})$ in $T_{l=1}S^2=\mathbb{P}^3(\mathbb{R})$. 
The fibration  property as in \cite{A'C} was extended by Masaharu Ishikawa and Hironobu Naoe \cite{I,I-N} and shows that $\dot{P}_{f,\gamma}$ is a fibered knot or link in $\mathbb{P}^3(\mathbb{R})=T_{l=1}S^2$. The same holds for its lift $c^{-1}(\dot{P}_{f,\gamma})\subset S^3$ and also the link $c^{-1}(\dot{\Gamma}_{f,\gamma})\subset S^3$.\qed
\end{proof}

The construction of the fiber surface and geometric monodromy  for the link $c^{-1}(\dot{\Gamma}_{f,\gamma})\subset S^3$ as in  \cite{A'C} is as follows. 
Let $P_{f,\gamma}$   be a small $C^1$-deformation of $\dot{\Gamma}_{f,\gamma}$ with only double points.
The links $c^{-1}(\dot{P}_{f,\gamma})$ and $c^{-1}(\dot{\Gamma}_{f,\gamma})$ are isotopic, hence equivalent as links. The main step in the construction is the choice of a Morse function $h:S^2 \to \mathbb{R}$ having one maximum $+1$ in each $+$-region, one minimum $-1$ in each $-$-region, and $P_{f,\gamma}=h^{-1}(0)$ as critical level through all saddle points. Define
$F_h=\{V_p\in T_p S^2 \mid h(p)\in ]0,1[,\,
(Dh)_p(V_p)=0,\, ||V_p||_{\rm Eucl}=1 \}\subset T_{l=1}S^2=\mathbb{P}^3(\mathbb{R})$. The closure $\bar{F_h}$ of $F_h$ in $\mathbb{P}^3(\mathbb{R})$ is a surface with boundary spanning the link $\dot{P}_{f,\gamma}$. The lift $c^{-1}(\bar{F_h}) $ is a spanning surface  and its interior $F_{f,\gamma}$ a fiber surface for 
the link $c^{-1}(\dot{P}_{f,\gamma})$.

On the surface $F$ appears a system of simple closed curves: for each maximum $p$ of $h$ the circle of length $1$ tangent vectors in the kernel of $(Dh)_p$, each gradient line of $h$ that connects a maximum to a minimum through a saddle point of $h$ lifts to a simple closed curve on $F$, and to each minimum of $h$ corresponds a simple closed curve built with pieces from preceding ones, see \cite{A'C}. The monodromy is a composition of the positive Dehn twists along these curves.

The examples of Belyi maps\index{Belyi map} are taken from the papers of K. Filom \cite{Filom2018} and K. Filom and A. Kamalinejad \cite{F-K2016}
in which dessins\index{dessin d'enfants} on modular curves are computed, studied and used. In the first example the divide is the union of  Moebius circles, which is   not always the case. More complex components can appear as the second example shows.

The link $L_{f,\gamma}$ of a rational Belyi map\index{Belyi map} $f:\mathbb{P}^1(\mathbb{C})\to\mathbb{P}^1(\mathbb{C})$
where the Speiser curve $\gamma=\mathbb{R}\cup\{\infty\}$ can be given as follows
explicitly as the closure in $S^3\subset \mathbb{C}^2$ of the set 
$$\{\frac{\lambda}{\sqrt{1+b\bar{b}}}(1,b)\mid b\in \mathbb{C}, f(b)\in \mathbb{R}\setminus \{0,1,\infty\}, \lambda^2=f'(b)/|f'(b)|\}$$
which is the union of $12d$ smooth arcs, $d={\rm degree}(f)$.

The tangential lift of an oriented Moebius circle $m$ on $\mathbb{P}^1(\mathbb{C})$ to $S^3={\rm SU}(2)$ is an oriented great circle on $S^3$. Indeed, by the isometric ${\rm SU}(2)$ action, move $m$ to be the parametrized curve
$s\in [0,2\pi] \mapsto m(s)=r\,\cos(s)+ri\,\sin(s)\in \mathbb{C}\cup \{\infty\}$. The curve $M(s)$ on $S^3$ given by 
$$s\in [0,4\pi] \mapsto \frac{1}{r^2+1} (-\sin(s/2)+i\,\cos(s/2),-r\,\sin(s/2)+ri\,\cos(s/2))\in \mathbb{C}^2$$ is its lift. From $M''(s)=\frac{-1/4}{r^2+1}M(s)$ follows the claim.

Let $D_1,D_2$ be closed disks in $\mathbb{P}^1(\mathbb{C})$ having 
as oriented boundary the Moebius circles $m_1,m_2$.
If $D_1,D_2$ are complementary, the circles differ only by orientation as do their lifts. If the boundaries of $D_1,D_2$ touch tangentially in one point, the lifts $M_1,M_2$  intersect in two points. If the intersection
$D_1\cap D_2$ is open, then the lifts $M_1,M_2$ are disjoint great circles.

Orient $S^3={\rm SU}(2)$ as the boundary of the unit ball in $\mathbb{C}^2$. Assume that $M_1,M_2$ are disjoint, so that
the boundaries of $D_1,D_2$ are disjoint or meet in two points. The linking number of the lifts $M_1,M_2$ 
of the oriented boundaries $m_1,m_2$ is given by 
${\rm Lk}_{S^3}(M_1,M_2)=1-\# m_1\cap m_2$.

In our example above with $6$ great circles as lifts, all $15$ pairwise linking numbers are equal to $-1$.

\bigskip 

\noindent {\bf Acknowledgement} Part of this survey is a revision of notes for lectures that were given by the first author at a CIMPA school in Varanasi (India), in December 2019. The authors would like to thank Bankteshwar Tiwari, the local organiser of this school, for his care and his excellent management. The second author is supported by the lnterdisciplinary Thematic lnstitute CREAA, as part of the ITl 2021-2028 program of the University of Strasbourg, CNRS and Inserm (ldEx Unistra ANR-10-IDEX-0002), and the French lnvestments for the Future Program.

\printindex

\noindent  Norbert A'Campo,  Universit\"at Basel, Departement Mathematik und Informatik, Fachbereich Mathematik, Spiegelgasse 1, CH-4051 Basel, Schweiz;

 \noindent Norbert.ACampo@unibas.ch

\medskip

\noindent Athanase Papadopoulos, Institut de Recherche Math\'ematique Avanc\'ee and Centre de Recherche et d'Expérimentation sur l'Acte Musical (Universit\'e de Strasbourg et CNRS),
7 rue Ren\'e Descartes
67084 Strasbourg Cedex France;

\noindent athanase.papadopoulos@math.unistra.fr

\end{document}